\PassOptionsToPackage{sort&compress}{natbib}
\documentclass[onecolumn]{elsarticle}

\usepackage{amsmath}
\usepackage{mathtools}
\usepackage{amssymb}
\usepackage{here}
\usepackage{mathrsfs}
\usepackage{multirow}
\usepackage{makecell}
\usepackage{enumerate}
\usepackage{dsfont}
\usepackage{fullpage}
\usepackage{color}

\numberwithin{equation}{section}

\everymath{\displaystyle}

\newtheorem{theorem}{Theorem}[section]

\newtheorem{corollary}[theorem]{Corollary}
\newtheorem{proposition}[theorem]{Proposition}

\newtheorem{example}[theorem]{Example}

\newtheorem{remark}[theorem]{Remark}

\newenvironment{proof}{{\it Proof :~}}{\hfill$\diamondsuit$\\}

\newcommand{\mendth}{\hfill \ensuremath{\vartriangle}}

\providecommand{\blue}[1]{\textcolor{black}{#1}}
\providecommand{\bluee}[1]{\textcolor{black}{#1}}

\DeclareMathOperator*{\col}{col}

\DeclareMathOperator*{\diag}{diag}

\DeclareMathOperator{\eps}{\varepsilon}

\def\cite{\citep}

\def\T{\top}
\begin{document}

\begin{frontmatter}

\title{Dwell-time stability and stabilization conditions for linear positive impulsive and switched systems}

\author{Corentin Briat}\ead{corentin.briat@bsse.ethz.ch,corentin@briat.info}\ead[url]{http://www.briat.info}
\address{Department of Biosystems Science and Engineering, ETH Z\"{u}rich, Switzerland.}


\begin{abstract}
 Several results regarding the stability and the stabilization of linear impulsive positive systems under arbitrary, constant, minimum, maximum and range dwell-time are obtained. The proposed stability conditions characterize the pointwise decrease of a linear copositive Lyapunov function and are formulated in terms of finite-dimensional or semi-infinite linear programs. To be applicable to uncertain systems and to control design, a lifting approach introducing a clock-variable is then considered in order to make the conditions affine in the matrices of the system. The resulting stability and stabilization conditions are stated as infinite-dimensional linear programs for which three asymptotically exact computational methods are proposed and compared with each other on numerical examples. Similar results are then obtained for linear positive switched systems by exploiting the possibility of reformulating a switched system as an impulsive system. Some existing stability conditions are retrieved and extended to stabilization using the proposed lifting approach. Several examples are finally given for illustration.
\end{abstract}

\end{frontmatter}

%\tableofcontents

\section{Introduction}

Linear positive systems \cite{Farina:00} have been recently the subject of an increasing attention because of their natural ability to represent many real-world processes such as, among others, communication networks \cite{Shorten:06,Briat:13f}, biological networks \cite{Murray:02,Briat:12c,Briat:13h,Briat:15e}, epidemiological networks \cite{Murray:02,Briat:09h}, disease dynamics \cite{Jonsson:14}, etc. Besides their applicative potential, linear positive systems have been shown to exhibit a number of interesting theoretical properties of independent interest. For instance, it is now well-known that linear copositive Lyapunov functions can be used in order to formulate exact stability conditions taking the form of linear programs \cite{Haddad:05}. The design of structured and bounded state-feedback controllers \cite{Aitrami:07,Briat:11h} and certain classes of static output feedback controllers \cite{AitRami:11} are known to be convex and hence easily tractable. The $L_p$-gains for $p=1,2,\infty$ can be exactly computed using convex programming and these gains are identical to the $p$-norm of the static matrix-gain of the system \cite{Tanaka:13b,Briat:11g}. The famous Kalman-Yakubovich-Popov Lemma has been shown to admit a linear formulation in this setting \cite{Rantzer:15}. Robust analysis results also nicely extends and simplifies in this context, and often becomes necessary and sufficient criteria for stability \cite{Colombino:15,Colombino:15b,Briat:15cdc,Briat:11g,Briat:11h}. Their generalization to delay-systems with discrete-delays also led to the surprise that the system is stable if and only if the system with zero delay is stable \cite{Haddad:04,AitRami:09,Briat:11h,Shen:15,Briat:16b}. Extensions to deterministically \cite{Fornasini:10,Zappavigna:10a,Zappavigna:10b,Blanchini:15} or stochastically \cite{Bolzern:14,Bolzern:15} switched systems have also been considered. \blue{Positive systems have also been recently used as (conservative) comparison systems for establishing the stability of various classes of systems such as systems with delays \cite{Agarwal:12,Mazenc:16,Ngoc:16,Ngoc:16b}. Finally, the design of interval observers heavily relies on the use of positive systems theory \cite{Gouze:00,Mazenc:11,Briat:15g}. It was notably shown in \cite{Briat:15g} that the observer-gain that minimizes the $L_\infty$-gain of map between the disturbance input and the observation error is independent of the output matrices of the error system.}

We consider here the case of linear positive impulsive systems, a class of systems that seems to have been quite overlooked until now as  only very few results can be found; see e.g. \cite{Dvirnyi:04,Wang:14,Zhang:14b}. Such systems can be used to represent certain classes biochemical, population or epidemiological models having deterministic jumps in their dynamics. They can also be used to represent processes that can be represented as linear positive switched systems; see e.g. \cite{Blanchini:15} for some examples including epidemiology \cite{Perelson:99,Hernandez:13,AitRami:14}, traffic congestion models \cite{Blanchini:12b}, etc. Impulsive systems are also known to be able to exactly represent sampled-data systems as emphasized in \cite{Khargonekar:91,Sivashankar:94,Naghshtabrizi:08,Briat:13d}. \bluee{Such systems are also interesting from a theoretical standpoint as they can be useful for the analysis and design of interval observers for linear impulsive systems (and hence sampled-data and switched systems) or for analyzing the stability of nonlinear impulsive, switched and sampled-data systems; see e.g. \cite{Dvirnyi:04}.}

The goal of this paper is hence the derivation of novel stability and stabilization conditions for linear positive impulsive systems using the concepts of arbitrary, constant, minimum, maximum and range dwell-times. The concept of \emph{minimum dwell-time} has been introduced by Morse in \cite{Morse:96} in order to formulate stability conditions for general (i.e. not necessarily positive) switched systems. The concept of \emph{average dwell-time} has been proposed in \cite{Hespanha:99} in order to obtained less conservative conditions that by using minimum dwell-time conditions. Since then, a large body of the literature has been focusing on these concepts as a way to efficiently  characterize the stability of switched systems or, more generally, the stability of hybrid systems; see e.g. \cite{Goebel:09}. The notion of minimum dwell-time has been revisited in \cite{Geromel:06b} where novel sufficient LMI conditions derived from mode-dependent quadratic Lyapunov functions were proposed. Based on a theoretical result proved in \cite{Wirth:05}, these conditions were later extended and made necessary and sufficient in \cite{Chesi:12} through the consideration of mode-dependent homogeneous Lyapunov functions. Analogous results using polyhedral Lyapunov functions have been also obtained in \cite{Blanchini:10}. Unfortunately, these conditions were inapplicable to uncertain systems and to control design because of their complex nonlinear dependency in the matrices of the system. This problem motivated the introduction of the so-called \emph{looped-functionals}, a particular class of indefinite (i.e. not necessarily positive definite) functionals having the advantage of reformulating the complex conditions of \cite{Geromel:06b} into conditions being linear in the matrices of the system; see e.g. \cite{Seuret:12,Briat:12h,Briat:13b}, thereby extending the scope of the conditions to uncertain and nonlinear systems. Yet, these conditions were difficult to apply in the context of control design because of the presence of multiple products between decision matrices and the matrices of the system; see e.g. \cite{Briat:12h,Briat:13b,Briat:15c,Briat:15f}.  \emph{Clock-dependent conditions} have been shown to provide an essential framework for solving this latter problem as they produce stability conditions that are linear/convex in the matrices of the system and can be used for design purposes. Their computational complexity has also been shown to be much lower than that of looped-functionals \cite{Briat:15f}. Since then, clock-dependent conditions have been used for the analysis and control of switched, impulsive, sampled-data and LPV systems; see e.g. \cite{Allerhand:11,Allerhand:13,Allerhand:13b,Briat:13d,Briat:14a,Shaked:14,Briat:14f,Briat:15d,Briat:15i,Xiang:16}. Such results have also been applied to more practical problems such as fault tolerant control \cite{Allerhand:15,Briat:TAC16} or estimation \cite{Allerhand:13b,Nguyen:15}.

The first part of the paper is similar to the ones in \cite{Briat:11l,Briat:13d} where stability conditions are formulated in terms of the decrease of a Lyapunov function of a given type. Unlike in the previous references where quadratic Lyapunov functions are involved, we exploit here the positivity of the system and consider a linear copositive Lyapunov function \cite{Haddad:05}. The resulting conditions are stated in terms of finite-dimensional linear programs or  semi-infinite dimensional linear programs, which are then relaxed into clock-dependent conditions using a lifting approach similar to that of \cite{Briat:13d,Briat:14f,Briat:15d,Briat:15f}. Since linear copositive Lyapunov functions are used here, the clock-dependent conditions consist of  infinite-dimensional linear programs. This has to be contrasted with the fact that when quadratic Lyapunov functions are used, clock-dependent conditions take the form of infinite-dimensional semidefinite programs, which may be harder to solve that their linear counterpart. Three possible ways for efficiently checking these conditions are then proposed. The first one relies on a discretization approach which is largely inspired from \cite{Allerhand:11} and where the infinite-dimensional decision variable is assumed to be continuous and piecewise linear. This method has also been considered, in turn, in \cite{Allerhand:13,Allerhand:13b,Shaked:14,Xiang:15a,Briat:15f,Briat:15i}. By doing so, the infinite-dimensional program becomes finite-dimensional and can be solved using conventional algorithms such as interior point methods; see e.g. \cite{Boyd:04}. The second method is based on Handelman's theorem \cite{Handelman:88} which characterizes the positivity of a given polynomial on a compact polytope by formulating it as a nonnegative linear combination of products of the (affine) basis functions that describe the polytope. This result has been applied in various contexts \cite{Scherer:06,Briat:11h,Briat:11g,Kamyar:15} and, notably, for characterizing the robust stability of uncertain linear positive systems in \cite{Briat:11g,Briat:11h}. An important property of this approach is that the obtained characterization for the positivity of the polynomial can be exactly formulated as a finite-dimensional linear program, which can again be solved using well-known approaches. Finally, the last one is based on Putinar's Positivstellensatz \cite{Putinar:93} which characterizes the positivity of a given polynomial on a compact semialgebraic set by formulating it as a weighted linear combination of the basis functions that describe the set and where the weights are sum of squares polynomials. The resulting problem takes, in this case, the form of a finite-dimensional semidefinite program \cite{Parrilo:00} that can be solved using standard semidefinite programming solvers such as SeduMi \cite{Sturm:01a} or SDPT3 \cite{Tutuncu:03} used in conjunction with the package SOSTOOLS \cite{sostools3}. It is notably emphasized that these relaxations are asymptotically exact meaning that when the discretization order, the number of products of basis functions or the degree of the sum of squares weights are sufficiently large, then the relaxed problem is feasible if the original one is. Several examples are considered in order to demonstrate the practicality of the relaxed conditions and to compare them in terms of number of variables and solving time. The results are then extended to control design by considering the clock-dependent conditions and the dual impulsive system \cite{Lawrence:10,Blanchini:15}. %Several illustrative examples are given as an illustration.
By finally exploiting the possibility of formulating a switched system as an impulsive system, we derive a number of stability conditions for linear positive switched systems. Notably, we recover stability conditions similar to those in \cite{Zappavigna:10b,Blanchini:15} which are the positive counterpart of those obtained in \cite{Geromel:06b} in the context of general linear switched systems while some other ones seem to be new. Novel stabilization conditions  under arbitrary, minimum and mode-dependent range dwell-time conditions are then obtained and illustrated through few examples.\\

\noindent\textbf{Outline.} The structure of the paper is as follows: in Section \ref{sec:preliminary} preliminary definitions and results are given. Section \ref{sec:stab} is devoted to the dwell-time stability analysis of linear positive impulsive systems while Section \ref{sec:stabz} addresses their stabilization. These results are then applied to switched systems in Section \ref{sec:switched}.\\

\noindent\textbf{Notations.} The cone of positive (nonnegative) vectors of dimension $n$ are denoted by $\mathbb{R}_{>0}^n$ ($\mathbb{R}_{\ge0}^n$). The set of diagonal matrices of dimension $n$ is denoted by $\mathbb{D}^n$ and that of diagonal matrices with positive diagonal elements by $\mathbb{D}^n_{\succ0}$. For a set of matrices $\{A_1,\ldots,A_N\}$ of compatible dimensions, we define $\textstyle\diag_i(A_i)$ to be the block diagonal matrix with the $A_i$'s as diagonal blocks and by $\textstyle\col_i(A_i)$ the matrix consisting of the $A_i$'s stacked in column. The $n$-dimensional vector of ones is denoted by $\mathds{1}_n$.

\section{Preliminaries}\label{sec:preliminary}

Let us consider here the following class of linear impulsive system:
\begin{equation}\label{eq:mainsyst}
\begin{array}{rcl}
  \dot{x}(t)&=&Ax(t),\ t\ne t_k\\
  x(t^+)&=&Jx(t),\ t=t_k\\
  x(t_0)&=&x_0
\end{array}
\end{equation}
where $x,x_0\in\mathbb{R}_{\ge0}^n$ are the state of the system and its initial condition, respectively. The notation $x(t^+)$ is a shorthand for $\textstyle\lim_{s\downarrow t}x(s)$, i.e. the trajectories are assumed to be left-continuous. The sequence of impulse instants $\{t_k\}_{k\in\mathbb{N}}$ is assumed to be strictly increasing and to grow unboundedly. As a consequence, it does not admit any accumulation point and may not lead to any Zeno behavior for dynamics of the system. Note that this assumption is not restrictive for the current paper as we only focus here on dwell-time results. We define the dwell-times as the values $T_k:=t_{k+1}-t_k$, $k\in\mathbb{N}$.

The following result states necessary and sufficient conditions for the impulsive system \eqref{eq:mainsyst} to be positive:
\begin{proposition}
The following statements are equivalent:
\begin{enumerate}[(a)]
  \item The system \eqref{eq:mainsyst} is positive; i.e. for any $x_0\ge0$, we have that $x(t)\ge0$ for all $t\ge0$.
  \item The matrix $A$ is Metzler (all its off-diagonal entries are nonnegative) and $J$ is nonnegative.
\end{enumerate}
\end{proposition}

The following result establishes conditions for which a Metzler matrix is Hurwitz stable:
\begin{proposition}[\cite{Berman:94,Farina:00}]\label{prop:Hurwitz}
Let $A\in\mathbb{R}^{n\times n}$ be a Metzler matrix. Then, the following statements are equivalent:
\begin{enumerate}[(a)]
  \item $A$ is Hurwitz stable;
  \item $A$ is nonsingular and $A^{-1}\le0$;
  \item There exists a $\lambda\in\mathbb{R}^n_{>0}$ such that $\lambda^\T A<0$;
  \item There exists a $\mu\in\mathbb{R}^n_{>0}$ such that $A\mu<0$;
\end{enumerate}
\end{proposition}

Similarly, the next  result establishes conditions for which a nonnegative matrix is Schur stable:
\begin{proposition}\label{prop:Schur}
Let $B\in\mathbb{R}^{n\times n}$ be a nonnegative matrix. Then, the following statements are equivalent:
\begin{enumerate}[(a)]
  \item $B$ is Schur stable;
  \item $B-I_n$ is Hurwiz stable;
  \item There exists a $\lambda\in\mathbb{R}^n_{>0}$ such that $\lambda^\T(B-I_n)<0$;
  \item There exists a $\mu\in\mathbb{R}^n_{>0}$ such that $(B-I_n)\mu<0$;
\end{enumerate}
\end{proposition}

The following generic stability result will be important in the sequel:
\begin{proposition}\label{th:equivalence}
    Let the sequence $\{t_k\}_{k\in\mathbb{N}}$ for the system \eqref{eq:mainsyst} be given and assume that it satisfies the strict increase and unboundedness conditions. Then, the following statements are equivalent:
    \begin{enumerate}[(a)]
      \item\label{item:equi:a} The impulsive system \eqref{eq:mainsyst} is asymptotically stable.
      \item\label{item:equi:a2} The dual impulsive system
      \begin{equation}\label{eq:dual}
        \begin{array}{rcl}
          \dot{\tilde{y}}(t)&=&A^\T \tilde{y}(t),\ t\ne t_k\\
          \tilde{y}(t^+)&=&J^\T\tilde{y}(t),\ t=t_k
        \end{array}
    \end{equation}
      is asymptotically stable.
      \item\label{item:equi:b} The discrete-time system
      \begin{equation}\label{eq:embedded}
        x_{k+1}=Je^{AT_k}x_k
     \end{equation}
     is asymptotically stable.
     \item\label{item:equi:c} The discrete-time system
       \begin{equation}\label{eq:embedded_swapped}
       y_{k+1}=e^{AT_k}Jy_k
        \end{equation}
        is asymptotically stable.
           \item\label{item:equi:d} The discrete-time system
\begin{equation}\label{eq:embedded_dual}
 \tilde x_{k+1}=J^\T e^{A^\T T_k}\tilde x_k
\end{equation}
    is asymptotically stable.
    \item\label{item:equi:e} The discrete-time system
\begin{equation}\label{eq:embedded_swapped_dual}
 \tilde y_{k+1}=e^{A^\T T_k}J^\T\tilde y_k
\end{equation}
    is asymptotically stable.
    \end{enumerate}
\end{proposition}

The meaning of the above result is that we can choose the most convenient system to work with in order to derive stability and stabilization conditions for the original impulsive system \eqref{eq:mainsyst}. This result is also important as it will be used to demonstrate that despite the systems can be all equivalent in terms of stability, the obtained stability conditions will not be necessarily equivalent. This point is discussed for the case of positive switched  systems in \cite{Blanchini:15} and the same discussion remains valid in the case of linear positive impulsive systems. Notably, the difference between the results obtained using the system \eqref{eq:embedded} and the ``swapped" version \eqref{eq:embedded_swapped} will be emphasized in Section \ref{sec:imp:examples} and in Section \ref{sec:SW:examples}. The interest for using the dual system \eqref{eq:embedded_dual} and its ``swapped" version \eqref{eq:embedded_swapped_dual} will be emphasized in the sections related to stabilization; i.e. Section \ref{sec:stabz} and Section \ref{sec:switched}. \blue{Note that the term ``dual system" is used here in the same way as in \cite{Blanchini:15} where the system is obtained by simply replacing the matrices of the system by their transpose.}

%%%%%%%%%%%%%%%%%%%
%%%   KEEP FOR ARXIV  %%%%%%%
%%%%%%%%%%%%%%%%%%%
%Note also that when the sequence $\{t_k\}_{k\in\mathbb{N}_0}$ is such that $t_{k+1}-t_k=\bar T$ for all $k\ge k_0$, $k_0\in\mathbb{N}$, then all the discrete-time systems in Theorem \ref{th:equivalence}  become time-invariant after time $k_0$. Hence, the stability conditions reduce to standard stability conditions for linear time-invariant discrete-time systems and will all be equivalent with each other. This result can be easily be understood through the fact that $\rho(Je^{A\bar T})=\rho(e^{A\bar T}J)=\rho(J^\T e^{A^\T\bar T})=\rho(e^{A^\T\bar T})J^\T$, where $\rho(\cdot)$ denotes the spectral radius.

\blue{Finally, it seems interesting to mention that the above result pertains on the establishment of the asymptotic stability of the system \eqref{eq:mainsyst}. The results can be easily extended to the case of exponential stability by considering the change of variables $z(t)=e^{\alpha t}x(t)$, $\alpha>0$, and the resulting comparison system
  \begin{equation}
\begin{array}{rcl}
      \dot{z}(t)&=&(A+\alpha I)z(t),\ t\ne t_k\\
      z(t_k^+)&=&Jz(t),\ t=t_k.
\end{array}
  \end{equation}
  By applying the asymptotic stability results to the above system, we can conclude on the $\alpha$-exponential stability of the system  \eqref{eq:mainsyst}. Note that, in this case and for some given sequence $\{t_k\}_{k\in\mathbb{N}}$, the geometric convergence rate of the discrete-time system \eqref{eq:embedded} will be given by $\textstyle\sup_{k\in\mathbb{N}_{0}}e^{-\alpha T_k}$.  }

\section{Stability of linear positive impulsive systems}\label{sec:stab}

We derive here several stability results for linear positive impulsive systems. The case of arbitrary dwell-time ($T_k\in\mathbb{R}_{>0}$) is considered first in Section \ref{sec:imp:arbitrary} and is followed by stability results under constant dwell-time ($T_k=\bar T$), minimum ($T_k\ge\bar T$), maximum ($T_k\le\bar T$) and range dwell-time ($T_k\in[T_{min},T_{max}]$) in Section \ref{sec:imp:constant}, Section \ref{sec:imp:minimum}, Section \ref{sec:imp:maximum} and Section \ref{sec:imp:range}, respectively. Computational results are given in Section \ref{sec:imp:computation} together with some discussions regarding their conservatism. Comparative examples are finally given in Section \ref{sec:imp:examples}.

\subsection{Stability under arbitrary dwell-time}\label{sec:imp:arbitrary}

Let us consider first the stability under arbitrary dwell-time ($T_k\in\mathbb{R}_{>0}$), with the additional condition that the sequence $\{t_k\}$ grows unboundedly. We then have the following result:
\begin{theorem}\label{th:arbDT}
  Let $A\in\mathbb{R}^{n\times n}$ and $J\in\mathbb{R}^{n\times n}$ be a Metzler and a nonnegative matrix, respectively. Then, the following statements are equivalent:
  \begin{enumerate}[(a)]
    \item There exists a $\lambda\in\mathbb{R}_{>0}^n$ such that
    \begin{equation}\label{eq:arbDT1}
      \lambda^\T A<0\ \textnormal{and}\   \lambda^\T (J-I_n)<0.
    \end{equation}
  \item We have that
  \begin{equation}
    \ker\begin{bmatrix}
      I & -A & -\left(J-I_n\right)
    \end{bmatrix}\cap\mathbb{R}^{3n}_{\ge0}=\{0\}.
  \end{equation}
  \end{enumerate}
Moreover, when one of the above statements holds, then the system \eqref{eq:mainsyst} is asymptotically stable under arbitrary dwell-time; i.e. for any sequence of impulse times verifying $T_k\in(0,\infty)$.
\end{theorem}
\begin{proof}
\blue{We prove first that (a) implies the asymptotic stability of the system \eqref{eq:mainsyst} under arbitrary dwell-time. To this aim, let us consider the linear copositive Lyapunov function $V(x)=\lambda^\T x$ where $\lambda\in\mathbb{R}^n_{>0}$.} The first condition in \eqref{eq:arbDT1} is equivalent to saying that $\dot{V}(x(t))\le-\mu V(x(t))$ for some $\mu\in\mathbb{R}_{>0}$ whereas the second one is equivalent to saying that $V(x(t_k^+))\le\eps V(x(t_k))$ for some $\eps\in(0,1)$ and all $k\in\mathbb{N}$. These conditions, all together, imply that $V(x(t))\to0$ as $t\to\infty$ regardless the impulse sequence $\{t_k\}_{k\in\mathbb{N}}$. \bluee{To prove the equivalence with (b), we first remark that the conditions of statement (a) coincide with the stability conditions obtained for the linear positive switched system $\dot{z}=M_{\sigma}z$ with $\sigma\in\{1,2\}$, $M_1=A$ and $M_2=J-I_n$ using a common copositive Lyapunov function. Using now \cite[Theorem 1]{Fornasini:10}, the equivalence with (b) directly follows.}
\end{proof}

\begin{remark}[Dual conditions]\label{rem:dual_arbDT}\label{rem:arbDT2}
  Dual conditions can be obtained by substituting the matrices of the dual system \eqref{eq:dual} inside the conditions of Theorem \ref{th:arbDT}. In such a case, the conditions of statement (a) become
     \begin{equation}\label{eq:arbDT2}
      A\lambda<0\ \textnormal{and}\   (J-I_n)\lambda<0.
    \end{equation}
    whereas that of statement (b) becomes
     \begin{equation}
    \ker\begin{bmatrix}
      I & -A^\T & -\left(J-I_n\right)^\T
    \end{bmatrix}\cap\mathbb{R}^{3n}_{\ge0}=\{0\}.
  \end{equation}
\bluee{It can be shown that the conditions \eqref{eq:arbDT2} could have also been obtained by considering the polyhedral Lyapunov function $\textstyle V(x)=\max_{i=1}^N\{\lambda_i^{-1}x_i\}$. Such functions have been successfully used for the analysis of LPV and linear switched systems in \cite{Blanchini:07,Blanchini:08} and linear positive switched systems in \cite{Blanchini:15}.}
\end{remark}

\begin{remark}[Persistent flowing]\label{rem:arbDT1_flow}
    In the case of persistent flowing (i.e. the flow never stops) the conditions in Theorem \ref{th:arbDT}, (a), can be relaxed to
  \begin{equation}\label{eq:arbDT1_flow}
      \lambda^\T A<0\ \textnormal{and}\   \lambda^\T (J-I_n)\le0
    \end{equation}
    and those in Remark \ref{rem:arbDT2}, (a), to
      \begin{equation}%\label{eq:arbDT2_flow}
      A\lambda<0\ \textnormal{and}\   (J-I_n)\lambda\le0.
    \end{equation}
    \blue{The interpretation of these conditions is that when the flow persists, we simply need to find a Lyapunov function that is decreasing along the flow of the system and non-increasing at the jumps. We can see that the conditions \eqref{eq:arbDT1_flow} exactly characterize this for the candidate Lyapunov function $V(x)=\lambda^T$, $\lambda\in\mathbb{R}^n_{>0}$; i.e. $\dot{V}(x)<0$ and $V(Jx)-V(x)\le0$ for all $x\in\mathbb{R}^n_{\ge0}$, $x\ne0$.}
\end{remark}

The following example illustrates the discussion below Proposition \ref{th:equivalence} about the non-equivalence between the stability conditions obtained from the impulsive system \eqref{eq:mainsyst} and its dual \eqref{eq:dual}:
\begin{example}
Let us consider the system \eqref{eq:mainsyst} with matrices
\begin{equation}
  A=\dfrac{1}{2}\begin{bmatrix}
    -3 & 1\\
    1/3 & -1
  \end{bmatrix}\ \textnormal{and}\   J=\dfrac{1}{2}\begin{bmatrix}
    1 & 1/2\\
    1 & 0
  \end{bmatrix}.
\end{equation}
There is no vector $\lambda\in\mathbb{R}^n_{>0}$ such that the conditions \eqref{eq:arbDT1} hold since $\lambda^\T A<0$ implies that $\lambda_1-\lambda_2<0$ while $\lambda^\T(J-I_n)<0$ implies that $-\lambda_1+\lambda_2<0$, yielding then a contradiction. On the other hand, we can readily see that $\lambda=\begin{bmatrix}
  1 & 1
\end{bmatrix}^\T$ solves the conditions \eqref{eq:arbDT2}, emphasizing then the gap between the conditions of Theorem \ref{th:arbDT} and Remark \ref{rem:arbDT2}.
\end{example}

\subsection{Stability under constant dwell-time}\label{sec:imp:constant}

We consider now the case of constant dwell-times -- the case where $T_k=\bar{T}$, for some $\bar T>0$ and for all $k\in\mathbb{N}$ -- or, in other words, the case where jumps occur periodically. \bluee{Even if quite restrictive, this case will be useful for deriving the results in Section \ref{sec:imp:minimum} and Section \ref{sec:imp:maximum}. The following result can be seen as a ``positive systems" counterpart of the constant dwell-time result in \citep{Briat:11l,Briat:12h,Briat:13b} and provides a stability condition in terms of the discrete-time system \eqref{eq:embedded}. It can also be connected to the result obtained in \cite{Zhang:14b}:}
\begin{theorem}[Stability under constant dwell-time]\label{th:cstDT2}
 Let $A\in\mathbb{R}^{n\times n}$ and $J\in\mathbb{R}^{n\times n}$ be a Metzler and a nonnegative matrix, respectively. Then, the following statements are equivalent:
 \begin{enumerate}[(a)]
  \item\label{item:th1:0} The linear positive impulsive system \eqref{eq:mainsyst} is asymptotically stable under constant dwell-time $\bar{T}$.
   \item\label{item:th1:1} The copositive linear form $V(x(t))=\lambda^\T x(t)$, $\lambda\in\mathbb{R}^n_{>0}$ is a discrete-time Lyapunov function for the $\bar{T}$-periodic impulsive system \eqref{eq:mainsyst} in the sense that the inequality
        \begin{equation}\label{eq:kdlsjdlsglab902}
          V(x(t^+_{k+1}))-V(x(t^+_{k}))\le -\mu^\T x(t^+_k)
        \end{equation}
     holds for some $\mu>0$, all $x(t_k)\in\mathbb{R}^n_{\ge0}$ and all $k\in\mathbb{N}$.
   \item\label{item:th1:2} There exists a vector $\lambda\in\mathbb{R}_{>0}^n$ such that the inequality
    \begin{equation}\label{eq:stabmono2}
  \lambda^\T\left[Je^{A\bar{T}}-I_n\right]<0
    \end{equation}
holds or, equivalently, the matrix $Je^{A\bar{T}}$ is Schur stable.
\item\label{item:th1:3} There exists a vector $\lambda\in\mathbb{R}_{>0}^n$ such that the inequality
    \begin{equation}
  \left[Je^{A\bar{T}}-I_n\right]\lambda<0
    \end{equation}
holds or, equivalently, the matrix $e^{A^\T\bar{T}}J^\T$ is Schur stable.
   \item\label{item:th1:4} There exist a differentiable vector-valued function $\zeta:[0,\bar{T}]\mapsto\mathbb{R}^n$, $\zeta(\bar T)\in\mathbb{R}_{>0}^n$, and a scalar $\eps>0$ such that the inequalities
  \begin{equation}\label{eq:c1b2}
    \zeta(\tau)^\T A -\dot{\zeta}(\tau)^\T\le0
  \end{equation}
  and
  \begin{equation}\label{eq:c2b2}
  \zeta(\bar{T})^\T J-\zeta(0)^\T+\eps \mathds{1}_n^\T\le0
  \end{equation}
  hold for all $\tau\in[0,\bar{T}]$.
   \item\label{item:th1:5} There exist a differentiable vector-valued function $\xi:[0,\bar{T}]\mapsto\mathbb{R}^n$, $\xi(0)\in\mathbb{R}_{>0}^n$, and a scalar $\eps>0$ such that the LMIs
  \begin{equation}\label{eq:c1c2}
    \xi(\tau)^\T A+\dot{\xi}(\tau)^\T \le0
  \end{equation}
  and
  \begin{equation}\label{eq:c1c3}
   \xi(0)^\T J-\xi(\bar{T})^\T +\eps \mathds{1}_n^\T \le0
  \end{equation}
  hold for all $\tau\in[0,\bar{T}]$.
   \end{enumerate}
  \end{theorem}
  \begin{proof}
    \noindent\textbf{Proof of \eqref{item:th1:0} $\boldsymbol{\Leftrightarrow}$ \eqref{item:th1:2}:} This follows from Proposition \ref{th:equivalence} and Proposition \ref{prop:Schur}.

\noindent\textbf{Proof of \eqref{item:th1:1} $\boldsymbol{\Leftrightarrow}$ \eqref{item:th1:2}:} It is immediate to see that the left-hand side of \eqref{eq:kdlsjdlsglab902} is given by
\begin{equation}
  \lambda^\T \left[Je^{A\bar{T}}-I_n\right]^\T x(t_k).
\end{equation}
The equivalence between the two statements immediately follows.

\noindent\textbf{Proof of \eqref{item:th1:2} $\boldsymbol{\Leftrightarrow}$ \eqref{item:th1:3}:} This follows from Proposition \ref{prop:Hurwitz} or  Proposition \ref{prop:Schur}.

\noindent\blue{\textbf{Proof of \eqref{item:th1:4} $\boldsymbol{\Rightarrow}$ \eqref{item:th1:2}:} Assume that the statement (e) holds. From \eqref{eq:c1b2}, after integration from 0 to $\bar{T}$, we find that
\begin{equation}\label{eq:djskodsjkklkl}
    \zeta(0)^\T e^{A\bar{T}}-\zeta(\bar T)^\T \le0.
\end{equation}
From \eqref{eq:c2b2}, we have that
  \begin{equation}
  \zeta(\bar{T})^\T J+\eps \mathds{1}_n^\T\le\zeta(0)^\T
  \end{equation}
  which, together with \eqref{eq:djskodsjkklkl}, implies that
  \begin{equation}
  \zeta(\bar T)^\T\left( J e^{A\bar T}-I_n\right)+\eps \mathds{1}_n^\T e^{A\bar T}\le0
  \end{equation}
which implies, in turn, that the condition \eqref{eq:stabmono2} holds with $\lambda=\zeta(\bar T)$.}

\noindent\blue{\textbf{Proof of \eqref{item:th1:2} $\boldsymbol{\Rightarrow}$ \eqref{item:th1:4}}: Assume that \eqref{eq:stabmono2} holds for some $\lambda\in\mathbb{R}_{>0}^n$. Define then the vector-valued function $\zeta^*(\tau) = e^{A^\T(\tau-\bar T)}\lambda$, $\tau\in[0,\bar T]$, where $\lambda$ is as in \eqref{eq:stabmono2}. Clearly, we have that
\begin{equation}
  \zeta^*(\tau)^\T A-\dot{\zeta}^*(\tau)^\T=0
\end{equation}
for all $\tau\in[0,\bar T]$ and hence the function $\zeta=\zeta^*$ verifies the inequality \eqref{eq:c1b2}.  From the definition of $\zeta^*(\tau)$, we have that $\zeta^*(\bar T)=\lambda$ and hence
\begin{equation}\label{kdlskdmls}
  \zeta^*(\bar T)^\T e^{-A\bar T}-\zeta^*(0)^\T=0.
\end{equation}
From \eqref{eq:stabmono2}, we have that $\lambda^\T Je^{A\bar T}<\lambda^\T$ or, equivalently, that $\zeta^*(\bar T)^\T Je^{A\bar T}<\zeta^*(\bar T)^\T$ which together with \eqref{kdlskdmls} implies that
\begin{equation}\label{kdlskdmls2}
  \zeta^*(\bar T)^\T J-\zeta^*(0)^\T<0
\end{equation}
which proves that the condition \eqref{eq:c2b2} holds with $\zeta=\zeta^*$ for some sufficiently small $\eps>0$.}

\noindent\textbf{Proof of \eqref{item:th1:2} $\boldsymbol{\Leftrightarrow}$ \eqref{item:th1:5}:} The proof between the two statements follows from the definition that $\xi(\tau)=\zeta(T-\tau)$. Alternatively, similar arguments as in the proof that \eqref{item:th1:2} $\boldsymbol{\Leftrightarrow}$ \eqref{item:th1:4} can also be considered.
  \end{proof}

\blue{\begin{remark}[Swapped system]\label{rem:swapped_constant}
Theorem \ref{th:cstDT2} characterizes the stability of the system via the Schur stability of the nonnegative matrix $Je^{A\bar T}$. Alternatively, the asymptotic stability of the system can be established via the Schur stability of the matrix $e^{A\bar T}J$. In this case, the condition \eqref{eq:stabmono2} naturally changes to $\lambda^\T\left[e^{A\bar{T}}J-I_n\right]<0$. Interestingly, the affine conditions given in the statements  \eqref{item:th1:4} and  \eqref{item:th1:5} remain the same with the exception that we require now the positivity of $\zeta(0)$ in place of the positivity of $\zeta(\bar T)$ in the former and the positivity of $\xi(\bar T)$ in place of $\xi(0)$ in the latter.
\end{remark}}

The main advantages of the conditions of statements \eqref{item:th1:4} and \eqref{item:th1:5} in the above results are the following.  First of all, the conditions are affine in the matrices $A$ and $J$ of the system, allowing then for an immediate extension to uncertain matrices and to control design (the latter requiring some additional steps that will be detailed in Section \ref{sec:stabz}). A particularity of the approach is that we only need the vectors $\zeta(\bar T)$ and $\xi(0)$ to be positive while most of the Lyapunov approaches would require the whole vector-valued functions $\zeta(\tau)$ and $\xi(\tau)$ to be positive on their domain (see e.g. \cite{Goebel:09}), a constraint that is computationally way more complex. The explanation is that, in spite of the fact that the conditions in statements \eqref{item:th1:4} and statement \eqref{item:th1:5} look like continuous-time stability conditions, they are in fact lifted discrete-time stability condition for which only the behavior of the inequalities at the extremal points of the domain of the lifting variable (i.e. the points $\tau=0$ and $\tau=\bar T$) are actually meaningful. The same discussion is made in \citep{Briat:11l,Briat:12h,Briat:13b} in the context of looped-functionals and in \cite{Briat:13d,Briat:14f,Briat:15f} in the context of clock-dependent conditions.

The main drawback of this procedure lies in the increase of the computational complexity as the conditions of statement \eqref{item:th1:4} and statement \eqref{item:th1:5} are infinite-dimensional linear programming conditions which may be hard to solve. Note, however, that they should be less complex than the infinite-dimensional semidefinite programming conditions obtained in \cite{Briat:13d,Briat:14f,Briat:15f}. Verifying these conditions in an efficient way will be the topic of Section \ref{sec:imp:computation}.

\subsection{Stability under minimum dwell-time}\label{sec:minimum}\label{sec:imp:minimum}

Let us consider now the minimum dwell-time case. In this setup, the dwell-times are assumed to satisfy the condition $T_k\ge\bar T$, $k\in\mathbb{N}$. Even though we do not assume here the existence of an upper-bound for the dwell-times a priori, we will restrict ourselves to the case of persistent impulses; i.e. for any chosen sequence $\{t_k\}_{k\in \mathbb{N}}$, there exists an upper-bound for the dwell-times. Hence, we exclude the case that $T_{k^*}=\infty$ for some $k^*\in\mathbb{N}$, which is not restrictive, as in such a case, the system would become a standard continuous-time linear positive system.

\bluee{The following result can be seen as a ``positive systems" counterpart of the results in \citep{Briat:11l,Briat:12h,Briat:13b} and provides a stability condition in terms of the discrete-time system \eqref{eq:embedded}. It can also be connected to \cite{Zhang:14b}:}
\begin{theorem}[Minimum dwell-time]\label{th:minDT2}
Let $A\in\mathbb{R}^{n\times n}$ and $J\in\mathbb{R}^{n\times n}$ be a Metzler and a nonnegative matrix, respectively. Then, the following statements are equivalent:
\begin{enumerate}[(a)]
  \item The linear form $V(x(t))=\lambda^\T x(t)$, $\lambda\in\mathbb{R}^n_{>0}$ is a Lyapunov function for the system \eqref{eq:mainsyst} in the sense that
      \begin{equation}
        \dot{V}(x(t))\le-\mu^\T x(t),\ t\in(t_k,t_{k+1})
      \end{equation}
      and
      \begin{equation}
        V(x(t^+_{k+1}))-V(x(t^+_k))\le-\nu^\T x(t^+_k)
      \end{equation}
      hold for some $\mu,\nu>0$, for all $x(t),x(t_k)\in\mathbb{R}_{\ge0}^n$ and for any sequence $\{t_k\}_{k\in\mathbb{N}}$ verifying $T_k\ge\bar T$, $k\in\mathbb{N}$.
      \item There exists a vector $\lambda\in\mathbb{R}_{>0}^n$ such that the inequalities
  \begin{equation}
    \lambda^\T A<0
  \end{equation}
  and
  \begin{equation}
    \lambda^\T Je^{A\theta}-\lambda^\T<0
  \end{equation}
  hold for all $\theta\ge\bar{T}$.
      \item There exists a vector $\lambda\in\mathbb{R}_{>0}^n$ such that the conditions
  \begin{equation}
    \lambda^\T A<0
  \end{equation}
  and
  \begin{equation}
  \lambda^\T Je^{A\bar T}-\lambda^\T<0
  \end{equation}
  hold.
  \item The pair $\left(A,Je^{A\bar{T}}-I\right)$ admits a common linear copositive Lyapunov function.
  \item We have that
  \begin{equation}
    \ker\begin{bmatrix}
      I & -A & -\left(Je^{A\bar{T}}-I\right)
    \end{bmatrix}\cap\mathbb{R}^{3n}_{\ge0}=\{0\}.
  \end{equation}
   \item There exist a matrix function $\zeta:[0,\bar{T}]\mapsto\mathbb{R}^n$, $\zeta(\bar T)\in\mathbb{R}_{>0}^n$, and a scalar $\eps>0$ such that the inequalities
\begin{equation}
      \zeta(\bar T)^\T A<0
    \end{equation}
  \begin{equation}
    \zeta(\tau)^\T A-\dot{\zeta}(\tau)^\T\le0
  \end{equation}
  and
  \begin{equation}
    \zeta(\bar{T})^\T J-\zeta(0)^\T +\eps \mathds{1}^\T \le0
  \end{equation}
 hold for all $\tau\in[0,\bar{T}]$.
     \item There exist a matrix function $\xi:[0,\bar{T}]\mapsto\mathbb{R}^n$, $\xi(0)\in\mathbb{R}_{>0}^n$, and a scalar $\eps>0$ such that the inequalities
\begin{equation}
      \xi(0)^\T A<0
    \end{equation}
  \begin{equation}
    \xi(\tau)^\T A+\dot{\xi}(\tau)^\T\le0
  \end{equation}
  and
  \begin{equation}
    \xi(0)^\T J-\xi(\bar{T})^\T +\eps \mathds{1}^\T \le0
  \end{equation}
 hold for all $\tau\in[0,\bar{T}]$.
\end{enumerate}
 Moreover, when one of the above statements holds, the linear positive impulsive system (\ref{eq:mainsyst}) is asymptotically stable under minimum dwell-time $\bar T$.\mendth
\end{theorem}
\begin{proof}
The proof that (a) is equivalent to (b) follows from the same arguments as in the proof of Theorem \ref{th:minDT2}. The proof that (b) implies (c) is immediate. To prove the converse, we assume that $\lambda^\T Je^{A\bar T}-\lambda^\T<0$, which implies that
\begin{equation}\label{dksdkhfdkdjdkd}
  \lambda^\T Je^{A(\bar T+s)}-\lambda^\T e^{As}<0,\ s\ge0
\end{equation}
where we have used the fact that $e^{As}\ge0$ for all $s\ge0$. Note also that since $\lambda^TA<0$, then we have that $\lambda^\T (e^{As}-I_n)\le0$ for all $s\ge0$. Finally, adding the latter inequality to \eqref{dksdkhfdkdjdkd} implies that
  \begin{equation}
     \lambda^\T Je^{A(\bar T+s)}-\lambda^\T<0
  \end{equation}
  holds for all $s\ge0$, which proves the implication. The rest of the proof can be carried out as for Theorem \ref{th:minDT2}.
\end{proof}

\blue{\begin{remark}[Swapped system]\label{rem:swapped_minimum}
As in the constant dwell-time case, a stability condition involving the swapped system \eqref{eq:embedded_swapped} can be formulated. In this case, the conditions of Theorem \ref{th:minDT2}, (c)  become
\begin{equation}\label{eq:minDTrem1}
    \lambda^\T A<0
  \end{equation}
  and
  \begin{equation}\label{eq:minDTrem2}
  \lambda^\T e^{A\bar T}J-\lambda^\T<0.
  \end{equation}
  Note, however, that the proof corresponding to the conditions above is completely different than the proof considered for Theorem \ref{th:minDT2}, (c) since it relies on perturbation arguments; see e.g. \cite{Briat:13d,Briat:14f,Briat:15f} for a proof in the quadratic Lyapunov function framework. Affine conditions corresponding to the conditions \eqref{eq:minDTrem1}-\eqref{eq:minDTrem2} can be obtained using the conditions in Theorem \ref{th:cstDT2} and Remark \ref{rem:swapped_constant}.
\end{remark}}

 The above results interestingly connect the problem of finding a common linear copositive Lyapunov functions for a linear positive switched system with two subsystems to the problem of establishing whether a linear positive impulsive system is stable under some minimum dwell-time constraint. Indeed, in the context of Theorem \ref{th:minDT2}, the matrix of the first subsystem is $A$ and that of the second subsystem is $Je^{A\bar T}-I_n$. When the swapped system is considered (see Remark \ref{rem:swapped_minimum}), the matrix of the first subsystem is also $A$ but that of the second subsystem is $e^{A\bar T}J-I_n$. In this regard, characterizing the existence of a linear copositive Lyapunov function can be achieved using existing methods; such as those developed in \cite{Mason:07,Blanchini:15}. However, these conditions are difficult to consider when the matrix $A$ is uncertain or when design is the main goal. The statements (f) and (g) provide alternative equivalent conditions that are affine in $A$ and $J$ and can then be hence considered for uncertain systems and for design purposes. A last interesting point is that the results for arbitrary dwell-time can be retrieved by letting $\bar T\to0$ in the above theorem as also noticed in \cite{Geromel:06b,Briat:15d}.

\subsection{Stability under maximum dwell-time}\label{sec:imp:maximum}

We consider now the maximum dwell-time case. In this setup, the dwell-times are assumed to satisfy a maximum dwell-time condition; i.e. $T_k\le\bar T$, $k\in\mathbb{N}$. \bluee{The following result can be seen as a ``positive systems" counterpart of the results in \citep{Briat:11l,Briat:12h} and provides a stability condition in terms of the discrete-time system \eqref{eq:embedded}:}

\begin{theorem}[Maximum dwell-time]\label{th:maxDT2}
Let $A\in\mathbb{R}^{n\times n}$ and $J\in\mathbb{R}^{n\times n}$ be a Metzler and a nonnegative matrix, respectively. Then, the following statements are equivalent:
\begin{enumerate}[(a)]
  \item The linear form $V(x(t))=\lambda^\T x(t)$, $\lambda\in\mathbb{R}^n_{>0}$ is a Lyapunov function for the system \eqref{eq:mainsyst} in the sense that
      \begin{equation}
        \dot{V}(x(t))\ge\mu^\T x(t),\ t\in(t_k,t_{k+1})
      \end{equation}
      and
      \begin{equation}
        V(x(t^+_{k+1}))-V(x(t^++_k))\le-\nu^\T x(t^+_k)
      \end{equation}
      hold for some $\mu,\nu>0$, all $x(t),x(t_k)\in\mathbb{R}_{\ge0}^n$ and any sequence $\{t_k\}_{k\in\mathbb{N}}\in\mathbb{I}_{\bar{T}}$.
      \item There exists a vector $\lambda\in\mathbb{R}_{>0}^n$ such that the inequalities
  \begin{equation}\label{eq:dkslddksld6868893}
    \lambda^\T A>0
  \end{equation}
  and
  \begin{equation}\label{eq:dkslddksld6868892}
    \lambda^\T(Je^{A\theta}-I_n)<0
  \end{equation}
  hold for all $\theta\le\bar{T}$.
      \item There exists a vector $\lambda\in\mathbb{R}_{>0}^n$ such that the conditions
  \begin{equation}\label{eq:dkslddksld6868890}
    \lambda^\T A>0
  \end{equation}
  and
  \begin{equation}\label{eq:dkslddksld6868891}
  \lambda^\T (Je^{A\bar T}-I_n)<0
  \end{equation}
  hold.
  \item The pair $\left(-A,Je^{A\bar{T}}-I\right)$ admits a common linear copositive Lyapunov function.
  \item We have that
  \begin{equation}
    \ker\begin{bmatrix}
      I & A & -\left(Je^{A\bar{T}}-I\right)
    \end{bmatrix}\cap\mathbb{R}^{3n}_{\ge0}=\{0\}.
  \end{equation}
    \item There exist a matrix function $\zeta:[0,\bar{T}]\mapsto\mathbb{R}^n$, $\zeta(\bar T)\in\mathbb{R}_{>0}^n$, and a scalar $\eps>0$ such that the inequalities
\begin{equation}
      \zeta(\bar T)^\T A>0
    \end{equation}
  \begin{equation}
    \zeta(\tau)^\T A-\dot{\zeta}(\tau)^\T\le0
  \end{equation}
  and
  \begin{equation}
    \zeta(\bar{T})^\T J-\zeta(0)^\T +\eps \mathds{1}^\T \le0
  \end{equation}
 hold for all $\tau\in[0,\bar{T}]$.
     \item There exist a matrix function $\xi:[0,\bar{T}]\mapsto\mathbb{R}^n$, $\xi(0)\in\mathbb{R}_{>0}^n$, and a scalar $\eps>0$ such that the inequalities
\begin{equation}
      \xi(0)^\T A>0
    \end{equation}
  \begin{equation}
    \xi(\tau)^\T A+\dot{\xi}(\tau)^\T\le0
  \end{equation}
  and
  \begin{equation}
    \xi(0)^\T J-\xi(\bar{T})^\T +\eps \mathds{1}^\T \le0
  \end{equation}
 hold for all $\tau\in[0,\bar{T}]$.
\end{enumerate}
 Moreover, when one of the above statements holds, the linear positive impulsive system (\ref{eq:mainsyst}) is asymptotically stable under maximum dwell-time $\bar T$.\mendth
\end{theorem}
\begin{proof}
  The proof follows from the same lines as the proof of Theorem \ref{th:minDT2}.
\end{proof}

\blue{\begin{remark}[Swapped system]\label{rem:swapped_maximum}
As in the constant and minimum dwell-time case, a stability condition involving the swapped system \eqref{eq:embedded_swapped} can be formulated. In this case, the conditions of Theorem \ref{th:maxDT2}, (c)  become
\begin{equation}\label{eq:maxDTrem1}
    \lambda^\T A>0
  \end{equation}
  and
  \begin{equation}\label{eq:maxDTrem2}
  \lambda^\T e^{A\bar T}J-\lambda^\T<0.
  \end{equation}
  As for the minimum dwell-time case, the proof corresponding to the conditions above is completely different than the proof considered for Theorem \ref{th:minDT2}, (c) since it relies on perturbation arguments; see e.g. \cite{Briat:13d,Briat:14f,Briat:15f} for a proof in the quadratic Lyapunov function framework. Affine conditions corresponding to the conditions \eqref{eq:maxDTrem1}-\eqref{eq:maxDTrem2} can be obtained using the conditions in Theorem \ref{th:cstDT2} and Remark \ref{rem:swapped_constant}.
\end{remark}}

\bluee{\begin{remark}\label{rem:antistable}
  The requirement that $\mu$ be positive in the statement (a) of Theorem \ref{th:maxDT2} (or equivalently that \eqref{eq:dkslddksld6868893} holds for some $\lambda\in\mathbb{R}^d_{>0}$ in the other statements) is equivalent to saying that the matrix $-A$ is Hurwitz stable, which excludes unstable matrices with stable eigenvalues. The motivation for its consideration is that it allows for the simplification of the stability conditions by turning the semi-infinite dimensional feasibility problem \eqref{eq:dkslddksld6868892} into the finite-dimensional problem \eqref{eq:dkslddksld6868891}. Indeed, while it is clear that if the condition \eqref{eq:dkslddksld6868892} holds for all $\theta\le\bar T$, then \eqref{eq:dkslddksld6868891} holds as well, the converse is not true in general. Adding the condition \eqref{eq:dkslddksld6868890} precisely allows to recover that reverse implication. Finally, it is important to stress that if $-A$ is not Hurwitz stable, then a maximum dwell-time condition can still be obtained using the range dwell-time results discussed in the next section.
\end{remark}}

\subsection{Stability under range dwell-time}\label{sec:range}\label{sec:imp:range}

We address now the range dwell-time case, that is the case when $T_k\in[T_{min},T_{max}]$, $k\in\mathbb{N}$, for some prescribed bounds $0<T_{min}\le T_{max}<\infty$. \bluee{The following result can be seen as a ``positive systems" counterpart of the results in \citep{Briat:11l,Briat:12h,Briat:13b} and provides a stability condition in terms of the discrete-time system \eqref{eq:embedded_swapped}:}
\begin{theorem}[Stability under range dwell-time]\label{th:rangeDT}
 Let $A\in\mathbb{R}^{n\times n}$ and $J\in\mathbb{R}^{n\times n}$ be a Metzler and a nonnegative matrix, respectively. Then, the following statements are equivalent:
 \begin{enumerate}[(a)]
   \item\label{item:th3:1} The copositive linear form $V(x(t))=\lambda^\T x(t)$, $\lambda\in\mathbb{R}^n_{>0}$ is a discrete-time Lyapunov function for the impulsive system \eqref{eq:mainsyst} in the sense that the inequality
        \begin{equation}\label{eq:kdlsjdlsglab901}
          V(x(t_{k+1}))-V(x(t_{k}))\le -\mu^\T x(t_k)
        \end{equation}
     holds for some $\mu>0$, all $x(t_k)\in\mathbb{R}^n_{\ge0}$, $T_k\in[T_{min},T_{max}]$ and all $k\in\mathbb{N}$.
   \item\label{item:th3:2} There exists a vector $\lambda\in\mathbb{R}_{>0}^n$ such that the inequality
    \begin{equation}\label{eq:stabmono1}
  \lambda^\T\left[e^{A\theta}J-I_n\right]<0
    \end{equation}
holds for all $\theta\in[T_{min},T_{max}]$.
   \item\label{item:th3:4} There exist a differentiable vector-valued function $\zeta:[0,\bar{T}]\mapsto\mathbb{R}^n$, $\zeta(0)\in\mathbb{R}_{>0}^n$, and a scalar $\eps>0$ such that the inequalities
  \begin{equation}
    \zeta(\tau)^\T A -\dot{\zeta}(\tau)^\T\le0
  \end{equation}
  and
  \begin{equation}
  \zeta(\theta)^\T J-\zeta(0)^\T+\eps \mathds{1}_n^\T\le0
  \end{equation}
  hold for all $\tau\in[0,T_{max}]$ and $\theta\in[T_{min},T_{max}]$.
   \end{enumerate}
  \end{theorem}
  \begin{proof}
It is readily seen that the Lyapunov condition \eqref{eq:kdlsjdlsglab901} is equivalent to the inequality \eqref{eq:stabmono1}. The proof of the equivalence between the statements (b) and (c) follows from exactly the same arguments as in the proof of  Theorem \ref{th:cstDT2}. It is thus omitted.
  \end{proof}

\bluee{\begin{remark}[Non-swapped system]\label{rem:swapped_range}
Analogous conditions can be obtained by using instead the original discretized system \eqref{eq:embedded} instead of the swapped one \eqref{eq:embedded_swapped}. In this case, the condition \eqref{eq:stabmono1} of statement (b) is naturally substituted by
\begin{equation}
  \lambda^\T\left[Je^{A\theta}-I_n\right]<0
\end{equation}
while the conditions of statement (c) have to be replaced by the existence of  a vector-valued function $\xi:[0,\bar{T}]\mapsto\mathbb{R}^n$, $\xi(0)\in\mathbb{R}_{>0}^n$, and a scalar $\eps>0$ such that the inequalities
  \begin{equation}
    \xi(\tau)^\T A+\dot{\xi}(\tau)^\T\le0
  \end{equation}
  and
  \begin{equation}
    \xi(0)^\T J-\xi(\theta)^\T +\eps \mathds{1}^\T \le0
  \end{equation}
 hold for all $\tau\in[0,\bar{T}]$ and all $\theta\in[T_{min},T_{max}]$.
\end{remark}}

\bluee{\begin{remark}\label{rem:antistable2}
  As stated in Remark \ref{rem:antistable}, when the matrix $-A$ is not Hurwitz stable, a maximum dwell-time condition can be obtained from Theorem \ref{th:rangeDT} by setting $T_{min}=0$ (or a very small value such as $10^{-5}$) and $T_{max}=\bar T$.
\end{remark}}

\subsection{Computational considerations}\label{sec:imp:computation}

It seems important to carefully address the problem of verifying the infinite-dimensional stability conditions of the previous results. We detail here three methods for accurately and efficiently verifying them. Only the relaxation of the conditions in Theorem \ref{th:cstDT2}, \eqref{item:th1:4} are considered here for simplicity.  However, relaxed finite-dimensional results for the other conditions can be easily obtained using the same procedure.

\subsubsection{Piecewise linear approach}\label{sec:discretization}

This approach is based on the piecewise linear approximation of the function $\zeta$ or $\xi$ in the results. This method has been initially proposed in \citep{Allerhand:11} in order to relax infinite-dimensional LMI conditions arising in the analysis of switched systems into finite-dimensional LMI conditions. Below is the approximation of the conditions in Theorem \ref{th:cstDT2}, statement \eqref{item:th1:4}:
\begin{proposition}\label{prop:discretization}
The following statements are equivalent:
\begin{enumerate}[(a)]
  \item For some $d_g\in\mathbb{N}$, the function
  \begin{equation}
    \zeta\left(\dfrac{i\bar{T}}{d_g}+\tau\right)=\dfrac{\zeta_{i+1}-\zeta_i}{\bar{T}/d_g}\tau+\zeta_i,\ \tau\in\left[0,\dfrac{\bar{T}}{d_g}\right],\ i=0,\ldots,d_g-1
  \end{equation}
  verifies the conditions of Theorem \ref{th:cstDT2}, statement \eqref{item:th1:4}.
  \item There exist vectors $\zeta_i\in\mathbb{R}^n_i$, $\zeta_{d_g}\in\mathbb{R}_{>0}^n$, such that the following conditions
  \begin{equation}
    -\dfrac{\zeta_{i+1}-\zeta_i}{\bar{T}/d_g}+\zeta_i^\T A\le0,
  \end{equation}
    \begin{equation}
    -\dfrac{\zeta_{i+1}-\zeta_i}{\bar{T}/d_g}+\zeta_{i+1}^\T A\le0,
  \end{equation}
  and
    \begin{equation}
    \zeta_{d_g}^\T J-\zeta_0^\T+\eps\mathds{1}_n^\T\le0
  \end{equation}
  hold for all $i=0,\ldots,d_g-1$.
\end{enumerate}
\end{proposition}
\begin{proof}
  The proof follows from direct substitutions.
\end{proof}

\blue{\begin{remark}[Asymptotic exactness]
  The above relaxation yields a finite-dimensional linear program that can be efficiently solved. Moreover, this relaxation can be shown to be asymptotically exact in the sense that if the original conditions of Theorem \ref{th:cstDT2}, statement \eqref{item:th1:4}, hold then we can find a $d_g\in\mathbb{N}$ for which the above conditions are feasible. This can be proved using the same arguments as in \cite{Xiang:15a}.
\end{remark}}

\subsubsection{Sum of squares approach}\label{sec:SOS}

This approach is based on Putinar's Positivstellensatz \citep{Putinar:93} and is solved using semidefinite programming methods, as initially proposed in \citep{Parrilo:00}. Before stating the main result, we need first to define some terminology. A multivariate polynomial $p(x)$ is said to be a sum-of-squares (SOS) polynomial if it can be written as $\textstyle p(x)=\sum_{i}q_i(x)^2$ for some polynomials $q_i(x)$. A polynomial vector $p(x)\in\mathbb{R}^n$ is said to \emph{componentwise sum-of-squares} (CSOS) if each of its components is an SOS polynomial. Checking whether a polynomial is SOS can be exactly cast as a semidefinite program \citep{Parrilo:00,Chesi:10b} that can be easily solved using semidefinite programming solvers such as SeDuMi \citep{Sturm:01a} or SDPT3 \cite{Tutuncu:03}. The package SOSTOOLS \citep{sostools3} can be used to formulate the SOS program in a convenient way. This approach has been used for instance for the analysis and control of delay systems \cite{Papa:07,Peet:09,Peet:13}, hybrid systems \cite{Prajna:03}, sampled-data systems \cite{Seuret:13} impulsive systems \citep{Briat:12h,Briat:13d}, switched systems \citep{Briat:13b,Briat:14f}, pseudo-periodic systems with impulses \citep{Briat:15f} and stochastic impulsive systems \citep{Briat:15i}.

Below is the approximation of the conditions in Theorem \ref{th:cstDT2}, statement \eqref{item:th1:4}:
\begin{proposition}\label{prop:SOS}
  Let $d_s\in\mathbb{N}$, $\eps>0$ and $\epsilon>0$ be given and assume that there exists polynomial vectors $\zeta:\mathbb{R}\mapsto\mathbb{R}^n$ and $\gamma:\mathbb{R}\mapsto\mathbb{R}^n$ of degree $2d_s$ such that
  \begin{enumerate}[(i)]
    \item $\gamma(\tau)$ is CSOS,
    \item $\zeta(\bar T)-\epsilon\mathds{1}_n\ge0$ (or is CSOS),
    \item $-\zeta(\tau)^\T A+\dot{\zeta}(\tau)-\tau(\bar{T}-\tau)\gamma(\tau)^\T$ is CSOS and
    \item $-\zeta(\bar T)^\T J+\zeta(0)^\T-\eps\mathds{1}_n^\T\ge0$ (or is CSOS).
  \end{enumerate}
  Then, the conditions of Theorem \ref{th:cstDT2}, statement \eqref{item:th1:4} hold with the same $\zeta(\tau)$ and the system \eqref{eq:mainsyst} is asymptotically stable under constant dwell-time $\bar T$.
\end{proposition}
\begin{proof}
  Clearly, the condition in (ii) is equivalent to saying that $\zeta(\bar T)>0$ and that in (iv) is equivalent to the condition \eqref{eq:c2b2}. We then consider the condition in (iii), which implies that $\zeta(\tau)^\T A-\dot{\zeta}(\tau)\le-\tau(\bar{T}-\tau)\gamma(\tau)^\T$ for all $\tau\in\mathbb{R}$. Since $\gamma(\tau)$ is CSOS, then we have that $\gamma(\tau)\ge0$ for all $\tau\in\mathbb{R}$ and, therefore, this implies that $\zeta(\tau)^\T A-\dot{\zeta}(\tau)\le0$ for all $\tau\in[0,\bar T]$ which is exactly the condition \eqref{eq:c1b2}. The proof is complete.
\end{proof}

\blue{\begin{remark}[Asymptotic exactness]
As for the piecewise linear approximation, it can be shown that the above relaxation is asymptotically exact in the sense that if the original conditions of Theorem \ref{th:cstDT2}, statement \eqref{item:th1:4}, hold then we can find a $d_s\in\mathbb{N}$ for which the above conditions are feasible. This was proved in the context of the relaxation of the LMI conditions characterizing the stability of switched systems in \cite{Briat:14f}. The same arguments apply here.
\end{remark}}

\subsubsection{Handelman-based approach}\label{sec:handelman}

This approach is based on Handelman's Theorem \citep{Handelman:88} that states that if a polynomial is positive on a compact polytope of the form $\{x:\mathbb{R}^n:f_i(x)\ge0,f_i\ \textnormal{affine},i=1,\ldots,N\}$, then it can be expressed as a nonnegative linear combinations of products of the functions $f_1,\ldots,f_N$. This elegant result has been considered in \citep{Briat:11h} for deriving finite-dimensional linear programs establishing the stability of linear uncertain positive systems or in \citep{Kamyar:14,Kamyar:15} for the construction of Lyapunov functions and the verification of polynomials optimization problems.

Below is the approximation of the conditions in Theorem \ref{th:cstDT2}, statement \eqref{item:th1:4}:
\begin{proposition}\label{prop:handelman}
  Let $d_h\in\mathbb{N}$, $\eps>0$ and $\epsilon>0$ be given and assume that there exists a vector-valued polynomial $\zeta:\mathbb{R}\mapsto\mathbb{R}^n$ of degree $d_h$ and parameters $\theta_{ij}\in\mathbb{R}^n_{\ge0}$, $0\le i+j\le d$ such that
  \begin{enumerate}[(i)]
    \item $\theta_{ij}\ge0$ for all $0\le i+j\le d_h$,
    \item $\zeta(0)-\epsilon\mathds{1}_n\ge0$,
    \item the polynomial vector $\displaystyle -\zeta(\tau)^\T A+\dot{\zeta}(\tau)-\sum_{0\le i+j\le d_h}\theta_{ij}\tau^i(\bar{T}-\tau)^j$ has nonnegative coefficients and
    \item  $-\zeta(\bar T)^\T J+\zeta(0)^\T-\eps\mathds{1}_n^\T\ge0$.
  \end{enumerate}
  Then, the conditions of Theorem \ref{th:cstDT2}, statement \eqref{item:th1:4} hold with the same $\zeta(\tau)$ and the system \eqref{eq:mainsyst} is asymptotically stable under constant dwell-time $\bar T$.
\end{proposition}
\begin{proof}
  The proof follows from the same line as for Proposition \ref{prop:SOS} with the difference that we consider here Handelman's theorem in place of Putinar's Positivstellensatz. Following similar arguments as in the proof of Proposition \ref{prop:SOS}, if the conditions in statement (i) and (iii) are satisfied then we have that $\zeta(\tau)^\T A-\dot{\zeta}(\tau)\le0$ for all $\tau\in[0,\bar T]$.  The fact that the condition is stated in terms of the nonnegativity of the coefficients (as opposed to the nonnegativity of the eigenvalues of the Gram matrix for the sum of squares relaxation) comes from the fact that Handelman's theorem deals with an equality between the supposedly positive polynomial and a nonnegative sum of products of the basis functions of the polytope. However, we are not interested here in finding an equality but in finding a positive lower bound for the polynomial that can be expressed as  a nonnegative sum of products of basis functions. \bluee{The polynomial in statement (iii) can be rewritten in the form $\textstyle\left(\col_{i=0}^{d_h}\{\tau^i\}\right)^\T \varpi(\theta)$ where $\varpi(\theta)\in\mathbb{R}^{(d_h+1)\times n}$ is affine in the scalars $\theta_{ij}\ge0$ and contains all the coefficients of the polynomials in (iii). Using now the fact that $\tau\ge0$, then if $\varpi(\theta)\ge0$ (i.e. the polynomials have nonnegative coefficients) we have that  $\textstyle\varpi(\theta)^\T\left(\col_{i=0}^{d_h+2}\{\tau^i\}\right)\ge0$ for all $\tau\ge0$. This then implies that  $\textstyle-\zeta(\tau)^\T A+\dot{\zeta}(\tau)\ge\sum_{0\le i+j\le d_h}\theta_{ij}\tau^i(\bar{T}-\tau)^j$ holds for all $\tau\in\mathbb{R}_{\ge0}$ which implies, in turn, that $\textstyle -\zeta(\tau)^\T A+\dot{\zeta}(\tau)\ge0$ for all $\tau\in[0,\bar T]$. This completes the proof. }
\end{proof}

\begin{remark}
  It is also important to stress that the third condition could have also been formulated in terms of an SOS condition instead of in terms of a condition on the coefficients on some polynomial. However, the former would lead to semidefinite programming conditions as opposed to linear programming ones for the latter. It is finally interesting to note the latter condition shares similarities with the those obtained from Polya's theorem \cite{Polya:28,Powers:01}.
\end{remark}

\blue{\begin{remark}[Asymptotic exactness, \cite{Handelman:88}]
As for the piecewise linear approximation, the relaxed problem is also a finite-dimensional linear program. Moreover, it can also be shown in this case that  the above relaxation is asymptotically exact in the sense that if the original conditions of Theorem \ref{th:cstDT2}, statement \eqref{item:th1:4}, hold then we can find a $d_h\in\mathbb{N}$ for which the above conditions are feasible.
\end{remark}}

\subsection{Examples}\label{sec:imp:examples}

We illustrate here the efficiency obtained results and compare them with existing ones. Comparisons between the proposed relaxation results are also made for completeness.
\begin{example}[Minimum dwell-time \#1]
Let us consider the system \eqref{eq:mainsyst} with the matrices \cite{Zhang:14b}
\begin{equation}\label{eq:minDT1}
  A=\begin{bmatrix}
    -3 & 1\\ 0 & -3  \end{bmatrix}\ \textnormal{and}\ J=\begin{bmatrix}
2 & 1\\ 0 & 2
  \end{bmatrix}.
\end{equation}
Clearly, the matrix $A$ is Hurwitz stable whereas the $J$ is Schur anti-stable. Simple calculations show that the matrix $e^{AT}J$ is Schur stable provided that $T>T^*_{min}:=\log(2)/3\approx 0.2311$. Using the minimum dwell-time stability result of Theorem \ref{th:minDT2}, (a), we find the value 0.2311 for the estimated minimum dwell-time value. As this value coincides with the lower bound obtained for the constant dwell-time case, then we can conclude that the estimated minimum dwell-time is exact. This proves that the proposed method is way less conservative than the one proposed in \cite{Zhang:14b} where the value 1.0986 is found for the minimum dwell-time.
\end{example}

\begin{example}[Minimum dwell-time \#2]
Let us consider the system \eqref{eq:mainsyst} with the matrices
\begin{equation}\label{eq:minDT2}
  A=\begin{bmatrix}
    -3 & 1\\ 2 & -8
  \end{bmatrix}\ \textnormal{and}\ J=\begin{bmatrix}
1 & \delta\\ 2 & 1
  \end{bmatrix}.
\end{equation}

When $\delta=1$, the conditions of Theorem \ref{th:minDT2}, (a) and of Remark \ref{rem:swapped_minimum} both yield the value $0.2443$ for the minimum dwell-time. This value was also found in \cite{Briat:11l,Briat:13d} using conditions based on quadratic Lyapunov functions. This value also coincides with the minimal value obtained in the constant dwell-time case, emphasizing then the exactness of the computed  minimum dwell-time.

On the other hand, when $\delta=3$, the conditions of Remark \ref{rem:swapped_minimum} yield the value 0.4290 while the value 0.3615 is obtained using the conditions of Theorem \ref{th:minDT2}, (a). The latter value is also found when using the conditions based on the use of quadratic Lyapunov functions stated in \cite{Briat:11l,Briat:13d} and coincides with the value obtained  in the constant dwell-time case. Hence exactness of the estimate of the minimum dwell-time is again proved. Moreover, the discrepancy between the results obtained with the conditions in Theorem \ref{th:minDT2} and in Remark \ref{rem:swapped_minimum} demonstrate the non-equivalence of the conditions and in the necessity of considering both in order to obtain less conservative result. Note that this phenomenon does not occur when considering quadratic Lyapunov function; see e.g. \cite{Briat:12h}. A comparison of the different computational methods described in Section \ref{sec:imp:computation} is performed and the results are summarized in Table \ref{table:minDT1} where we can see that all the approaches are able to approach the value obtained using the condition of Theorem \ref{th:minDT2}, (b), when the discretization order $d_g$ or the polynomial degrees $d_s$ or $d_h$ increase. However, this comes at the price of an increase of the computational complexity and the solving time. \blue{We can see that the sum of squares approach performs the best here in terms of solving time and total number of variables (sum of the number primal and dual variables). Then, comes the relaxation based on Handelman's theorem and, finally, the discretization approach, which performs the worst here. Note that the poor convergence and high computational complexity of the discretization approach was also pointed out in \cite{Briat:14f,Briat:15f}. However, for larger problems, the approach based on Handelman's Theorem may be beneficial over the SOS one because of the linear nature of the conditions, as opposed to the semidefinite structure of the SOS conditions that may indeed scale badly.}
\setlength\extrarowheight{1pt}
\begin{table}
  \centering
  \caption{Comparison of the numerical results obtained for the minimum dwell-time using the discretization approach (Section \ref{sec:discretization}), the sum of squares approach (Section \ref{sec:SOS}) and the Handelman-based approach (Section  \ref{sec:handelman}). For each method, we give the estimate for the minimum dwell-time $\bar T$, the number of primal/dual variables of the optimization problem and the overall solving time in seconds. For fairness, the solver SeDuMi is used for solving all the optimization problems.}\label{table:minDT1}
  \begin{tabular}{|c||c||c|c|c|c|c|}
  \hline
    System & Result & Method & Computed $\bar{T}$ & No. vars. & Solving time\\
    \hline
    \hline
  \multirow{11}{*}{\makecell{System \eqref{eq:minDT2},\\ $\delta=1$}} & Rem. \ref{rem:swapped_minimum}    & -- &  0.2443 &  6/2 & 0.5278\\
                                                                                                                                         & Th. \ref{th:minDT2}, (c)  & -- & 0.2443  & 6/2 & 0.4703\\
  \cline{2-6}
                                                                                                            & Th. \ref{th:minDT2}, (f) & Discretized ($d_g=11$) &0.2843  & 46/22 & 0.7269\\
                                                                                                            & Th. \ref{th:minDT2}, (f) & Discretized ($d_g=51$) & 0.2521 & 206/102 & 1.3825\\
                                                                                                            & Th. \ref{th:minDT2}, (f) & Discretized ($d_g=101$) & 0.2482 & 406/202  & 3.2078\\
                                                                                                            & Th. \ref{th:minDT2}, (f) & Discretized ($d_g=151$) & 0.2469 & 606/302 & 5.7835\\
                                                                                                              \cline{2-6}
                                                                                                            & Th. \ref{th:minDT2}, (f) & SOS ($d_s=1$) & 0.2769 & 38/16 & 0.1925\\
                                                                                                            & Th. \ref{th:minDT2}, (f) & SOS ($d_s=2$) & 0.2450 &  66/20 & 0.2584\\
                                                                                                            & Th. \ref{th:minDT2}, (f) & SOS ($d_s=3$) & 0.2444 & 102/24 & 0.3112\\
                                                                                                              \cline{2-6}
                                                                                                            & Th. \ref{th:minDT2}, (f) & Handelman ($d_h=3$) & 0.2598 & 34/28 & 0.7398\\
                                                                                                            & Th. \ref{th:minDT2}, (f) & Handelman ($d_h=5$) & 02450 & 60/54& 0.8404\\
                                                                                                            & Th. \ref{th:minDT2}, (f) & Handelamn ($d_h=7$) & 0.2443 & 94/88 & 0.9914\\
  \hline
  \hline
  \multirow{11}{*}{\makecell{System \eqref{eq:minDT2},\\ $\delta=3$}} & Rem. \ref{rem:swapped_minimum}    & --&  0.4290 & 6/2 & 0.5801\\
                                                                                                                                         & Th. \ref{th:minDT2}, (c)  & -- & 0.3615  & 6/2 & 0.5054\\
    \cline{2-6}
                                                                                                            & Th. \ref{th:minDT2}, (f) & Discretized  ($d_g=11$) & 0.4501 & 46/22 & 0.6581\\
                                                                                                            & Th. \ref{th:minDT2}, (f) & Discretized  ($d_g=51$) & 0.3778 & 206/102  & 1.4356\\
                                                                                                            & Th. \ref{th:minDT2}, (f) & Discretized  ($d_g=101$) & 0.3696 & 406/202  & 3.0860\\
                                                                                                            & Th. \ref{th:minDT2}, (f) & Discretized  ($d_g=151$) & 0.3669 & 606/302 & 5.8264\\
                                                                                                              \cline{2-6}
                                                                                                            & Th. \ref{th:minDT2}, (f) & SOS  ($d_s=1$) & 0.6078 & 38/16 & 0.2907\\
                                                                                                            & Th. \ref{th:minDT2}, (f) & SOS  ($d_s=2$) & 0.3686 & 66/20 & 0.2385\\
                                                                                                            & Th. \ref{th:minDT2}, (f) & SOS  ($d_s=3$) & 0.3617 & 102/24 & 0.3108\\
                                                                                                              \cline{2-6}
                                                                                                            & Th. \ref{th:minDT2}, (f) & Handelman  ($d_h=3$) & 0.4698 & 34/28 & 0.6792\\
                                                                                                            & Th. \ref{th:minDT2}, (f) & Handelman  ($d_h=6$) & 0.3636 & 76/70 & 0.8507\\
                                                                                                            & Th. \ref{th:minDT2}, (f) & Handelamn  ($d_h=10$) & 0.3615 & 160/154 & 1.0940\\
  \hline
  \end{tabular}
\end{table}
\end{example}

\begin{example}[Maximum dwell-time \#1]
Let us consider the system \eqref{eq:mainsyst} with the matrices \cite{Zhang:14b}
\begin{equation}\label{eq:maxDT1}
  A=\begin{bmatrix}
    0.5 & 1\\ 0 & 0.5  \end{bmatrix}\ \textnormal{and}\ J=\begin{bmatrix}
0.1 & 0.2\\ 0 & 0.1
  \end{bmatrix}.
\end{equation}
Clearly, the matrix $A$ is anti-(Hurwitz)stable whereas $J$ is Schur stable. It is immediate to establish that for any $T<T^*_{max}=2\log(10)\approx 4.6051$, the matrix $e^{AT}J$ is Schur stable. Using Theorem \ref{th:maxDT2}, (c), we find the value 4.6051 for the estimated maximum dwell-time, which coincides with the value $T^*_{max}$ computed for the constant dwell-time case. Hence, the estimate of the maximum dwell-time is exact and is way more accurate than the value 1.2040 computed with the method of \cite{Zhang:14b}.
\end{example}

\begin{example}[Maximum dwell-time \#2]
Let us consider the system \eqref{eq:mainsyst} with the matrices \cite{Zhang:14b}
\begin{equation}\label{eq:maxDT2}
  A=\begin{bmatrix}
    -1 & 5\\ 2 & 3  \end{bmatrix}\ \textnormal{and}\ J=\begin{bmatrix}
   0.15 & 0.1\\ 0.05 &0.25
  \end{bmatrix}.
\end{equation}
The matrix $A$ is unstable whereas $J$ is Schur stable, hence the system admits a maximum dwell-time. However, since $A$ is not is anti-(Hurwitz)stable, then Theorem \ref{th:maxDT2} does not apply (see Remark \ref{rem:antistable} and Remark \ref{rem:antistable2}) and we have to consider instead a range dwell-time result; i.e. Theorem \ref{th:rangeDT}. We get the results summarized in Table \ref{table:rangeDT1} where we compare the SOS approach for solving the conditions of Theorem \ref{th:rangeDT}, (c) and a gridding approach for solving the condition of \ref{th:rangeDT}, (b) where we consider $N_p=201$ gridding points. We can observe that despite the complexity in terms of the total number of variables is comparable, the SOS conditions are way faster to solve than the gridded ones. This demonstrates that, once again, the SOS approach is the method of choice. It seems also important to mention that, unlike the SOS approach, the gridded approach in non-exact as we only consider $N_p$ points in the interval $[T_{min},T_{max}]$ instead of the whole interval.
\end{example}

\setlength\extrarowheight{1pt}
\begin{table}
\blue{\centering
    \caption{Comparison of different numerical results obtained for the maximum dwell-time using Theorem \ref{th:maxDT2} and Theorem \ref{th:rangeDT}. When using Theorem \ref{th:rangeDT}, we set $T_{min}=10^{-5}$ and $T_{max}=\bar T$. The gridded conditions are considered over $N_p=201$ gridding points. For each method, we give the estimate for the maxium dwell-time $\bar T$, the number of primal/dual variables of the optimization problem and the overall solving time in seconds. For fairness, the solver SeDuMi is used for solving all the optimization problems.}\label{table:maxDT1}
  \begin{tabular}{|c||c||c|c|c|c|}
  \hline
    & Result & Method & Computed $\bar T$  & No. vars. & Solving time\\
    \hline
    \hline
        \multirow{5}{*}{System \eqref{eq:maxDT1}}&  Th. \ref{th:maxDT2}, (c) & -- & 4.6051 & 6/2 & 0.3096\\
        \hline
    &  Th. \ref{th:maxDT2}, (f)& SOS ($d_s=1$) & 3.2724 & 37/15 & 0.3323\\
    &Th. \ref{th:maxDT2}, (f)& SOS ($d_s=2$) &  4.5610 &  65/19  & 0.3748\\
    &Th. \ref{th:maxDT2}, (f)& SOS ($d_s=3$) &  4.6023 &  101/23  & 0.3846 \\
    \cline{2-6}
    &constant case & $\rho(e^{AT}J)<1$ &  4.6051 & -- & --\\
    \hline
    \hline
        \multirow{5}{*}{System \eqref{eq:maxDT2}}& Th. \ref{th:rangeDT}, (b)& Gridded ($N_p=201$) & 0.2633 & 204/2 & 1.9869\\
        \hline
     &   Th. \ref{th:rangeDT}, (c)& SOS ($d_s=1$) & 0.2337 & 60/22 & 0.2049\\
    &Th. \ref{th:rangeDT}, (c)& SOS ($d_s=2$) & 0.2631 & 112/30  & 0.2645\\
    &Th. \ref{th:rangeDT}, (c)& SOS ($d_s=3$) & 0.2633 & 180/38  &0.3005 \\
    \cline{2-6}
    &constant case & $\rho(e^{AT}J)<1$ &(0,0.2633) & -- & --\\
    \hline
  \end{tabular}}
\end{table}

\begin{example}[Range dwell-time]
Let us consider the system \eqref{eq:mainsyst} with the matrices
\begin{equation}\label{eq:rangeDT1}
  A=\begin{bmatrix}
    -4 & 1\\ 2 & 1
  \end{bmatrix}\ \textnormal{and}\ J=\begin{bmatrix}
2 & 0\\ 1 & 0.1
  \end{bmatrix}.
\end{equation}
In this case, the matrices $A$ and $J$ are unstable but we can prove that when $T\in(0.2779, 0.6056)$ then the matrix $e^{AT}J$ is Schur stable. As in the previous example, we compare the SOS approach for solving the conditions of Theorem \ref{th:rangeDT}, (c) and a gridding approach for solving the condition of \ref{th:rangeDT}, (b) where we consider $N_p=201$ gridding points. While the methods have a comparable complexity in terms of the number of variables, the SOS method is again faster than the gridded approach. Note also that it is not possible to verify the accuracy of the estimate for the minimum dwell-time by comparing it with the values obtained for the constant dwell-time case. The maximum dwell-time is readily seen to be accurate in the present case.
\end{example}

\begin{table}
  \centering
    \caption{Comparison of the numerical results obtained for the range dwell-time using the sum of squares approach (Section \ref{sec:SOS}) and a naive and inaccurate gridding of the conditions of statement (b) of Theorem \ref{th:rangeDT} with $N_p=201$ gridding points. For each method, we give the estimate for the range dwell-time $(T_{min},T_{max})$, the number of primal/dual variables of the optimization problem and the overall solving time in seconds. SeDuMi is used for solving all the optimization problems.}\label{table:rangeDT1}
  \begin{tabular}{|c||c||c|c|c|c|}
  \hline
    & Result & Method & $(T_{min},T_{max})$  & No. vars. & Solving time\\
    \hline
    \hline
   %     \multirow{5}{*}{System \eqref{eq:maxDT2}}& Th. \ref{th:rangeDT}, (c)& SOS ($d_s=1$) & (0,0.2337) & 60/22 & 0.2049\\
%    &Th. \ref{th:rangeDT}, (c)& SOS ($d_s=2$) & (0, 0.2631) & 112/30  & 0.2645\\
%    &Th. \ref{th:rangeDT}, (c)& SOS ($d_s=3$) & (0, 0.2633) & 180/38  &0.3005 \\
%    &Th. \ref{th:rangeDT}, (b)& Gridded ($N_p=201$) & (0, 0.2633) & 204/2 & 1.9869\\
%    \cline{2-6}
%    &constant case & $\rho(e^{AT}J)<1$ &(0,0.2633) & -- & --\\
%    \hline
%    \hline
    \multirow{5}{*}{System \eqref{eq:rangeDT1}} &Th. \ref{th:rangeDT}, (b) & Gridded ($N_p=201$) & (0.3275, 0.6056) & 204/2 & 2.0025\\
    \hline
    &Th. \ref{th:rangeDT}, (c) & SOS ($d_s=1$) & infeasible & 60/22 & --\\
    &Th. \ref{th:rangeDT}, (c) &SOS ($d_s=2$) & (0.3339,0.5923) & 112/30 & 0.3347\\
    &Th. \ref{th:rangeDT}, (c) & SOS ($d_s=3$) & (0.3275, 0.6054) & 180/38& 0.4642\\
        \cline{2-6}
    & constant case & $\rho(e^{AT}J)<1$ & (0.2779, 0.6056) & -- & --\\
    \hline
  \end{tabular}
\end{table}

\section{Stabilization of positive linear impulsive systems}\label{sec:stabz}

As mentioned in the introduction, the rationale for introducing lifted conditions for characterizing the stability of linear positive impulsive systems under dwell-time constraints is to allow for the possibility of deriving similar conditions for uncertain systems and to extend them to control design; see e.g. \cite{Briat:13d,Briat:14f,Briat:15d}.  To this aim, let us consider in this section the system
\begin{equation}\label{eq:mainsystu}
\begin{array}{rcl}
    \dot{x}(t)&=&Ax(t)+B_cu_c(t),\ t\ne t_k\\
    x(t^+)&=&Jx(t)+B_du_d(t),\ t= t_k
\end{array}
\end{equation}
where $u_c\in\mathbb{R}^{m_c\times n}$ and $u_d\in\mathbb{R}^{m_d\times n}$ are the continuous and the discrete control inputs, respectively. In this section, no assumption is made on the system as only the positivity of the closed-loop system will matter; see e.g. \citep{Aitrami:07,Briat:11h}. As in Section \ref{sec:stab}, we will cover the cases of arbitrary dwell-time $(T_k\in\mathbb{R}_{>0}$ in Section \ref{sec:impz:arbitrary}, constant dwell-time ($T_k=\bar T$) in Section \ref{sec:impz:constant}, minimum dwell-time ($T_k\ge\bar T$) in Section \ref{sec:impz:minimum}, maximum dwell-time ($T_k\le\bar T$) in Section \ref{sec:impz:maximum} and range dwell-time ($T_k\in[T_{min},T_{max}]$) in Section \ref{sec:impz:range}. Illustrative examples are given in Section \ref{sec:impz:examples}.

\subsection{Stabilization under arbitrary dwell-time}\label{sec:impz:arbitrary}

For the arbitrary dwell-time case, we propose the following state-feedback control law
\begin{equation}\label{eq:sfcl0}
\begin{array}{rcl}
    u_c(t)&=&K_cx(t)\\
    u_d(t)&=&K_dx(t)
\end{array}
\end{equation}
where $K_c\in\mathbb{R}^{m_c\times n}$ and $K_d\in\mathbb{R}^{m_d\times n}$. \blue{The reason for considering such structure is that such controllers can be easily designed using the stability conditions for arbitrary dwell-times. Another motivation is that the control problem under arbitrary dwell-time can be interpreted as an infinite horizon control since we do not keep track in the control law of the next impulse time.} \bluee{By considering the stability conditions in Remark \ref{rem:arbDT2}, (a), we can readily obtain the following result:}
\begin{theorem}[Stabilization under arbitrary dwell-time]\label{th:arbDT_stabz}
Assume that there exist matrices $X\in\mathbb{D}_{\succ0}^n$, $U_c\in\mathbb{R}^{m_c\times n}$ and $U_d\in\mathbb{R}^{m_d\times n}$ and a scalar $\alpha\in\mathbb{R}$ such that the conditions
\begin{equation}\label{eq:arbDT_stabz:1}
  AX+B_cU_c+\alpha I_n\ge0,\ JX+B_dU_d\ge0,
\end{equation}
and
\begin{equation}\label{eq:arbDT_stabz:2}
[AX+B_cU_c]\mathds{1}_n<0\ \textnormal{and}\ [JX+B_dU_d-X]\mathds{1}_n<0
\end{equation}
hold. Then, there exists a controller of the form \eqref{eq:sfcl0} such that the closed-loop system \eqref{eq:mainsystu}-\eqref{eq:sfcl0} is positive and asymptotically stable for arbitrary dwell-time and suitable controller matrices are given by $K_c=U_cX^{-1}$ and $K_d=U_dX^{-1}$.
\end{theorem}
\begin{proof}
  The closed-loop system \eqref{eq:mainsystu}-\eqref{eq:sfcl0} is given by
  \begin{equation}
\begin{array}{rcl}
    \dot{x}(t)&=&(A+B_cK_c)x(t),\ t\ne t_k,\\
    x(t^+)&=&(J+B_dK_d)x(t),\ t=t_k.
\end{array}
\end{equation}
By substituting the matrices of the closed-loop system in the conditions of Remark \ref{rem:arbDT2}, (a), yields the conditions
\begin{equation}
  (A+B_cK_c)\lambda<0\ \textnormal{and}\ \blue{(J+B_dK_d-I_n)\lambda<0.}
\end{equation}
Defining $X\in\mathbb{D}^n_{\succ0}$ as $\lambda=:X\mathds{1}_n$, we obtain the conditions in \eqref{eq:arbDT_stabz:2} where we have used the changes of variables $U_c=K_cX$ and $U_d=K_dX$. The conditions in \eqref{eq:arbDT_stabz:1} are readily seen to be equivalent to saying that $A+B_cK_c$ is Metzler and $J+B_dK_d$ is nonnegative. The proof is complete.
\end{proof}

\subsection{Stabilization under constant dwell-time}\label{sec:impz:constant}

For the constant dwell-time case, we propose the following state-feedback control law \cite{Briat:13d,Briat:15i}:
\begin{equation}\label{eq:sfcl1}
\begin{array}{rcl}
    u_c(t_k+\tau)&=&K_c(\tau)x(t_k+\tau),\ \tau\in(0,\bar T]\\
    u_d(t)&=&K_dx(t)\\
\end{array}
\end{equation}
where $K_c:[0,\bar T]\mapsto\mathbb{R}^{m_c\times n}$ and $K_d\in\mathbb{R}^{m_d\times n}$. \blue{Unlike in the arbitrary dwell-time case, the proposed controller does depend on the clock which measures the time elapsed since the last impulse time and anticipates over the next impulse time. In this regard, this controller can be interpreted as a finite-time stability controller over one dwell-time interval. Stabilization is then ensured by repeating this procedure over all dwell-time intervals. This controller structure is reminiscent to the problems of  designing optimal finite-horizon LQ controllers \cite{Sontag:98} whose solutions take the form of a time-varying state-feedback control law obtained from the solution of a differential Riccati equation. The difference here is that the time is replaced by the clock.} \bluee{The following result can be understood as being the ``positive systems" version of the results in \cite{Briat:11l,Briat:12h}:}
\begin{theorem}[Stabilization under constant dwell-time]\label{th:cstDT_stabz}
  The following statements are equivalent:
  \begin{enumerate}[(a)]
  \item There exists a controller of the form \eqref{eq:sfcl1} such that the closed-loop system \eqref{eq:mainsystu}-\eqref{eq:sfcl1} is positive and asymptotically stable under constant dwell-time $\bar T$.
    \item There exist a vector $\lambda\in\mathbb{R}_{>0}^n$, a matrix-valued function $K_c:[0,\bar T]\mapsto\mathbb{R}^{m_c\times n}$ and a matrix $K_d\in\mathbb{R}^{m_d\times n}$ such that the matrix $A+B_cK_c(\tau)$ is Metzler for all $\tau\in[0,\bar{T}]$, the matrix $J+B_dK_d$ is nonnegative and such that the inequality
    \begin{equation}\label{eq:constantDTZ0}
      [(J+B_dK_d)\Psi(\bar{T})-I_n]\lambda<0
    \end{equation}
    holds with
    \begin{equation}
      \dfrac{d\Psi(s)}{ds}=\left(A+B_cK_c(s)\right)\Psi(s),\ \Psi(0)=I,\ s\ge0.
\end{equation}
 \item There exist a matrix-valued function ${X:[0,\bar{T}]\mapsto\mathbb{D}^n}$, $X(0)\in\mathbb{D}_{\succ0}^n$, ${U:[0,\bar{T}]\mapsto\mathbb{R}^{m_c\times n}}$ and scalars $\eps,\alpha>0$ such that the inequalities
         \begin{equation}\label{eq:constantDTZ11}
          AX(\tau)+B_cU(\tau)+\alpha I\ge0,\ JX(\bar T)+B_dU_d\ge0,
         \end{equation}
    \begin{equation}\label{eq:constantDTZ12}
      \left[-\dot{X}(\tau)+AX(\tau)+B_cU_c(\tau)\right]\mathds{1}_n<0
    \end{equation}
  and
  \begin{equation}\label{eq:constantDTZ13}
    \left[JX(\bar T)+B_dU_d-X(0)+\eps I\right]\mathds{1}_n\le0
  \end{equation}
 hold for all $\tau\in[0,\bar{T}]$. Moreover, in such a case, suitable controller gains can be computed using the relations  $K(\tau)=U(\tau)X(\tau)^{-1}$ and $K_d=U_dX(\bar T)^{-1}$.
  \item There exist a matrix-valued function ${X:[0,\bar{T}]\mapsto\mathbb{D}^n}$, $X(\bar T)\in\mathbb{D}_{\succ0}^n$, ${U:[0,\bar{T}]\mapsto\mathbb{R}^{m_c\times n}}$ and scalars $\eps,\alpha>0$ such that the inequalities
         \begin{equation}\label{eq:constantDTZ21}
                  AX(\tau)+B_cU(\tau)+\alpha I\ge0,\ JX(0)+B_dU_d\ge0,
         \end{equation}
    \begin{equation}\label{eq:constantDTZ22}
      \left[\dot{X}(\tau)+AX(\tau)+B_cU_c(\tau)\right]\mathds{1}_n<0
    \end{equation}
  and
  \begin{equation}\label{eq:constantDTZ23}
    \left[JX(0)+B_dU_d-X(\bar T)+\eps I\right]\mathds{1}_n\le0
  \end{equation}
 hold for all $\tau\in[0,\bar{T}]$. Moreover, in such a case, suitable controller gains can be computed using the relations  $K(\tau)=U(\tau)X(\tau)^{-1}$ and $K_d=U_dX(0)^{-1}$.
   \end{enumerate}
\end{theorem}
\begin{proof}
The equivalence between statement (a) and statement (b) comes from Proposition \ref{prop:Hurwitz} and the fact that the state-transition matrix $(J+B_dK_d)\Psi(\bar T)$ is constant. We prove now that (c) implies (b). To prove this, it seems important to stress first that the proof of Theorem \ref{th:cstDT2} can be straightforwardly extended to time-inhomogeneous systems using state-transition matrices; see e.g. \cite{Briat:15c,Briat:15f}. First of all, the conditions in \eqref{eq:constantDTZ11} are equivalent to saying that the matrix $A+B_cK_c(\tau)$ is Metzler for all $\tau\in[0,\bar T]$ and that the matrix $J+B_dK_d$ is nonnegative. Let us consider now the conditions \eqref{eq:constantDTZ12} and \eqref{eq:constantDTZ13} where we use the changes of variables $\zeta(\tau)=X(\tau)\mathds{1}_n$, $K(\tau)=U(\tau)X(\tau)^{-1}$ and $K_d=U_dX(\bar T)^{-1}$, to get the following equivalent conditions:
\begin{equation}
  \begin{array}{l}
      -\dot{\zeta}(\tau)+(A+B_cK_c(\tau))\zeta(\tau)<0\\
    (J+B_dK_d)\zeta(\bar T)-\zeta(0)+\eps \mathds{1}_n\le0.
  \end{array}
\end{equation}
We can recognize above the constant dwell-time stability conditions stated in Remark \ref{rem:swapped_constant} applied to the dual of the closed-loop system \eqref{eq:mainsystu}-\eqref{eq:sfcl1}. The rest of the proof uses the same argument as in the proof of  Theorem \ref{th:cstDT2}. The reverse implication (i.e. (b) implies (c)) and the equivalence between the statements (c) and (d) also follow from the same arguments as in the proof of Theorem \ref{th:cstDT2}.
\end{proof}

\subsection{Stabilization under minimum dwell-time}\label{sec:impz:minimum}

For the minimum dwell-time case, we propose the following state-feedback control law \cite{Briat:13d,Briat:15i}:
\begin{equation}\label{eq:sfcl2}
\begin{array}{rcl}
    u_c(t_k+\tau)&=&\left\{\begin{array}{lcl}
    K_c(\tau)x(t_k+\tau)&&\text{if\ }\tau\in[0,\bar{T})\\
   K_c(\bar{T})x(t_k+\tau)&&\text{if\ }\tau\in[\bar{T},T_k).
  \end{array}\right.\\
    u_d(t)&=&K_dx(t)\\
\end{array}
\end{equation}
where $K_c:[0,\bar T]\mapsto\mathbb{R}^{m_c\times n}$ and $K_d\in\mathbb{R}^{m_d\times n}$. \blue{The rationale behind the use of this controller is the following. We consider a clock-dependent controller that depends on the time elapsed since the last impulse as in the constant dwell-time case. The only difference is that the matrix gain becomes constant when the value of the clock exceeds the value of the minimum dwell-time. The reason for locking the value of the controller gain is twofold. The first reason is for convenience since the stability conditions can be used in a non-conservative way for designing such controllers. The second one is to avoid implementation difficulties by enforcing the controller gain to be bounded as $\tau$ increase towards infinity.} \bluee{The following result can be understood as being the ``positive systems" version of the results in \cite{Briat:11l,Briat:12h}:}

\begin{theorem}[Stabilization under minimum dwell-Time]\label{th:minDT_stabz2}
 There exists a controller of the form \eqref{eq:sfcl2} such that the closed-loop system \eqref{eq:mainsystu}-\eqref{eq:sfcl2} is positive and asymptotically stable under minimum dwell-time $\bar T$ if one of the following equivalent statements holds:
  \begin{enumerate}[(a)]
    \item There exist a vector $\lambda\in\mathbb{R}_{>0}^n$, a matrix-valued function $K_c:[0,\bar T]\mapsto\mathbb{R}^{m_c\times n}$ and a matrix $K_d\in\mathbb{R}^{m_d\times n}$ such that the matrix $A+B_cK_c(\tau)$ is Metzler for all $\tau\in[0,\bar{T}]$, the matrix $J+B_dK_d$ is nonnegative and such that the inequalities
    \begin{equation}\label{eq:minDTza1}
      (A+B_cK_c(\bar{T}))\lambda<0
    \end{equation}
    and
    \begin{equation}\label{eq:minDTza2}
     \left[\Psi(\bar{T})(J+B_dK_d)-I_n\right]\lambda<0
    \end{equation}
    hold where
    \begin{equation}\label{eq:evol2}
      \dfrac{d\Psi(s)}{ds}=\left(A+B_cK_c(s)\right)\Psi(s),\ \Psi(0)=I,\ s\ge0.
\end{equation}
 \item There exist a matrix-valued function ${X:[0,\bar{T}]\mapsto\mathbb{D}^n}$, $X(\bar T)\in\mathbb{D}_{\succ0}^n$, ${U:[0,\bar{T}]\mapsto\mathbb{R}^{m_c\times n}}$ and scalars $\eps,\alpha>0$ such that the inequalities
         \begin{equation}\label{eq:dksqldksqmdkdkmdskqkm6}
                AX(\tau)+B_cU(\tau)+\alpha I\ge0,\ JX(0)+B_dU_d\ge0,
         \end{equation}
          \begin{equation}\label{eq:dksqldksqmdkdkmdskqkm62}
                  [AX(\bar T)+B_cU(\bar T)]\mathds{1}_n<0,
         \end{equation}
    \begin{equation}\label{eq:minCDTzc1}
      \left[-\dot{X}(\tau)+AX(\tau)+B_cU_c(\tau)\right]\mathds{1}_n<0
    \end{equation}
  and
  \begin{equation}\label{eq:minCDTzc3}
    \left[JX(\bar T)+B_dU_d-X(0)+\eps I\right]\mathds{1}_n\le0
  \end{equation}
 hold for all $\tau\in[0,\bar{T}]$. Moreover, in such a case, suitable controller gains can be computed using the relations  $K(\tau)=U(\tau)X(\tau)^{-1}$ and $K_d=U_dX(\bar T)^{-1}$.
   \end{enumerate}
\end{theorem}
\blue{\begin{proof}
As for the constant dwell-time case, the conditions in \eqref{eq:dksqldksqmdkdkmdskqkm6} ensure that the matrix $A+BK_c(\tau)$ is Metzler for all $\tau\in[0,\bar T]$ and that the matrix $J+B_dK_d$ is nonnegative. Considering now the change of variables $\zeta(\tau)=X(\tau)\mathds{1}_n$, $K(\tau)=U(\tau)X(\tau)^{-1}$ and $K_d=U_dX(\bar T)^{-1}$, we get that the conditions \eqref{eq:dksqldksqmdkdkmdskqkm62}, \eqref{eq:minCDTzc1} and \eqref{eq:minCDTzc3} are equivalent to
\begin{equation}
  \begin{array}{l}
   (A+B_cK_c(\bar T))\zeta(\bar T)<0\\
      -\dot{\zeta}(\tau)+(A+B_cK_c(\tau))\zeta(\tau)<0,\ \tau\in[0,\bar T]\\
    (J+B_dK_d)\zeta(\bar T)-\zeta(0)+\eps \mathds{1}_n\le0.
  \end{array}
\end{equation}
We can recognize above the minimum dwell-time stability conditions of Theorem \ref{th:minDT2}, (f), applied to the dual of the closed-loop system \eqref{eq:mainsystu}-\eqref{eq:sfcl2}. The rest of the proof follows from the same arguments as in the proof of Theorem \ref{th:cstDT2}.
\end{proof}}

%\blue{\begin{remark}
%\textbf{ADD REMARK THAT THE CONSTANT DT CONTROLLER CAN BE RETRIEVED USING T go to 0}
%\end{remark}}

\subsection{Stabilization under maximum dwell-time}\label{sec:impz:maximum}

For the maximum dwell-time case, we propose the following state-feedback control law \cite{Briat:13d,Briat:15i}:
\begin{equation}\label{eq:sfcl3}
\begin{array}{rcl}
    u_c(t_k+\tau)&=&K_c(\tau)x(t_k+\tau),\ \tau\in(0,T_k]\\
    u_d(t)&=&K_dx(t)\\
\end{array}
\end{equation}
where $K_c:[0,\bar T]\mapsto\mathbb{R}^{m_c\times n}$ and $K_d\in\mathbb{R}^{m_d\times n}$. \blue{As for the constant dwell-time case, the controller matrix gain depends on the clock in order to exploit the stability conditions in a convenient and efficient way.} \bluee{The following result can be understood as being the ``positive systems" version of the results in \cite{Briat:11l,Briat:12h} and is the maximum dwell-time counterpart of Theorem \ref{th:minDT_stabz2}:}

\begin{theorem}[Stabilization under maximum dwell-time]\label{th:maxDT_stabz}
 There exists a controller of the form \eqref{eq:sfcl2} such that the closed-loop system \eqref{eq:mainsystu}-\eqref{eq:sfcl2} is positive and asymptotically stable under maximum dwell-time $\bar T$ if one of the following equivalent statements holds:
  \begin{enumerate}[(a)]
    \item There exist a vector $\lambda\in\mathbb{R}_{>0}^n$, a matrix-valued function $K_c:[0,\bar T]\mapsto\mathbb{R}^{m_c\times n}$ and a matrix $K_d\in\mathbb{R}^{m_d\times n}$ such that the matrix $A+B_cK_c(\tau)$ is Metzler for all $\tau\in[0,\bar{T}]$, the matrix $J+B_dK_d$ is nonnegative and such that the inequalities
    \begin{equation}
      (A+B_cK_c(\bar{T}))\lambda>0
    \end{equation}
    and
    \begin{equation}
     \left[(J+B_dK_d)\Psi(\bar{T})-I_n\right]\lambda<0
    \end{equation}
    hold where
    \begin{equation}
      \dfrac{d\Psi(s)}{ds}=\left(A+B_cK_c(s)\right)\Psi(s),\ \Psi(0)=I,\ s\ge0.
\end{equation}
 \item There exist a matrix-valued function ${X:[0,\bar{T}]\mapsto\mathbb{D}^n}$, $X(\bar T)\in\mathbb{D}_{\succ0}^n$, ${U:[0,\bar{T}]\mapsto\mathbb{R}^{m_c\times n}}$ and scalars $\eps,\alpha>0$ such that the inequalities
         \begin{equation}\label{eq:kdsmkdsmlk1}
                    AX(\tau)+B_cU(\tau)+\alpha I\ge0,\ JX(0)+B_dU_d\ge0,
         \end{equation}
          \begin{equation}\label{eq:kdsmkdsmlk2}
                  [AX(\bar T)+B_cU(\bar T)]\mathds{1}_n>0,
         \end{equation}
    \begin{equation}\label{eq:kdsmkdsmlk3}
      \left[\dot{X}(\tau)+AX(\tau)+B_cU_c(\tau)\right]\mathds{1}_n<0
    \end{equation}
  and
  \begin{equation}\label{eq:kdsmkdsmlk4}
    \left[JX(0)+B_dU_d-X(\bar{T})+\eps I\right]\mathds{1}_n\le0
  \end{equation}
 hold for all $\tau\in[0,\bar{T}]$. Moreover, in such a case, suitable controller gains can be computed using the relations  $K(\tau)=U(\tau)X(\tau)^{-1}$ and $K_d=U_dX(0)^{-1}$.
   \end{enumerate}
\end{theorem}
\begin{proof}
Similarly to as in the constant and minimum dwell-time case, the conditions  in \eqref{eq:kdsmkdsmlk1} ensure that the matrix $A+BK_c(\tau)$ is Metzler for all $\tau\in[0,\bar T]$ and that the matrix $J+B_dK_d$ is nonnegative. Considering now the change of variables $\zeta(\tau)=X(\tau)\mathds{1}_n$, $K(\tau)=U(\tau)X(\tau)^{-1}$ and $K_d=U_dX(0)^{-1}$, we get that the conditions \eqref{eq:kdsmkdsmlk2}, \eqref{eq:kdsmkdsmlk3} and \eqref{eq:kdsmkdsmlk4} are equivalent to
\begin{equation}
  \begin{array}{l}
    (A+B_cK_c(\bar T))\zeta(\bar T)<0\\
      \dot{\zeta}(\tau)+(A+B_cK_c(\tau))\zeta(\tau)<0,\ \tau\in[0,\bar T]\\
    (J+B_dK_d)\zeta(0)-\zeta(\bar T)+\eps \mathds{1}_n\le0.
  \end{array}
\end{equation}
We can recognize above the maximum dwell-time stability condition stated in Remark \ref{rem:swapped_maximum} applied to the dual of the closed-loop system \eqref{eq:mainsystu}-\eqref{eq:sfcl3}. The rest of the proof follows from the same arguments as in the proof of Theorem \ref{th:cstDT2}.
\end{proof}

\subsection{Stabilization under range dwell-time}\label{sec:impz:range}

For the range dwell-time case, we also consider the state-feedback control law of the form \eqref{eq:sfcl3}. \blue{The following result can be understood as being the ``positive systems" version of the results in \cite{Briat:11l,Briat:12h}:}
%
%\begin{theorem}[Range dwell-time - Robust case \#1]\label{th:rangeDT_stabz}
\begin{theorem}[Range dwell-time]\label{th:rangeDT_stabz}
   There exists a controller of the form \eqref{eq:sfcl2} such that the closed-loop system \eqref{eq:mainsystu}-\eqref{eq:sfcl2} is positive and asymptotically stable under range dwell-time $(T_{min},T_{max})$ if one of the following equivalent statements holds:
  \begin{enumerate}[(a)]
    \item There exist a vector $\lambda\in\mathbb{R}_{>0}^n$, a matrix-valued function $K_c:[0,\bar T]\mapsto\mathbb{R}^{m_c\times n}$ and a matrix $K_d\in\mathbb{R}^{m_d\times n}$ such that the matrix $A+B_cK_c(\tau)$ is Metzler for all $\tau\in[0,\bar{T}]$, the matrix $J+B_dK_d$ is nonnegative and such that the inequality
    \begin{equation}
     \left[\Psi(\theta)(J+B_dK_d)-I_n\right]\lambda<0
    \end{equation}
    holds where
    \begin{equation}
      \dfrac{d\Psi(s)}{ds}=\left(A+B_cK_c(s)\right)\Psi(s),\ \Psi(0)=I,\ s\ge0.
\end{equation}
     \item There exist a matrix-valued function ${X:[0,T_{max}]\mapsto\mathbb{D}^n}$, $X(0)\in\mathbb{D}_{\succ0}^n$, ${U:[0,\bar{T}]\mapsto\mathbb{R}^{m_c\times n}}$ and scalars $\eps,\alpha>0$ such that the inequalities
         \begin{equation}\label{eq:kdsmkdsmlkk1}
                    AX(\tau)+B_cU(\tau)+\alpha I\ge0,\ JX(0)+B_dU_d\ge0,
         \end{equation}
    \begin{equation}\label{eq:kdsmkdsmlkk2}
      \left[\dot{X}(\tau)+AX(\tau)+B_cU_c(\tau)\right]\mathds{1}_n<0
    \end{equation}
  and
  \begin{equation}\label{eq:kdsmkdsmlkk3}
    \left[JX(0)+B_dU_d-X(\theta)+\eps I\right]\mathds{1}_n\le0
  \end{equation}
 hold for all $\tau\in[0,T_{max}]$ and $\theta\in[T_{min},T_{max}]$. Moreover, in such a case, suitable controller gains can be computed using the relations  $K(\tau)=U(\tau)X(\tau)^{-1}$ and $K_d=U_dX(0)^{-1}$.
 \end{enumerate}
\end{theorem}
\begin{proof}
Similarly to as in the constant, minimum and maximum dwell-time case, the conditions  in \eqref{eq:kdsmkdsmlkk1} ensure that the matrix $A+BK_c(\tau)$ is Metzler for all $\tau\in[0,\bar T]$ and that the matrix $J+B_dK_d$ is nonnegative. Considering now the change of variables $\zeta(\tau)=X(\tau)\mathds{1}_n$, $K(\tau)=U(\tau)X(\tau)^{-1}$ and $K_d=U_dX(0)^{-1}$, we get that the conditions \eqref{eq:kdsmkdsmlkk2} and \eqref{eq:kdsmkdsmlkk3} are equivalent to
\begin{equation}
  \begin{array}{l}
      \dot{\zeta}(\tau)+(A+B_cK_c(\tau))\zeta(\tau)<0,\ \tau\in[0,\bar T]\\
    (J+B_dK_d)\zeta(0)-\zeta(\bar T)+\eps \mathds{1}_n\le0.
  \end{array}
\end{equation}
We can recognize above the range dwell-time stability condition stated in Remark \ref{rem:swapped_range} applied to the dual of the closed-loop system \eqref{eq:mainsystu}-\eqref{eq:sfcl2}. The rest of the proof follows from the same arguments as in the proof of Theorem \ref{th:cstDT2}.
\end{proof}

\begin{remark}
 A dwell-time-scheduled control law of the form
\begin{equation}\label{eq:sfcl3b}
\begin{array}{rcl}
    u_c(t_k+\tau)&=&K_c(\tau)x(t_k+\tau),\ \tau\in(0,T_k]\\
    u_d(t)&=&K_d(T_k)x(t),\ t=t_k\\
\end{array}
\end{equation}
where $K_c:[0,\bar T]\mapsto\mathbb{R}^{m_c\times n}$ and $K_d:[T_{min},T_{max}]\mapsto\mathbb{R}^{m_d\times n}$ can be easily designed using the above result by simply substituting the matrix $U_d\in\mathbb{R}^{m_d\times n}$ by a matrix-valued function $U_d:[T_{min},T_{max}]\mapsto\mathbb{R}^{m_d\times n}$. Then, the controller gain can be simply computed using $K_d(\theta)=U_d(\theta)X(0)^{-1}$.
\end{remark}

\subsection{Examples}\label{sec:impz:examples}

We illustrate in this section the efficiency of the proposed result on several stabilization problems.
\begin{example}[Stabilization under minimum dwell-time]
Let us consider the impulsive system \eqref{eq:mainsystu} with matrices
\begin{equation}\label{eq:stabz:minDT1}
  A=\begin{bmatrix}
    3 & -1\\
    2 & -1
  \end{bmatrix},\ B_c=\begin{bmatrix}
    1\\0
  \end{bmatrix},\ J = \begin{bmatrix}
    2 & 1\\ 0 & 0.7
  \end{bmatrix},\ \blue{B_d=\begin{bmatrix}
    1\\
    0
  \end{bmatrix}}.
\end{equation}
\blue{The uncontrolled version of the above system cannot be stable under minimum dwell-time since the matrix $A$ is not Hurwitz stable. Hence, we propose to compute a control law of the form \eqref{eq:sfcl2} that makes the closed-loop system positive and asymptotically stable under minimum dwell-time $\bar T=0.1$. To this aim, we apply the SOS method of Section \ref{sec:SOS} to the conditions of Theorem \ref{th:minDT_stabz2}, (b), with polynomial degree $d_s=1$ and we get the following controller matrices
\begin{equation}\label{eq:stabz:minDT1_cl}
  \begin{array}{rcl}
  K_d&=&\begin{bmatrix}
   -1.7215 &  -0.8685
  \end{bmatrix},\vspace{2mm}\\
K_c(\tau)&=&\begin{bmatrix}
    \dfrac{-1.3523\tau^2    +0.2862\tau   -0.0660}{0.2513\tau^2   -0.0857\tau+    0.1383} &    \dfrac{-1.8806\tau^2   -0.3758\tau   -0.2180}{-1.4274\tau^2   -0.2895\tau+   0.4749}
  \end{bmatrix}.
  \end{array}
\end{equation}
To validate the design, we perform a simulation with randomly generated impulse times satisfying the minimum dwell-time condition and we obtain the results depicted in Figure~\ref{fig:stabz_minDT} where we can see the controller efficiently stabilizes the system and drives the state of the closed-loop system to zero.
\begin{figure}[H]
  \centering
  \includegraphics[width=0.45\textwidth]{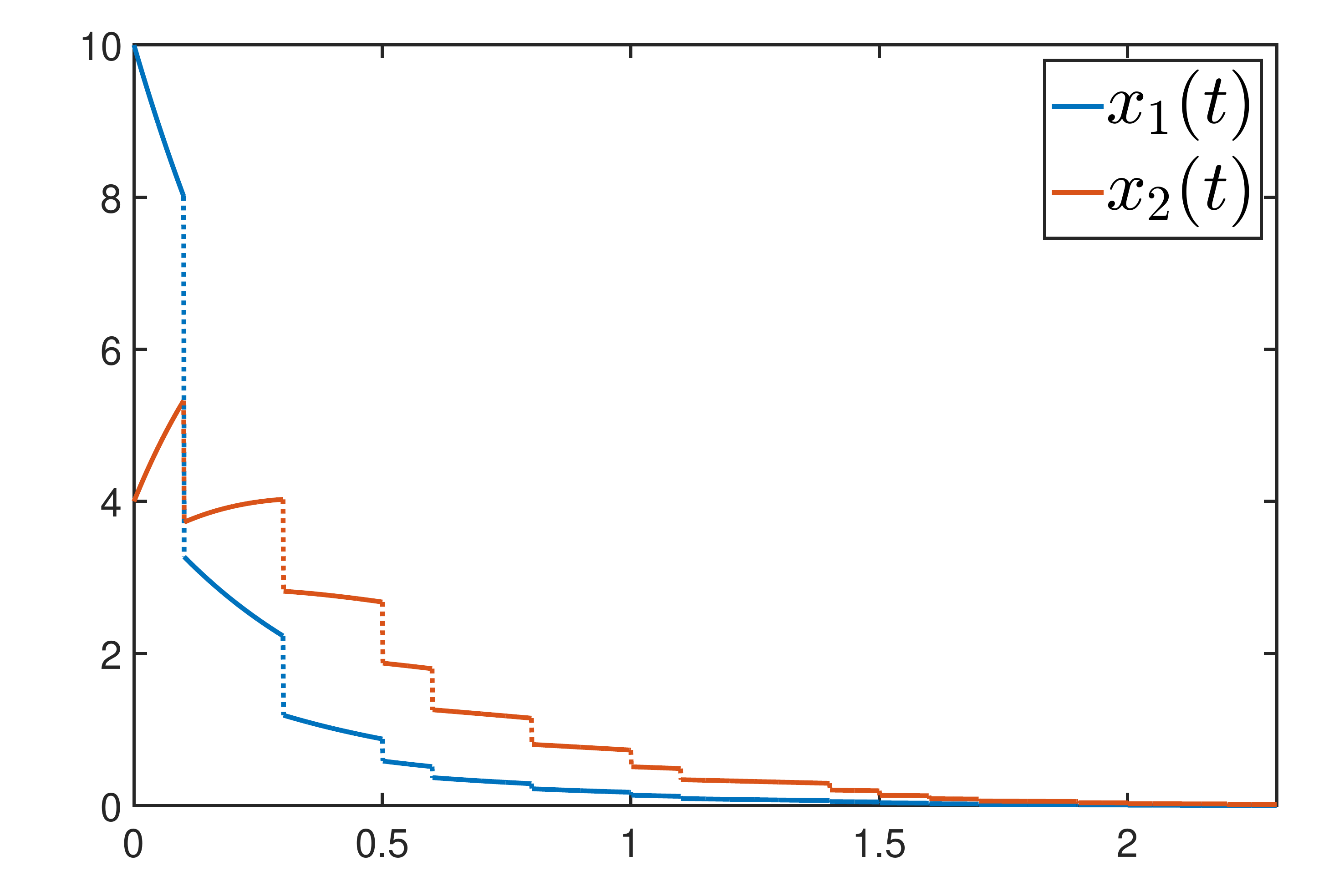}\hfill  \includegraphics[width=0.47\textwidth]{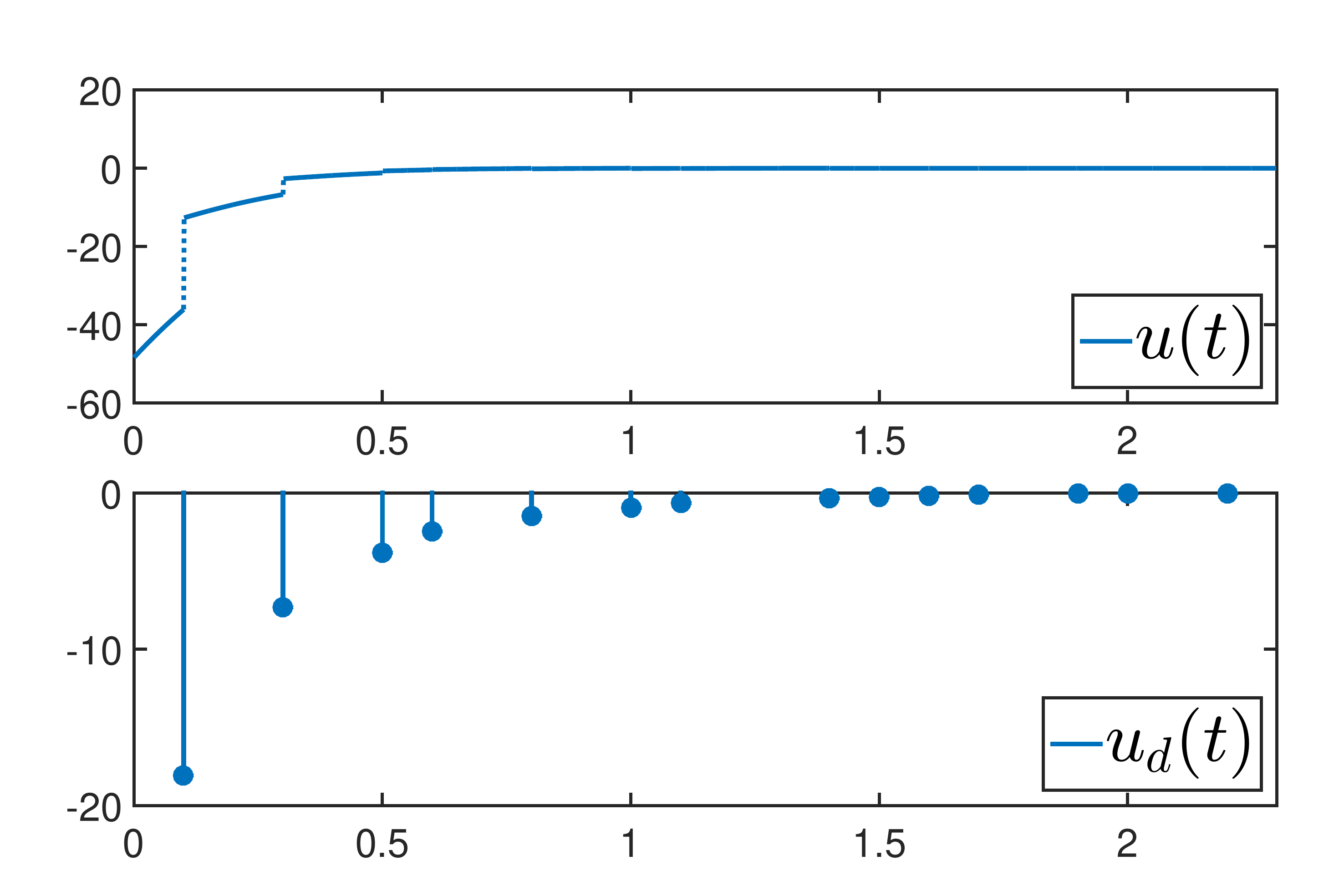}
  \caption{Evolution of the states (left) of closed-loop system \eqref{eq:mainsystu}-\eqref{eq:stabz:minDT1}-\eqref{eq:sfcl2}-\eqref{eq:stabz:minDT1_cl} and the associated control inputs (right).}\label{fig:stabz_minDT}
\end{figure}}
\end{example}

\blue{\begin{example}[Stabilization under maximum dwell-time]
Let us consider the impulsive system \eqref{eq:mainsystu} with matrices
\begin{equation}\label{eq:stabz:maxDT1}
  A=\begin{bmatrix}
    3 & 2\\
    1 & 1
  \end{bmatrix},\ B_c=\begin{bmatrix}
    0\\0
  \end{bmatrix},\ J = \begin{bmatrix}
    2 & 1\\ 0 & 0.7
  \end{bmatrix},\ B_d=\begin{bmatrix}
    1\\
    0
  \end{bmatrix}.
\end{equation}
When $u\equiv0$ and $u_d\equiv0$, the above system is not stable under maximum dwell-time since the matrix $J$ is not Schur stable. Hence, we propose to find a control law of the form \eqref{eq:sfcl1} that makes the closed-loop system positive and asymptotically stable with maximum dwell-time $\bar T=0.1$. We apply the SOS method of Section \ref{sec:SOS} to the conditions of Theorem \ref{th:maxDT_stabz}, (b), with polynomial degree $d_s=1$ and we get the following controller matrices
\begin{equation}\label{eq:stabz:maxDT1_cl}
  \begin{array}{rcl}
  K_d&=&\begin{bmatrix}
    -1.7080 &   -0.7511
  \end{bmatrix},\vspace{2mm}\\
K_c(\tau)&=&\begin{bmatrix}
    \dfrac{-2.2491 \tau^2+   1.4879 \tau  -0.8083}{ 0.6314\tau^2   -0.6576 \tau+   0.5351} &    \dfrac{-0.7639\tau^2 +   2.4375\tau   -0.9827}{ 0.7345 \tau^2  -1.5943\tau+    0.9738}
  \end{bmatrix}.
  \end{array}
\end{equation}
The design is validated based on a simulation where randomly generated impulse times satisfying the minimum dwell-time condition are considered. The obtained results are depicted in Figure~\ref{fig:stabz_maxDT} where we can see the controller efficiently stabilizes the system and drives the state of the closed-loop system to zero.
\begin{figure}[H]
  \centering
  \includegraphics[width=0.45\textwidth]{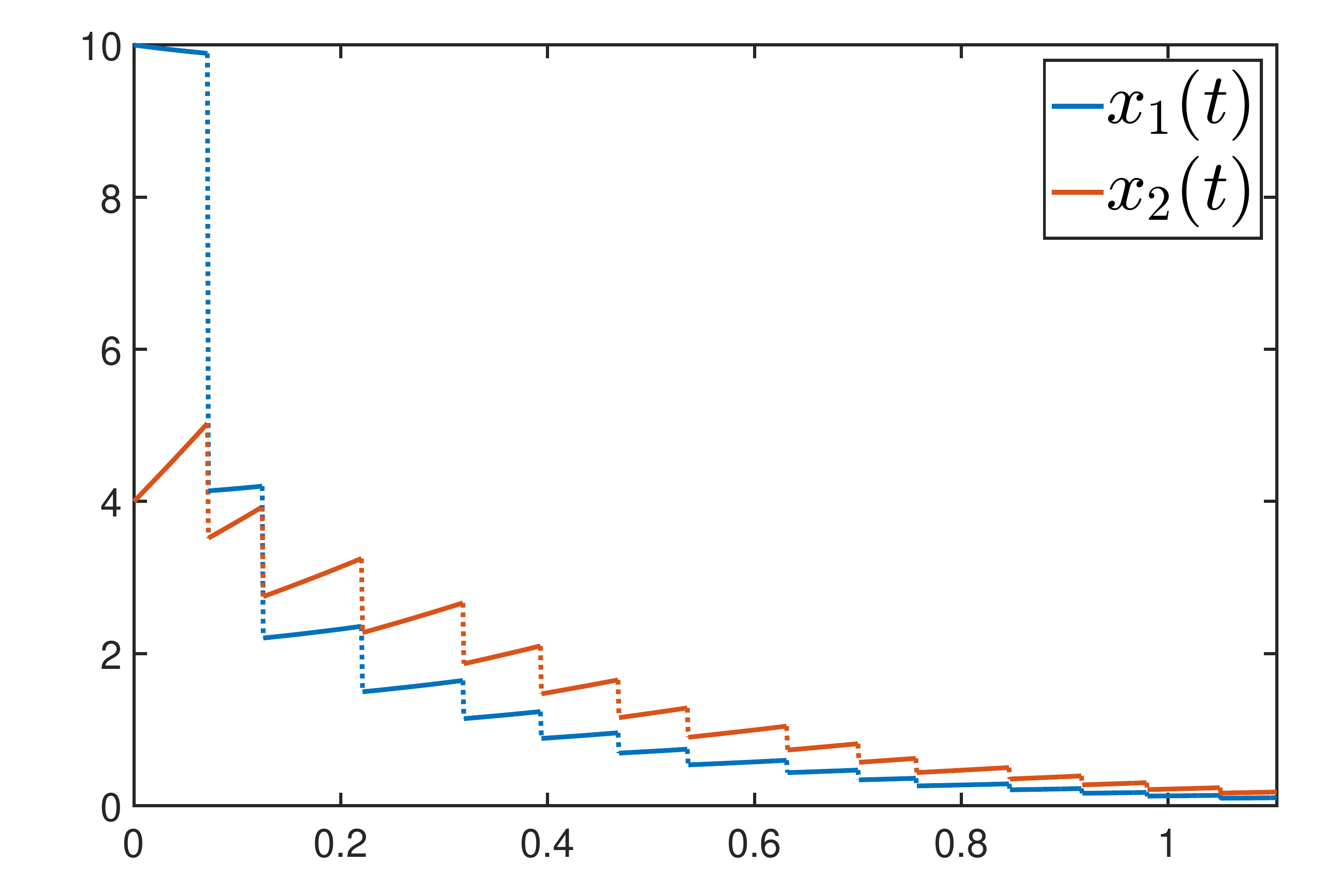}\hfill  \includegraphics[width=0.47\textwidth]{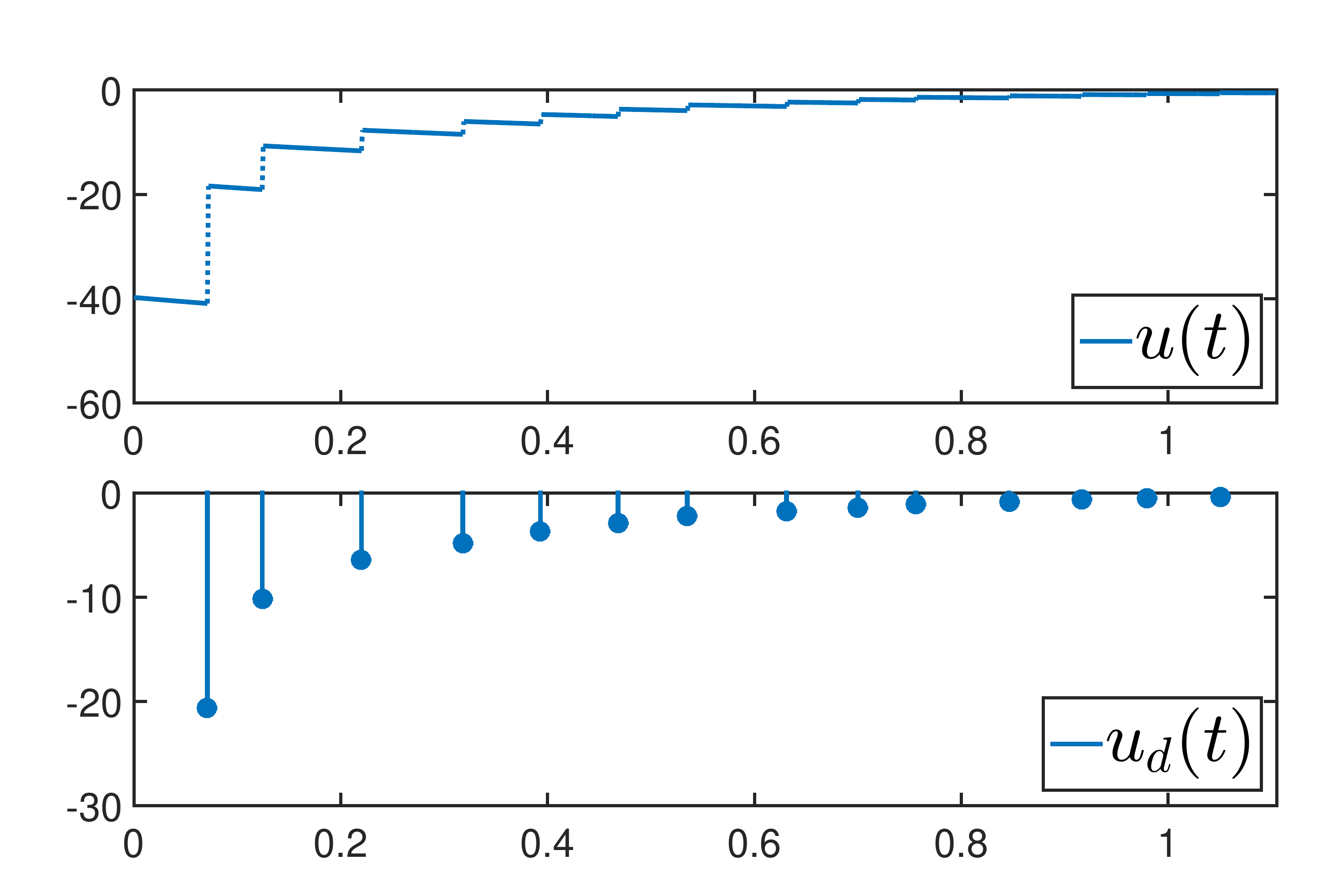}
  \caption{Evolution of the states (left) of closed-loop system \eqref{eq:mainsystu}-\eqref{eq:stabz:maxDT1}-\eqref{eq:sfcl1}-\eqref{eq:stabz:maxDT1_cl} and the associated control inputs (right).}\label{fig:stabz_maxDT}
\end{figure}
\end{example}}

\begin{example}[Stabilization under range dwell-time]
Let us consider the impulsive system \eqref{eq:mainsystu} with matrices
\begin{equation}\label{eq:stabz:rangeDT1_cl}
  A=\begin{bmatrix}
    1 & 1\\
    0 & 1
  \end{bmatrix},\ B_c=\begin{bmatrix}
    1\\0
  \end{bmatrix},\ J = \begin{bmatrix}
    1.2 & 0\\ 1 & 0.1
  \end{bmatrix},\ B_d=\begin{bmatrix}
    0\\
    1
  \end{bmatrix}.
\end{equation}
This system is unstable for any impulse sequence $\{t_k\}_{k\in\mathbb{N}}$ and hence the goal is to find a control law of the form \eqref{eq:sfcl3} that makes the closed-loop system positive and asymptotically stable with range dwell-time $(T_{min},T_{max})=(0.1,0.3)$. We apply the SOS method of Section \ref{sec:SOS} to the conditions of Theorem \ref{th:rangeDT_stabz}, (b), with polynomial degree $d_s=1$ and we get the following controller matrices
\begin{equation}\label{eq:stabz:rangeDT1_cl}
\begin{array}{rcl}
  K_d&=&\begin{bmatrix}
     0.1693   & 0.1006
  \end{bmatrix},\vspace{2mm}\\
K_c(\tau)&=&\begin{bmatrix}
    \dfrac{ 1.6045\tau^2+    0.5235\tau   -0.7268}{-0.1609\tau^2   -0.5823\tau+    0.1522} &    \dfrac{-0.5299\tau^2+    0.9808\tau+    0.1312}{0.9635\tau^2   -1.2439\tau+    0.5688}
  \end{bmatrix}.
\end{array}
\end{equation}
To validate the design, we perform a simulation with randomly generated impulse times satisfying the minimum dwell-time condition and we obtain the results depicted in Figure~\ref{fig:stabz_rangeDT} where we can see the controller efficiently stabilizes the system and drives the state of the closed-loop system to zero.

\begin{figure}[H]
  \centering
  \includegraphics[width=0.45\textwidth]{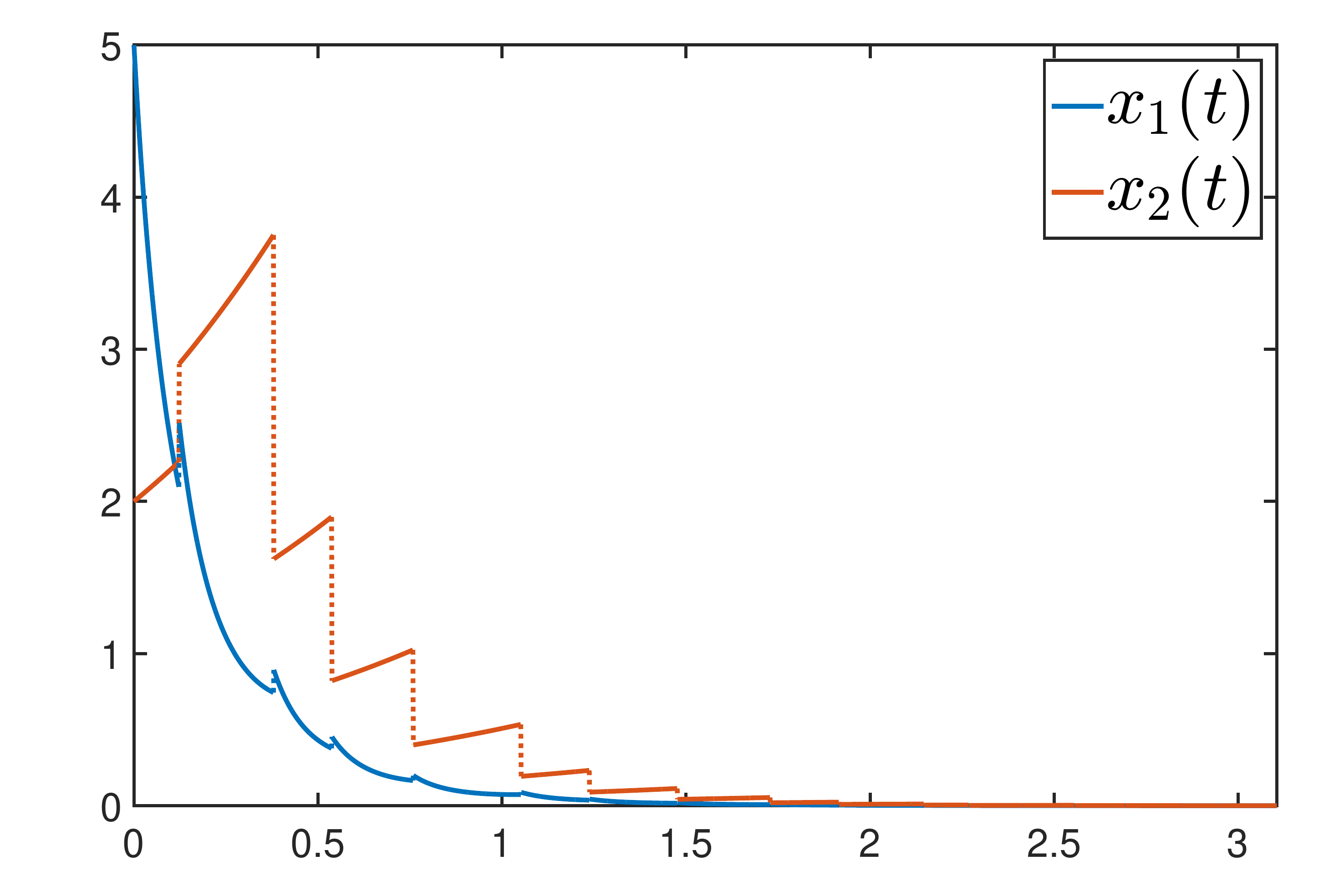}\hfill  \includegraphics[width=0.47\textwidth]{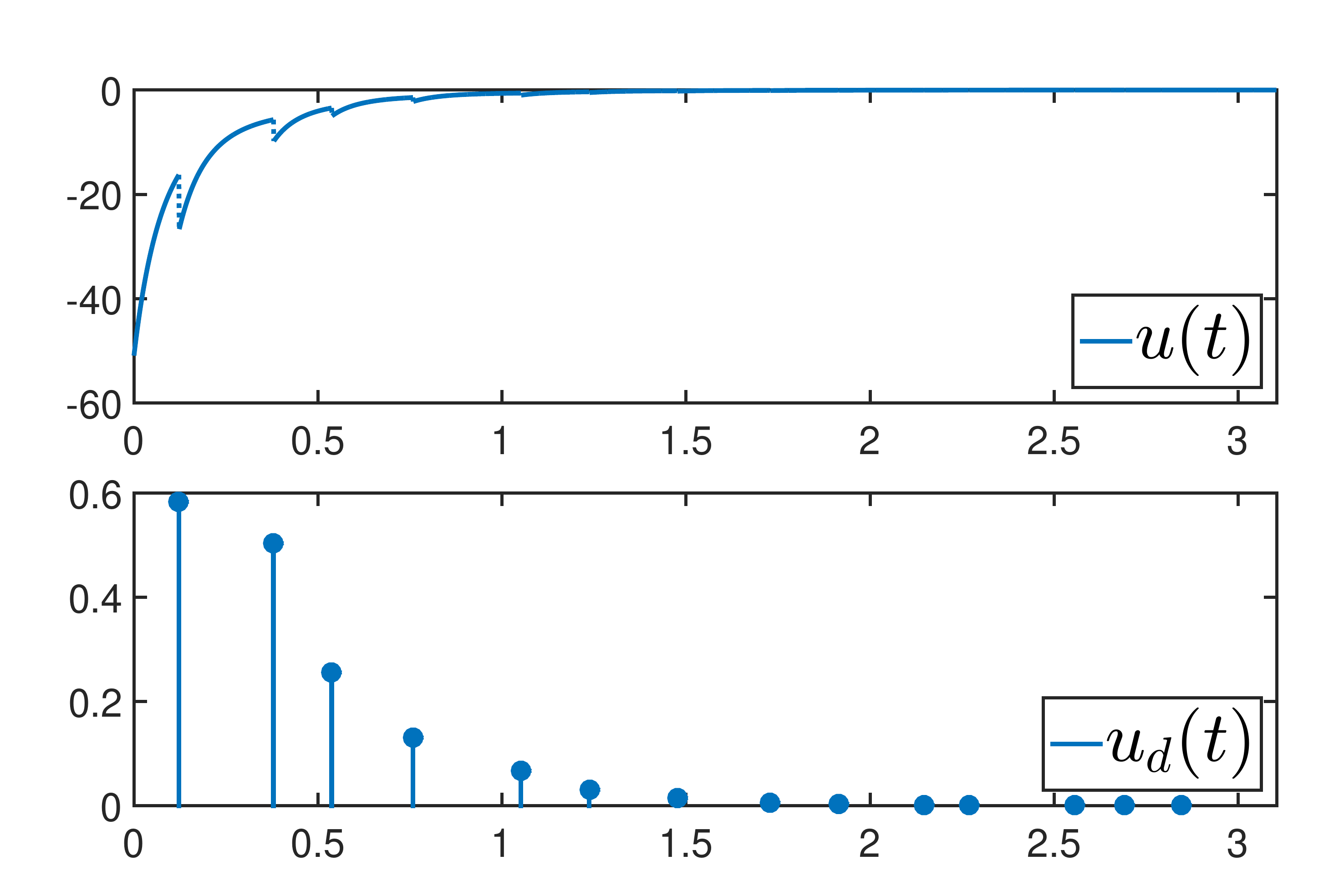}
  \caption{Evolution of the states (left) of closed-loop system \eqref{eq:mainsystu}-\eqref{eq:stabz:rangeDT1_cl}-\eqref{eq:sfcl3}-\eqref{eq:stabz:rangeDT1_cl} and the associated control inputs (right).}\label{fig:stabz_rangeDT}
\end{figure}
\end{example}

\section{Application to linear positive switched systems}\label{sec:switched}

The objective of this section is to derive several stability and stabilization conditions for linear positive switched systems lying in the same spirit as those obtained for impulsive systems. This is performed by exploiting the possibility for representing any switched system as an impulsive system, a procedure that will be described in Section \ref{sec:refo}. Using this reformulation, the results obtained in the previous sections are applied in order to derive stability and stabilization conditions linear positive switched systems under arbitrary switching (Section \ref{sec:SW:arbitrary}), minimum dwell-time switching (Section \ref{sec:SW:minDT}) and (mode-dependent) range dwell-time switching (Section \ref{sec:SW:range}). Some existing stability conditions are retrieved (e.g. minimum dwell-time) whereas the others seem to be novel. The stabilization conditions are also all novel and have not been reported anywhere before. Several illustrative examples are finally treated in Section \ref{sec:SW:examples}.

\subsection{Switched systems as impulsive systems}\label{sec:refo}

Let us consider the linear switched system
\begin{equation}\label{eq:switched}
\begin{array}{lcl}
    \dot{x}(t)&=&A_{\sigma(t)}x(t)+B_{\sigma(t)}u(t),\\
     x(0)&=&x_0
\end{array}
\end{equation}
where $x,x_0\in\mathbb{R}^n$ and $u\in\mathbb{R}^m$ are the state of the system, the initial condition and the control input. The signal $\sigma:\mathbb{R}_{\ge0}\mapsto\{1,\ldots,N\}$ is a left-continuous switching signal that changes values at the time instants $\{t_k\}_{k\in\mathbb{N}}$,  where this sequence obeys the same assumptions as the sequence considered for the system \eqref{eq:mainsyst}. This switched system can be reformulated as an impulsive system (actually a reset system) of the form
\begin{equation}\label{eq:switched_imp}
\begin{array}{rcl}
    \dot{\bar x}(t)&=&\bar A\bar x(t)+\bar B\bar u(t),\ t\ne t_k\\
    \bar x(t^+)&=&J_{ij}\bar x(t),\ t=t_k,\ i,j=1,\ldots,N,\ i\ne j
\end{array}
\end{equation}
with $\bar x\in\mathbb{R}^{Nn}$, $\bar u\in\mathbb{R}^{Nm}$, $\bar A = \diag(A_1,\ldots,A_N)$, $\bar B=\diag(B_1,\ldots,B_N)$ and $J_{ij}=e_ie_j^\T\otimes I_n$ where  $\{e_i\}_{i=1}^N$ is the standard basis for $\mathbb{R}^N$. \bluee{In this regard, we can clearly see that the set of impulsive systems with multiple jump maps contains the set of switched systems. Note that these sets are not equal as switched systems cannot implement state discontinuities. As a final remark, it is important to mention that the switched system \eqref{eq:switched} with $u\equiv0$ is positive if and only if all the matrices $A_i$, $i=1,\ldots,N$ are Metzler; see e.g. \cite{Gurvits:07,Mason:07}.}

\subsection{Stability and stabilization under arbitrary switching}\label{sec:SW:arbitrary}

\subsubsection{Stability under arbitrary switching}

Interestingly, we can recover from Theorem \ref{th:arbDT} the well-known conditions for stability under arbitrary switching of linear positive systems that can be found, for instance, in \cite{Mason:07,Mason:07b,Knorn:09,Fornasini:10}. This is stated in the following result:
\begin{corollary}\label{cor:arbDT}
  Let the matrices $A_1,\ldots,A_N$ be Metzler. Then, the following statements are equivalent:
  \begin{enumerate}[(a)]
    \item There exists a vector $\bar\lambda\in\mathbb{R}^{Nn}_{>0}$ such that the conditions
    \begin{equation}
      \bar\lambda^\T \bar A<0\ \textnormal{and}\ \bar\lambda^\T(J_{ij}-I_{nN})\le0,\ i,j=1,\ldots,N,\ i\ne j,
    \end{equation}
    hold.
    \item There exists a vector $\lambda\in\mathbb{R}^{n}_{>0}$ such that the conditions
    \begin{equation}\label{eq:arbDT22}
      \lambda^\T A_i<0,\ i=1,\ldots,N,
    \end{equation}
    hold.
  \end{enumerate}
  Moreover, when one of the above equivalent statements holds, then the linear positive switched system \eqref{eq:switched} with $u\equiv0$ is asymptotically stable under arbitrary switching.
\end{corollary}
\begin{proof}
  \blue{We use here Theorem \ref{th:arbDT} and Remark \ref{rem:arbDT1_flow} on persistent flowing. These conditions are exactly those in \eqref{th:arbDT}.} Decompose $\textstyle\bar\lambda=\col_i(\bar\lambda_i)$ where $\bar\lambda_i\in\mathbb{R}^n_{>0}$, $i=1,\ldots,N$. Then, the jump conditions hold if only if $\bar\lambda_i=\bar\lambda_j$ for all $i,j=1,\ldots,N$. Letting $\bar\lambda_i=\lambda$, $i=1,\ldots,N$, and substituting then this value in the flow condition gives \eqref{eq:arbDT22}. The proof is complete.
\end{proof}

The dual stability conditions \cite{Blanchini:15} can be retrieved using Remark \ref{rem:arbDT2}:
\begin{corollary}\label{cor:arbDT2}
  Let the matrices $A_1,\ldots,A_N$ be Metzler. Then, the following statements are equivalent:
  \begin{enumerate}[(a)]
    \item There exists a vector $\bar\lambda\in\mathbb{R}^{Nn}_{>0}$ such that the conditions
    \begin{equation}
      \bar A\bar\lambda<0\ \textnormal{and}\ (J_{ij}-I_{nN})\bar\lambda\le0,\ i,j=1,\ldots,N,\ i\ne j,
    \end{equation}
    hold.
    \item There exists a vector $\lambda\in\mathbb{R}^{n}_{>0}$ such that the conditions
    \begin{equation}\label{eq:dksmldkdskd}
       A_i\lambda<0,\ i=1,\ldots,N,
    \end{equation}
    hold.
  \end{enumerate}
    Moreover, when one of the above equivalent statements holds, then the linear positive switched system \eqref{eq:switched} with $u\equiv0$ is asymptotically stable under arbitrary switching.
\end{corollary}

\bluee{\begin{remark}
  It is interesting to remark that the stability condition of statement (b) can be retrieved by using the polyhedral Lyapunov function  $V(x)=\max_{i=1}^n\{\lambda_i^{-1}x_i\}$; see e.g. \cite{Blanchini:07,Blanchini:08,Blanchini:10,Blanchini:15} for more details.
\end{remark}}

\subsubsection{Stabilization under arbitrary switching}

\bluee{We have the following stabilization result under arbitrary switching can be obtained from Corollary \ref{cor:arbDT}:}
\begin{corollary}
  The following statements hold:
  \begin{enumerate}[(a)]
    \item There exists a control law of the form $u(t)=K_{\sigma(t)}x(t)$ such that the controlled system \eqref{eq:switched} is positive and asymptotically stable under arbitrary switching if there exist matrices $X\in\mathbb{D}_{\succ0}^n$, $U_i\in\mathbb{R}^{m\times n}$, $i=1,\ldots,N$, and a scalar $\alpha>0$ such that the conditions
        \begin{equation}\label{eq:stabZswarb1}
          A_iX+B_iU_i+\alpha I_n\ge0,\ \textnormal{and}\ [A_iX+B_iU_i]\mathds{1}_n<0
        \end{equation}
        hold for all $i=1,\ldots,N$. Moreover, in such a case, suitable controller gains can be obtained using the expression $K_i=U_iX^{-1}$.
        \item There exists a control law of the form $u(t)=Kx(t)$ such that the controlled system \eqref{eq:switched} is positive and asymptotically stable under arbitrary switching if there exist matrices $X\in\mathbb{D}_{\succ0}^n$, $U\in\mathbb{R}^{m\times n}$ and a scalar $\alpha>0$ such that the conditions
        \begin{equation}\label{eq:stabZswarb2}
          A_iX+B_iU+\alpha I_n\ge0,\ \textnormal{and}\ [A_iX+B_iU]\mathds{1}_n<0
        \end{equation}
        hold for all $i=1,\ldots,N$.  Moreover, in such a case, suitable controller gains can be obtained using the expression $K=UX^{-1}$.
  \end{enumerate}
\end{corollary}
\begin{proof}
  The proof of statement (a) is based on the substitution of the matrix of the closed-loop system $A_i+B_iK_i$ into \eqref{eq:dksmldkdskd}. Defining $X\in\mathbb{D}_{\succ0}^n$ as $\lambda=:D\mathds{1}_n$, then the change of variables $U_i:=K_iX$ yields the second condition of \eqref{eq:dksmldkdskd}. The positivity of the closed-loop system is ensured by the first condition of \eqref{eq:dksmldkdskd}. The proof of statement (b) follows from the same lines and the change of variables $U:=KX$.
\end{proof}

\subsection{Stability and stabilization under minimum dwell-time switching}\label{sec:SW:minDT}

\subsubsection{Stability under minimum dwell-time switching}

\bluee{By applying now Theorem \ref{th:minDT2} to the system \eqref{eq:switched_imp}, we get the following result where statement (b) was previously obtained in \cite{Zappavigna:10b,Blanchini:10,Blanchini:15} and the other statements are the "positive systems" analogues of the results in \cite{Geromel:06b,Briat:14f}:}
\begin{corollary}[Minimum dwell-time]\label{th:minDT_switched}\label{cor:minDT_switched2}
   Let the matrices $A_1,\ldots,A_N$ be Metzler. Then, the linear positive switched system \eqref{eq:switched} with $u\equiv0$ is asymptotically stable under minimum dwell-time $\bar{T}$ if one of the following equivalent statements holds:
  \begin{enumerate}[(a)]
  \item There exists a vector $\lambda\in\mathbb{R}_{>0}^{Nn}$such that the inequalities
    \begin{equation}
      \lambda^\T \bar A<0
    \end{equation}
    and
    \begin{equation}
          \lambda^\T\left[J_{ij}e^{\bar A\bar{T}}-I_{nN}\right] <0,\ i,j=1,\ldots,N,\ i\ne j
    \end{equation}
    hold.
    \item There exist vectors $\lambda_i\in\mathbb{R}_{>0}^n$, $i=1,\ldots,N$, such that the inequalities
    \begin{equation}
      \lambda_i^\T A_i<0,\ i=1,\ldots,N
    \end{equation}
    and
    \begin{equation}
          \lambda_i^\T e^{A_j\bar{T}} -\lambda_j^\T <0,\ i,j=1,\ldots,N,\ i\ne j
    \end{equation}
    hold.
    \item There exist a matrix function $\zeta_i:[0,\bar{T}]\mapsto\mathbb{R}^n$, $\zeta_i(\bar T)\in\mathbb{R}_{>0}^n$, $i=1,\ldots,N$, and a scalar $\eps>0$ such that the inequalities
\begin{equation}
      \zeta_i(\bar T)^\T A_i <0,\ i=1,\ldots,N
    \end{equation}
  \begin{equation}
    \zeta_i(\tau)^\T A_i -\dot{\zeta}_i(\tau)^\T \le0,\ i=1,\ldots,N
  \end{equation}
  and
  \begin{equation}
    \zeta_j(\bar{T})^\T -\zeta_i(0)^\T +\eps \mathds{1}_n^\T \le0,\ i,j=1,\ldots,N,\ i\ne j
  \end{equation}
 hold for all $\tau\in[0,\bar{T}]$.
     \item There exist a matrix function $\xi_i:[0,\bar{T}]\mapsto\mathbb{R}^n$, $\xi_i(0)\in\mathbb{R}_{>0}^n$, $i=1,\ldots,N$, and a scalar $\eps>0$ such that the inequalities
\begin{equation}
      \xi_i(0)^\T A_i<0,\ i=1,\ldots,N
    \end{equation}
  \begin{equation}
    \xi_i(\tau)^\T A_i+\dot{\xi}_i(\tau)^\T \le0,\ i=1,\ldots,N
  \end{equation}
  and
  \begin{equation}
    \xi_j(0) -\xi_i(\bar{T})^\T +\eps \mathds{1}_n^\T \le0,\ i,j=1,\ldots,N,\ i\ne j
  \end{equation}
 hold for all $\tau\in[0,\bar{T}]$.\mendth
   \end{enumerate}
\end{corollary}
\begin{proof}
  Decomposing $\lambda$ as $\lambda=\col(\lambda_1,\ldots,\lambda_N)$ and evaluating the conditions of Theorem \ref{th:minDT2} for all $J=J_{ij}$,  $i,j=1,\ldots,N,\ i\ne j$, leads to the result.
\end{proof}

\subsubsection{Stabilization under minimum dwell-time switching}

We consider here the following control law
\begin{equation}\label{eq:sfcl4}
    u_{\sigma(t_k)}(t_k+\tau)=\left\{\begin{array}{lcl}
    K_{\sigma(t_k)}(\tau)x(t_k+\tau)&&\text{if\ }\tau\in[0,\bar{T})\\
 K_{\sigma(t_k)}(\bar{T})x(t_k+\tau)&&\text{if\ }\tau\in[\bar{T},T_k).
  \end{array}\right.
\end{equation}
where $K_i:[0,\bar T]\mapsto\mathbb{R}^{m_c\times n}$, $i=1,\ldots,N$. \bluee{The above control law is analogous to the control law \eqref{eq:sfcl2} with the difference that it is now mode-dependent. The rationale for such a control law is explained in Section \ref{sec:impz:minimum}.} The following result is obtained by applying Theorem \ref{th:minDT_stabz2} to the closed-loop system \eqref{eq:switched_imp}-\eqref{eq:sfcl4} and can be seen as the "positive systems" version of a result in \cite{Briat:14f}:
\begin{corollary}[Stabilization under minimum dwell-Time]\label{th:sw_dtz2}\label{th:minDT_switched_2}
There exists a controller of the form \eqref{eq:sfcl4} such that the closed-loop system \eqref{eq:switched}-\eqref{eq:sfcl4} is positive and asymptotically stable under minimum dwell-time $\bar T$ if one of the following equivalent statements holds:
  \begin{enumerate}[(a)]
    \item There exist vectors $\lambda_i\in\mathbb{R}_{>0}^n$, $i=1,\ldots,N$, and matrix-valued functions $K_i:[0,\bar T]\mapsto\mathbb{R}^{m\times n}$, $i=1,\ldots,N$, such that the matrix $A+B_cK_c(\tau)$ is Metzler for all $\tau\in[0,\bar{T}]$ and such that the inequalities
    \begin{equation}\label{eq:minDTza1_2}
      (A_i+B_iK_i(\bar{T}))\lambda_i<0
    \end{equation}
    and
    \begin{equation}\label{eq:minDTza2_2}
     \Psi_i(\bar{T})\lambda_j-\lambda_i<0
    \end{equation}
    hold for all $i,j=1,\ldots,N$, $i\ne j$, where
    \begin{equation}
      \dfrac{d\Psi_i(s)}{ds}=\left(A_i+B_iK_i(s)\right)\Psi_i(s),\ \Psi_i(0)=I_n,\ s\ge0.
\end{equation}
 \item There exist matrix-valued functions ${X_i:[0,\bar{T}]\mapsto\mathbb{D}^n}$, $i=1,\ldots,N$, $X_i(\bar T)\in\mathbb{D}_{\succ0}^n$, $i=1,\ldots,N$, $U_i:[0,\bar{T}]\mapsto\mathbb{R}^{m_c\times n}$, $i=1,\ldots,N$, and scalars $\eps,\alpha>0$ such that the inequalities
         \begin{equation}
                  A_iX_i(\tau)+B_iU_i(\tau)+\alpha I\ge0,
         \end{equation}
          \begin{equation}
                  [A_iX_i(\bar T)+B_i(\bar T)U_i(\bar T)]\mathds{1}_n<0,
         \end{equation}
    \begin{equation}\label{eq:minCDTzc1_2}
      \left[-\dot{X}_i(\tau)+A_iX_i(\tau)+B_iU_i(\tau)\right]\mathds{1}_n<0
    \end{equation}
  and
  \begin{equation}\label{eq:minCDTzc3_2}
    \left[X_j(\bar T)-X_i(0)+\eps I\right]\mathds{1}_n\le0
  \end{equation}
 hold for all $\tau\in[0,\bar{T}]$. Moreover, in such a case, suitable controller gains can be computed using the relations  $K_i(\tau)=U_i(\tau)X_i(\tau)^{-1}$, $i=1,\ldots,N$.
   \end{enumerate}
\end{corollary}
\begin{proof}
  The proof readily follows from the substitution of the impulsive system formulation \eqref{eq:switched_imp} of the system \eqref{eq:switched} in the conditions of Theorem \ref{th:minDT_stabz2}, (b).
\end{proof}

\subsection{Stability and stabilization under range dwell-time switching}\label{sec:SW:range}

\subsubsection{Stability under range dwell-time switching}

The following result can be obtained by applying the conditions for the stability under range dwell-time described in Remark \ref{rem:swapped_range} to the system \eqref{eq:switched_imp} and can be seen as the "positive systems" version of a result in \cite{Briat:14f}:
\begin{corollary}\label{th:range_switched_2}
 Let the matrices $A_1,\ldots,A_N$ be Metzler. Then, the following statements are equivalent:
\begin{enumerate}[(a)]
  \item There exist vectors $\lambda_i\in\mathbb{R}^n_{>0}$ such that the conditions
  \begin{equation}
    \lambda_j^\T e^{A_i\theta_i}-\lambda_j^\T<0
  \end{equation}
 hold  for all $i,j=1,\ldots,N$, $i\ne j$ and $\theta_i\in[T^i_{min},T^i_{max}]$.
 \item There exist vector-valued functions $\zeta_i:[0,T^i_{max}]\mapsto\mathbb{R}^n$, $\zeta_i(0)\in\mathbb{R}_{>0}^n$, $i=1,\ldots,N$,  and a scalar $\eps>0$ such that the inequalities
  \begin{equation}
    \zeta_i(\tau_i)^\T A_i+\dot{\zeta}_i(\tau_i)^\T \le0,\ i=1,\ldots,N,
  \end{equation}
  and
  \begin{equation}
    \zeta_j(0)^\T-\zeta_i(\theta_i)^\T  +\eps \mathds{1}_n^\T \le0,\ i,j=1,\ldots,N,\ i\ne j
  \end{equation}
 hold for all $\tau_i\in[0,T^i_{max}]$ and all $\theta_i\in[T^i_{min},T^i_{max}]$.
\end{enumerate}
Moreover, when one of the above statements hold, the linear positive switched system \eqref{eq:switched} is asymptotically stable under range dwell-time $(T_{min},T_{max})$.
\end{corollary}
\begin{proof}
    The proof readily follows from the substitution of the impulsive system formulation \eqref{eq:switched_imp} of the system \eqref{eq:switched} in the conditions stated in Remark \ref{rem:swapped_range}.
\end{proof}

\subsubsection{Stabilization under range dwell-time switching}

We consider here the control law
\begin{equation}\label{eq:sfcl5}
    u_{\sigma(t_k)}(t_k+\tau)=K_{\sigma(t_k)}(\tau)x(t_k+\tau),\ \tau\in[0,T_k)
\end{equation}
where $K_i:[0,\bar T]\mapsto\mathbb{R}^{m_c\times n}$, $i=1,\ldots,N$.  \bluee{The above control law is the switched systems version of the control law \eqref{eq:sfcl3}. The rationale for such a control law is explained in Section \ref{sec:impz:maximum}.} The following result is obtained by applying Theorem \ref{sec:impz:range} to the system \eqref{eq:switched_imp}-\eqref{eq:sfcl5} and can be seen as the "positive systems" version of a result in \cite{Briat:14f}:
\begin{corollary}\label{th:range_switched_stabz_2}
There exists a controller of the form \eqref{eq:sfcl5} such that the closed-loop system \eqref{eq:switched}-\eqref{eq:sfcl5} is positive and asymptotically stable under range dwell-time if one of the following equivalent statements holds:
\begin{enumerate}[(a)]
\item There exist vectors $\lambda_i\in\mathbb{R}_{>0}^n$, $i=1,\ldots,N$, and matrix-valued functions $K_i:[0,T_{max}^i]\mapsto\mathbb{R}^{m\times n}$, $i=1,\ldots,N$, such that the matrix $A_i+B_iK_i(\tau_i)$, $i=1,\ldots,N$, is Metzler for all $\tau_i\in[0,T_{max}^i]$ and such that the inequality
    \begin{equation}\label{eq:minDTza2_2}
     \Psi_i(\theta_i)\lambda_j-\lambda_i<0
    \end{equation}
    holds for all $\theta_i\in[T_{min}^i,T_{max}^i]$ and all $i,j=1,\ldots,N$, $i\ne j$, where
    \begin{equation}
      \dfrac{d\Psi_i(s)}{ds}=\left(A_i+B_iK_i(s)\right)\Psi_i(s),\ \Psi_i(0)=I_n,\ s\ge0.
\end{equation}
 \item There exist matrix-valued functions $X_i:[0,T^i_{max}]\mapsto\mathbb{D}^n$, $X_i(0)\in\mathbb{D}_{\succ0}^n$, $i=1,\ldots,N$, matrix-valued functions $U_i:[0,T_{max}^i]\mapsto\mathbb{R}^{m\times n}$, $i=1,\ldots,N$, and scalars $\alpha,\eps>0$ such that the inequalities
         \begin{equation}
                  A_iX_i(\tau_i)+B_iU_i(\tau_i)+\alpha I\ge0,
         \end{equation}
    \begin{equation}
      \left[\dot{X}_i(\tau_i)+A_iX_i(\tau_i)+B_iU_i(\tau_i)\right]\mathds{1}_n<0
    \end{equation}
  and
  \begin{equation}
    \left[X_j(0)-X_i(\theta_i)+\eps I\right]\mathds{1}_n\le0
  \end{equation}
 hold for all $\tau_i\in[0,T_{max}^i]$, all $\theta_i\in[T_{min}^i,T_{max}^i]$ and all $i,j=1,\ldots,N$, $i\ne j$. Moreover, in such a case, suitable controller gains can be computed using the relations  $K_i(\tau)=U_i(\tau)X_i(\tau)^{-1}$, $i=1,\ldots,N$.
\end{enumerate}
\end{corollary}
\begin{proof}
      The proof readily follows from the substitution of the impulsive system formulation \eqref{eq:switched_imp} of the system \eqref{eq:switched} in the conditions stated in Theorem \ref{th:rangeDT_stabz}.
\end{proof}

\subsection{Examples}\label{sec:SW:examples}

\begin{example}[Stability under minimum dwell-time \#1]
Let us consider here the system \eqref{eq:switched} with the matrices \cite{Zhao:12b}
\begin{equation}\label{eq:switched_syst:minDT1}
  A_1=\begin{bmatrix}
    -0.5302& 0.0012& 0.0873\\
    0.2185 &-0.7494& 0.5411\\
    0.7370 &0.1543& -0.3606
  \end{bmatrix}\ \textnormal{and}\ A_2=\begin{bmatrix}
  -0.5136 &0.4419 &0.3689\\
    0.1840& -0.3951& 0.0080\\
    0.3163& 0.6099& -1.0056
  \end{bmatrix}.
\end{equation}
  Both matrices are Hurwitz stable but do not admit any common linear copositive Lyapunov function. We then look for an estimate for the minimum dwell-time and we get the results summarized in Table~\ref{tab:switched:1}. We can see that by increasing the degree of the polynomials in the sum of squares program, we can reach the value obtained by the original linear program. %We can also observe that the condition stated in Corollary \ref{cor:minDT_switched1} results in an infeasible program while the condition involving the swapped impulsive system associated with the switched system (i.e. Corollary \ref{cor:minDT_switched2}) yields a feasible program. This illustrates once again the need for stating stability conditions for the swapped system.
  The sum of squares conditions are also able to approach the values of minimum dwell-time computed using the conditions of Corollary  \ref{th:minDT_switched}, (a), when the degree $d_s$ of the polynomials increases.
\end{example}

\begin{example}[Stability under minimum dwell-time \#2]
  Let us consider here the system \eqref{eq:switched} with the matrices \cite{Blanchini:15}
\begin{equation}\label{eq:switched_syst:minDT2}
  A_1=\begin{bmatrix}
   -1.1309& 0.0087& 0.8499\\
0.0222& -1.0413 &0.5865\\
0.4105& 0.4817 &-0.8792
  \end{bmatrix}\ \textnormal{and}\ A_2=\begin{bmatrix}
  -2.9923& 1.5069 &2.9142\\
4.0681& -3.9685 &1.8570\\
0.1072 &0.0618& -0.7999
  \end{bmatrix}.
\end{equation}
  As in the previous example, these matrices are Hurwitz stable but do not admit any common linear copositive Lyapunov function. %We can observe in Table~\ref{tab:switched:1} that, for this example again, the conditions stated in Corollary \ref{cor:minDT_switched1} are more conservative than the conditions  stated in Corollary \ref{cor:minDT_switched2}.
  We can observe  in Table~\ref{tab:switched:1} that the sum of squares conditions are again able to approach the values of minimum dwell-time computed using the conditions of Corollary  \ref{th:minDT_switched}, (a), when the degree $d_s$ of the polynomials increases.
\end{example}

\begin{table}[H]
  \centering
  \caption{Minimum dwell-time results for the switched systems \eqref{eq:switched}-\eqref{eq:switched_syst:minDT1} and \eqref{eq:switched}-\eqref{eq:switched_syst:minDT2} obtained using Corollary \ref{th:minDT_switched}.}\label{tab:switched:1}
  \begin{tabular}{|c|c|c|c|c|c|}
  \hline
     & Result & Method & $\bar T$ & No. variables & Solving time\\
    \hline
    \hline
    \multirow{5}{*}{System \eqref{eq:switched_syst:minDT1}} %&Cor. \ref{cor:minDT_switched1}, (a) & -- & infeasible & -- & --\\
                                                                                & Cor. \ref{th:minDT_switched}, (a) & -- & 3.4296 & 18/6 & 0.9024\\
                                                                                & Cor. \ref{th:minDT_switched}, (c) & SOS ($d_s=1$) & infeasible & 114/48 & --\\
                                                                                & Cor. \ref{th:minDT_switched}, (c) & SOS ($d_s=2$) & 3.7063 & 198/60 & 1.2610\\
                                                                                & Cor. \ref{th:minDT_switched}, (c) & SOS ($d_s=3$) & 3.4538 & 306/72 & 1.6763\\
\hline
  \multirow{5}{*}{System \eqref{eq:switched_syst:minDT2}} %& Cor. \ref{cor:minDT_switched1}, (a) & -- & 4.2022 & 18/6 & 0.9090\\
                                                                                & Cor. \ref{th:minDT_switched}, (a) & -- & 1.0717 & 18/6 & 1.4464\\
                                                                                & Cor. \ref{th:minDT_switched}, (c) & SOS ($d_s=1$) & 5.0992 & 114/48 & 1.0844\\
                                                                                & Cor. \ref{th:minDT_switched}, (c) & SOS ($d_s=2$) & 2.2637 & 198/60 & 1.2228\\
                                                                                & Cor. \ref{th:minDT_switched}, (c) & SOS ($d_s=3$) & 1.0862  & 306/72 & 1.5291 \\
\hline
  \end{tabular}
\end{table}

\begin{example}[Stability under range dwell-time]
 Let us consider the system \eqref{eq:switched} with the matrices \cite{Briat:13b}
  \begin{equation}\label{eq:ex1}
  \begin{array}{lclclcl}
        A_1&=&\begin{bmatrix}
    -2 & 1\\
    5 & -3
    \end{bmatrix},& &A_2&=&\begin{bmatrix}
      0.1 & 0\\
      0.1 & 0.2
    \end{bmatrix}.
  \end{array}
  \end{equation}
Note that the matrix $A_1$ is Hurwitz stable and $A_2$ is anti-stable. Then the system cannot be stable under arbitrary switching or under minimum dwell-time switching. We hence consider a mode-dependent range dwell-time criterion and let $T_{min}^2=0.01$ and $T^1_{max}=\infty$. We then consider Theorem  \ref{th:range_switched_2} for different values for $T^1_{min}=1$ and compute the associated maximal values for $T_{max}^2$. The results are gathered in Table \ref{tab:switched:2} where we can recover the values by choosing a sufficiently large polynomial degree $d_s$. Compared to the gridding approach, the SOS approach yields more accurate conditions that are faster to solve.
\end{example}

\begin{table}[H]
  \centering
  \caption{Mode-dependent range dwell-time results for the switched system \eqref{eq:switched}-\eqref{eq:ex1} with  $T_{min}^2=0.01$ and $T^1_{max}=\infty$ obtained using  Corollary \ref{th:range_switched_2}.}\label{tab:switched:2}
  \begin{tabular}{|c|c|c|c|c|c|}
  \hline
     & Result & Method & $T^2_{max}$ & No. variables & Solving time\\
    \hline
    \hline
\multirow{5}{*}{\makecell{System~\eqref{eq:ex1},\\ $T^1_{min}=1$}} %& Th. \ref{th:range_switched_1}, (a) & Gridded ($N_g=201$) & 1.2847  & 410/4 & 3.9100\\
                                                                                & Th. \ref{th:range_switched_2}, (a) & Gridded ($N_g=201$) & 1.2847& 410/4 & 3.6096\\
                                                                                & Th. \ref{th:range_switched_2}, (b) & SOS ($d_s=1$) &  1.2788 & 80/28 & 0.8220\\
                                                                                & Th. \ref{th:range_switched_2}, (b) & SOS ($d_s=2$) &   1.2847 & 146/36 & 0.9641\\
                                                                                & Th. \ref{th:range_switched_2}, (b) & SOS ($d_s=3$) &   1.2847 & 232/44 & 1.1172\\
\hline
\multirow{5}{*}{\makecell{System~\eqref{eq:ex1},\\ $T^1_{min}=2$}} %& Th. \ref{th:range_switched_1}, (a) & Gridded ($N_g=201$) & 2.4154  & 410/4 & 3.5536\\
                                                                                & Th. \ref{th:range_switched_2}, (a) & Gridded ($N_g=201$) & 2.5471 & 410/4 & 3.5899\\
                                                                                & Th. \ref{th:range_switched_2}, (b) & SOS ($d_s=1$) &  2.5010 & 80/28 & 0.7727\\
                                                                                & Th. \ref{th:range_switched_2}, (b) & SOS ($d_s=2$) &   2.5470 & 146/36 & 1.0925\\
                                                                                & Th. \ref{th:range_switched_2}, (b) & SOS ($d_s=3$) &   2.5470 & 232/44 & 1.2273\\
\hline
\multirow{5}{*}{\makecell{System~\eqref{eq:ex1},\\ $T^1_{min}=5$}} %& Th. \ref{th:range_switched_1}, (a) & Gridded ($N_g=201$) & 5.5461  & 410/4 & 3.6841\\
                                                                                & Th. \ref{th:range_switched_2}, (a) & Gridded ($N_g=201$) & 6.2158 & 410/4 & .37078\\
                                                                                & Th. \ref{th:range_switched_2}, (b) & SOS ($d_s=1$) &  5.5783 & 80/28 & 0.8722\\
                                                                                & Th. \ref{th:range_switched_2}, (b) & SOS ($d_s=2$) &   6.2112 & 146/36 & 1.0439\\
                                                                                & Th. \ref{th:range_switched_2}, (b) & SOS ($d_s=3$) &   6.2140 & 232/44 & 1.3031\\
\hline
%\multirow{5}{*}{\parbox{2cm}{System~\eqref{eq:ex1} {$T^1\in[7,\infty)$}}} & Th. \ref{th:range_switched_1}, (a) & Gridded ($N_g=201$) & 7.6332  & 410/4 &3.5478\\
%                                                                                & Th. \ref{th:range_switched_2}, (a) & Gridded ($N_g=201$) & 8.5803& 410/4 &3.3789\\
%                                                                                & Th. \ref{th:range_switched_2}, (b) & SOS ($d_s=1$) &  7.0319 & 80/28 & 0.8488\\
%                                                                                & Th. \ref{th:range_switched_2}, (b) & SOS ($d_s=2$) &  8.5569 & 146/36 & 1.0847\\
%                                                                                & Th. \ref{th:range_switched_2}, (b) & SOS ($d_s=3$) &  8.6006 & 232/44 & 1.3273\\
%\hline
  \end{tabular}
\end{table}

\begin{example}[Stabilization under minimum dwell-time]
Let us consider the system \eqref{eq:switched} with the matrices
  \begin{equation}\label{eq:ex:minDT_switched_stabz1}
    A_1=\begin{bmatrix}
      1 &2\\ 0& -1
    \end{bmatrix},\ B_1=\begin{bmatrix}
      1\\ 0
    \end{bmatrix},\ A_2=\begin{bmatrix}
      -2 &1\\ 2& 3
    \end{bmatrix},\ B_2=\begin{bmatrix}
      0\\
      1
    \end{bmatrix}.
  \end{equation}
  Note that both subsystems are unstable. The goal is then to design a controller of the form \eqref{eq:sfcl4} such that the closed-loop system is positive and stable under minimum dwell-time $\bar T=0.1$. We consider Theorem \ref{th:minDT_switched_2} and solve the conditions using the SOS relaxation with $d_s=1/2$ (degree of $X$ and $U$ is one while the additional polynomials are of degree 2). The obtained semidefinite program has 180 primal variables and 71 dual variables. Solving the conditions  yields the controller matrices
\begin{equation}\label{eq:ex:minDT_switched_stabz2}
\begin{array}{lcl}
  K_1(\tau)&=&\begin{bmatrix}
    \dfrac{-0.7115\tau+2.5772}{0.8189\tau   -0.8234} &  \dfrac{-0.5314\tau+0.9017}{0.2750\tau   -0.6962}
  \end{bmatrix},\\
  K_2(\tau)&=&\begin{bmatrix}
    \dfrac{-0.5525 \tau + 1.0274}{0.2470 \tau  -0.7918} &  \dfrac{-1.1025 \tau+   3.6721}{0.4632 \tau  -0.7159}.
  \end{bmatrix}
  \end{array}
\end{equation}
For simulation purposes, we generate a random sequence of dwell-times satisfying the minimum dwell-time constraint and we obtain the trajectories depicted in Figure \ref{fig:minDT_switched_stabz} where we can see observe the stabilizing effect of the controller.

\begin{figure}[H]
  \centering
  \includegraphics[width=0.47\textwidth]{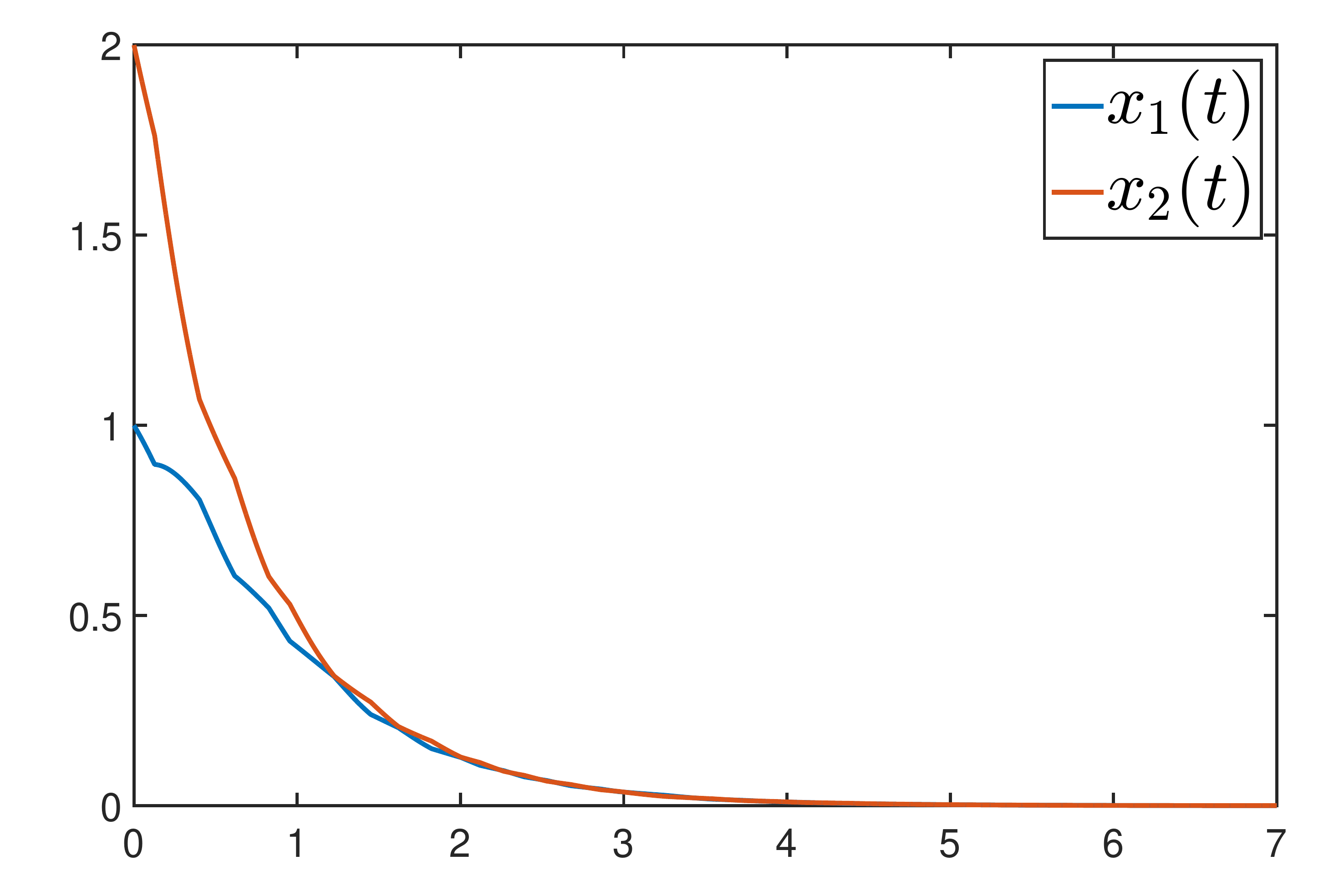}\hfill  \includegraphics[width=0.47\textwidth]{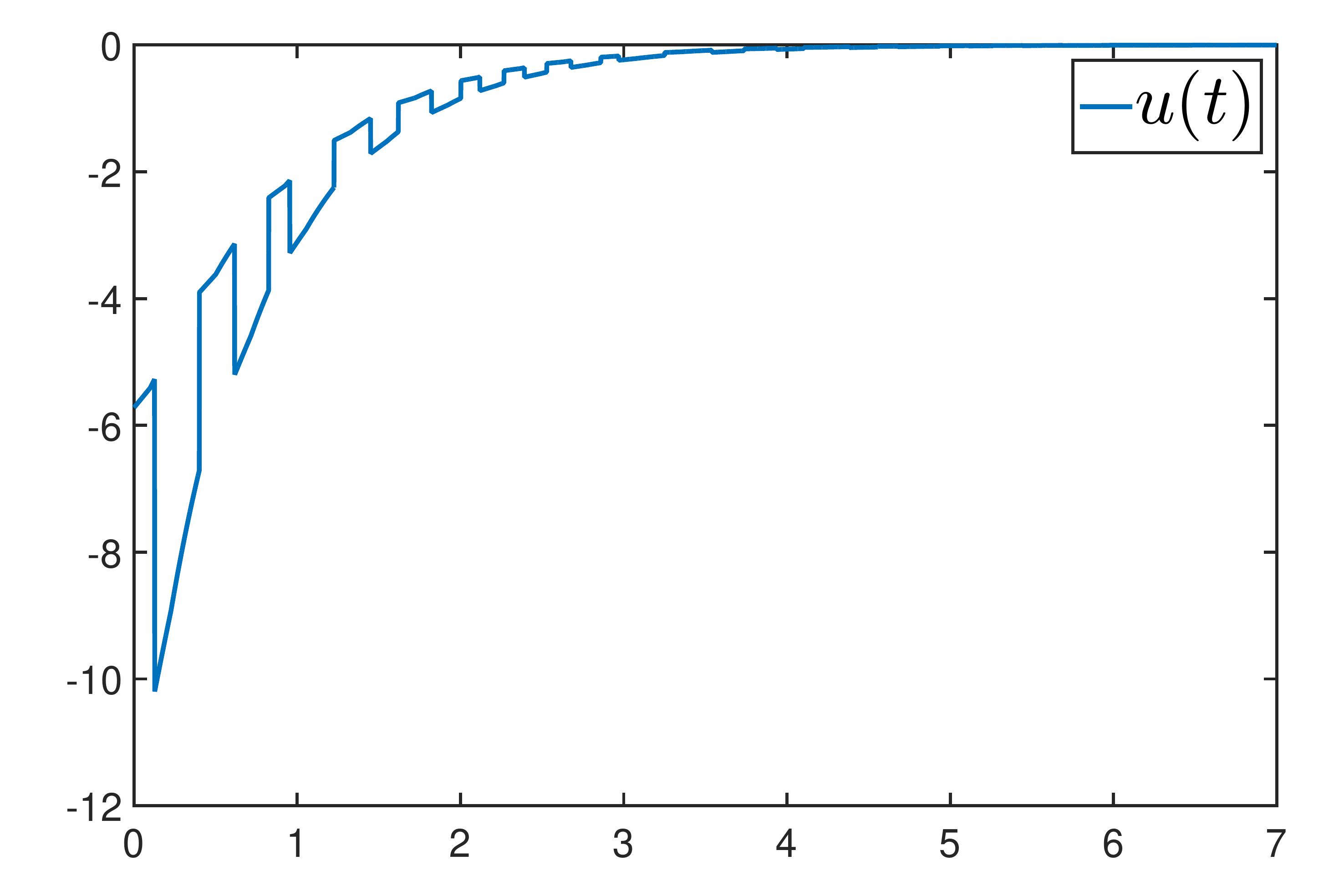}
  \caption{State (left) and control input (right) trajectories of the closed-loop system \eqref{eq:switched}-\eqref{eq:ex:minDT_switched_stabz1}-\eqref{eq:sfcl4}-\eqref{eq:ex:minDT_switched_stabz2}.}\label{fig:minDT_switched_stabz}
\end{figure}
\end{example}

\begin{example}[Stabilization under range dwell-time]
Let us consider the system \eqref{eq:switched} with the matrices
  \begin{equation}\label{eq:ex:rangeDT_switched_stabz1}
    A_1=\begin{bmatrix}
      2 &3\\ 1& -3
    \end{bmatrix},\ B_1=\begin{bmatrix}
      1\\ 0
    \end{bmatrix},\ A_2=\begin{bmatrix}
      -4 &2\\ 3& -2
    \end{bmatrix},\ B_2=\begin{bmatrix}
      0\\
      1
    \end{bmatrix}.
  \end{equation}
  As in the previous examples, both subsystems are unstable and we seek for a controller of the form \eqref{eq:sfcl5} that makes the closed-loop system positive and stable under range dwell-time  $(T_{min}^1,T_{max}^1)=(0.2,0.3)$ and $(T_{min}^2,T_{max}^2)=(0.1,0.2)$.  We consider Theorem \ref{th:range_switched_2} and solve the conditions using the SOS relaxation with $d_s=1/2$ (degree of $X$ and $U$ is one while the additional polynomials are of degree 2). The semidefinite program to solve has  229 primal variables and 84 dual variables. Solving the conditions takes 1.5875 seconds and returns the controller matrices
\begin{equation}\label{eq:ex:rangeDT_switched_stabz2}
\begin{array}{lcl}
  K_1(\tau)&=&\begin{bmatrix}
    \dfrac{-0.9381\tau-2.7094}{0.4876\tau+0.4790} &  \dfrac{-0.7981\tau   -0.9812}{0.3152\tau+  0.5428}
  \end{bmatrix},\\
  K_2(\tau)&=&\begin{bmatrix}
    \dfrac{-0.2044\tau   -0.6353}{0.1183\tau+    0.5053} &  \dfrac{0.2601\tau   -0.7524}{0.2289\tau+ 0.5581}
  \end{bmatrix}.
  \end{array}
\end{equation}
For simulation purposes, we generate a random sequence of dwell-times satisfying the range dwell-time constraint  and we obtain the trajectories depicted in Figure \ref{fig:rangeDT_switched_stabz} where we can see observe the stabilizing effect of the controller.

\begin{figure}[H]
  \centering
  \includegraphics[width=0.47\textwidth]{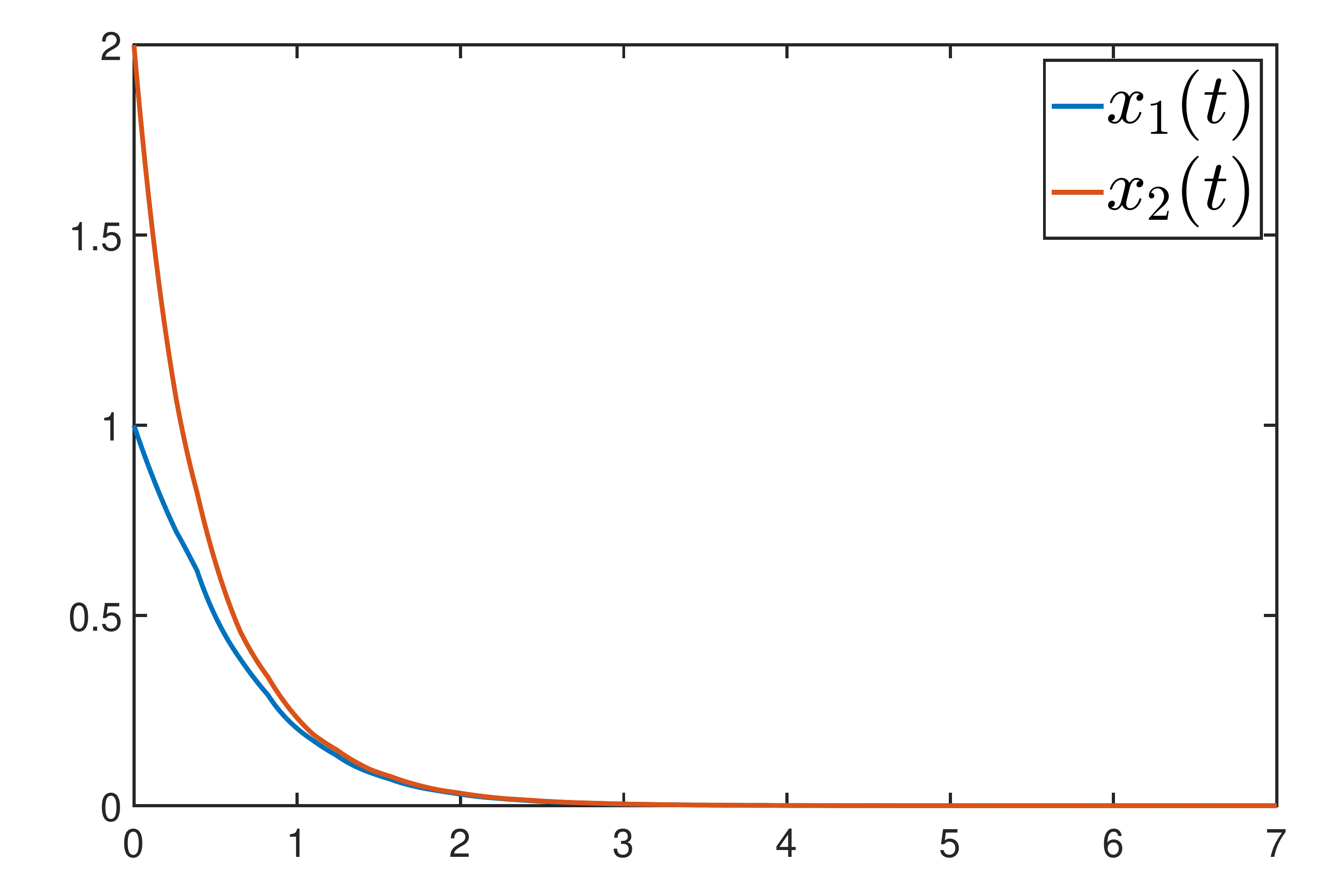}\hfill  \includegraphics[width=0.47\textwidth]{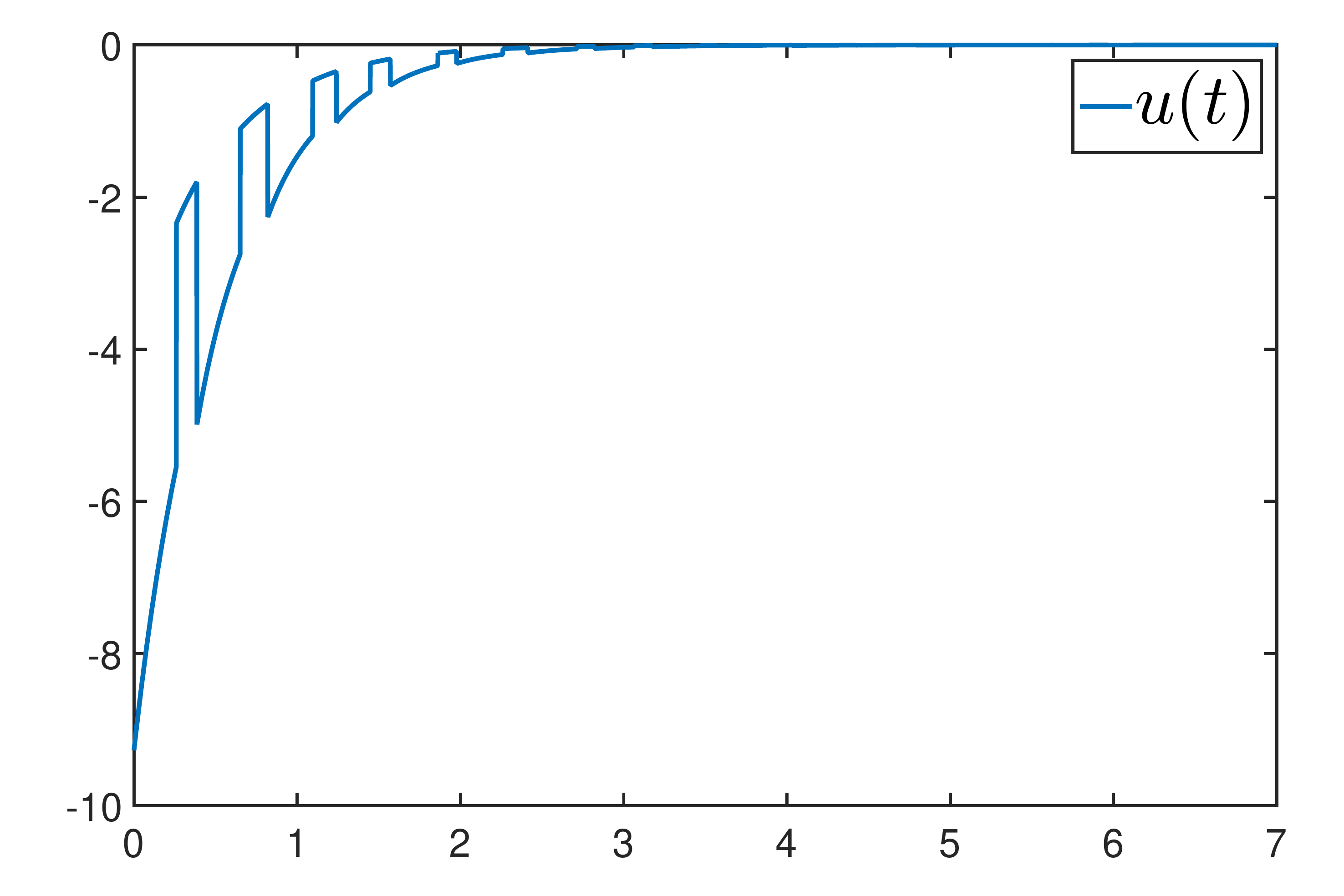}
   \caption{State (left) and control input (right) trajectories of the closed-loop system \eqref{eq:switched}-\eqref{eq:ex:rangeDT_switched_stabz1}-\eqref{eq:sfcl5}-\eqref{eq:ex:rangeDT_switched_stabz2}.}\label{fig:rangeDT_switched_stabz}
\end{figure}

\end{example}

\section{Conclusion}

Several stability and stabilization conditions for linear positive impulsive systems and linear positive switched systems have been obtained using the notion of dwell-times. Discrete-time stability conditions have been obtained and reformulated using a lifting approach in order to get conditions that can be applied for design purposes. Because of their infinite-dimensional nature, several relaxation approaches have been proposed to make the lifted conditions finite-dimensional. Several examples have been given for illustration.

\blue{Possible extensions are the stability analysis and stabilization linear positive impulsive systems with inputs using the $L_1$-, $L_2$- and $L_\infty$-gains in the same way it has been done with linear positive systems \cite{Briat:11g,Briat:11g,Tanaka:10,Tanaka:13a}, linear positive systems with delays \cite{Briat:11g,Shen:14,Shen:15}, linear positive switched systems \cite{Blanchini:15} and linear positive Markov jump systems \cite{Bolzern:15}. The consideration of the obtained conditions for the design of interval observers \cite{Gouze:00,Mazenc:11,Briat:15g} for (not necessarily positive) linear impulsive systems or linear switched systems is also an interesting application of the developed theory. Developing lifted conditions for polyhedral Lyapunov functions is also quite appealing due to the universal properties these functions have. Finally, the extension to positive/monotone nonlinear systems is also of potential interest.}


\begin{thebibliography}{104}
\providecommand{\natexlab}[1]{#1}
\providecommand{\url}[1]{\texttt{#1}}
\expandafter\ifx\csname urlstyle\endcsname\relax
  \providecommand{\doi}[1]{doi: #1}\else
  \providecommand{\doi}{doi: \begingroup \urlstyle{rm}\Url}\fi

\bibitem[Farina and Rinaldi(2000)]{Farina:00}
L.~Farina and S.~Rinaldi.
\newblock \emph{Positive Linear Systems: Theory and Applications}.
\newblock John Wiley \& Sons, 2000.

\bibitem[Shorten et~al.(2006)Shorten, Wirth, and Leith]{Shorten:06}
R.~Shorten, F.~Wirth, and D.~Leith.
\newblock A positive systems model of {TCP}-like congestion control: asymptotic
  results.
\newblock \emph{{IEEE} Transactions on Networking}, 14(3):\penalty0 616--629,
  2006.

\bibitem[Briat et~al.(2015)Briat, Yavuz, Hjalmarsson, Johansson, J{\"{o}}nsson,
  Karlsson, and Sandberg]{Briat:13f}
C.~Briat, E.~A. Yavuz, H.~Hjalmarsson, K.-H. Johansson, U.~T. J{\"{o}}nsson,
  G.~Karlsson, and H.~Sandberg.
\newblock The conservation of information, towards an axiomatized modular
  modeling approach to congestion control.
\newblock \emph{{IEEE} Transactions on Networking}, 23(3), 2015.

\bibitem[Murray(2002)]{Murray:02}
J.~D. Murray.
\newblock \emph{Mathematical Biology Part I. An Introduction. 3rd Edition}.
\newblock Springer-Verlag Berlin Heidelberg, 2002.

\bibitem[Briat and Khammash(2012)]{Briat:12c}
C.~Briat and M.~Khammash.
\newblock Computer control of gene expression: Robust setpoint tracking of
  protein mean and variance using integral feedback.
\newblock In \emph{51st {IEEE} Conference on Decision and Control}, pages
  3582--3588, Maui, Hawaii, USA, 2012.

\bibitem[Briat and Khammash(2013)]{Briat:13h}
C.~Briat and M.~Khammash.
\newblock Integral population control of a quadratic dimerization process.
\newblock In \emph{52nd {IEEE} Conference on Decision and Control}, pages
  3367--3372, Florence, Italy, 2013.

\bibitem[Briat et~al.(2016)Briat, Gupta, and Khammash]{Briat:15e}
C.~Briat, A.~Gupta, and M.~Khammash.
\newblock Antithetic integral feedback ensures robust perfect adaptation in
  noisy biomolecular networks.
\newblock \emph{Cell Systems}, 2:\penalty0 17--28, 2016.

\bibitem[Briat and Verriest(2009)]{Briat:09h}
C.~Briat and E.~I. Verriest.
\newblock A new delay-{SIR} model for pulse vaccination.
\newblock \emph{Biomedical signal processing and control}, 4(4):\penalty0
  272--277, 2009.

\bibitem[Jonsson et~al.(2014)Jonsson, Rantzer, and Murray]{Jonsson:14}
V.~Jonsson, A.~Rantzer, and R.~M. Murray.
\newblock A scalable formulation for engineering combination therapies for
  evolutionary dynamics of disease.
\newblock In \emph{American Control Conference}, pages 2771--2778, Portland,
  USA, 2014.

\bibitem[Haddad and Chellaboina(2005)]{Haddad:05}
W.~M. Haddad and V.~Chellaboina.
\newblock Stability and dissipativity theory for nonnegative dynamic systems: a
  unified analysis framework for biological and physiological systems.
\newblock \emph{Nonlinear Analysis: Real World Applications}, 6:\penalty0
  35--65, 2005.

\bibitem[{Ait Rami} and Tadeo(2007)]{Aitrami:07}
M.~{Ait Rami} and F.~Tadeo.
\newblock Controller synthesis for positive linear systems with bounded
  controls.
\newblock \emph{{IEEE} Transactions on Circuits and Systems -- II. Express
  Briefs}, 54(2):\penalty0 151--155, 2007.

\bibitem[Briat(2013{\natexlab{a}})]{Briat:11h}
C.~Briat.
\newblock Robust stability and stabilization of uncertain linear positive
  systems via integral linear constraints - ${L_1}$- and ${L_\infty}$-gains
  characterizations.
\newblock \emph{{I}nternational {J}ournal of {R}obust and {N}onlinear
  {C}ontrol}, 23(17):\penalty0 1932--1954, 2013{\natexlab{a}}.

\bibitem[{Ait Rami}(2011)]{AitRami:11}
M.~{Ait Rami}.
\newblock Solvability of static output-feedback stabilization for {LTI}
  positive systems.
\newblock \emph{Systems \& Control Letters}, 60:\penalty0 704--708, 2011.

\bibitem[Tanaka et~al.(2013)Tanaka, Langbort, and Ugrinovskii]{Tanaka:13b}
T.~Tanaka, C.~Langbort, and V.~Ugrinovskii.
\newblock {DC}-dominant property of cone-preserving transfer functions.
\newblock \emph{Systems \& Control Letters}, 62(8):\penalty0 699--707, 2013.

\bibitem[Briat(2011)]{Briat:11g}
C.~Briat.
\newblock Robust stability analysis of uncertain linear positive systems via
  integral linear constraints - ${L_1}$- and ${L_\infty}$-gains
  characterizations.
\newblock In \emph{50th {IEEE} Conference on Decision and Control}, pages
  3122--3129, Orlando, Florida, USA, 2011.

\bibitem[Rantzer(2015)]{Rantzer:15}
A.~Rantzer.
\newblock An extended {K}alman-{Y}akubovich-{P}opov lemma for positive systems.
\newblock In \emph{1st IFAC Conference on Modelling, Identification and Control
  of Nonlinear Systems}, Saint Petersburg, Russia, 2015.

\bibitem[Colombino and Smith(2016)]{Colombino:15}
M.~Colombino and R.~S. Smith.
\newblock A convex characterization of robust stability for positive and
  positively dominated linear systems.
\newblock \emph{IEEE Transactions on Automatic Control}, 61(7):\penalty0
  1965--1971, 2016.

\bibitem[Colombino et~al.(2015)Colombino, Hempel, and Smith]{Colombino:15b}
M.~Colombino, A.~B. Hempel, and R.~S. Smith.
\newblock Robust stability of a class of interconnected nonlinear positive
  systems.
\newblock In \emph{American Control Conference}, pages 5312--5317, Chicago,
  USA, 2015.

\bibitem[Khong et~al.(2015)Khong, Briat, and Rantzer]{Briat:15cdc}
S.~Z. Khong, C.~Briat, and A.~Rantzer.
\newblock Positive systems analysis via integral linear constraints.
\newblock In \emph{54th {IEEE} Conference on Decision and Control}, pages
  6373--6378, Osaka, Japan, 2015.

\bibitem[Haddad and Chellaboina(2004)]{Haddad:04}
W.~M. Haddad and V.~Chellaboina.
\newblock Stability theory for nonnegative and compartmental dynamical systems
  with time delay.
\newblock \emph{Systems \& Control Letters}, 51(5):\penalty0 355--361, 2004.

\bibitem[{Ait Rami}(2009)]{AitRami:09}
M.~{Ait Rami}.
\newblock Stability analysis and synthesis for linear positive systems with
  time-varying delays.
\newblock In \emph{Positive systems - Proceedings of the 3rd
  {M}ultidisciplinary {I}nternational {S}ymposium on {P}ositive {S}ystems:
  {T}heory and {A}pplications ({POSTA} 2009)}, pages 205--216. Springer-Verlag
  Berlin Heidelberg, 2009.

\bibitem[Shen and Lam(2015)]{Shen:15}
J.~Shen and J.~Lam.
\newblock $\ell_\infty$/${L}_\infty$-gain analysis for positive linear systems
  with unbounded time-varying delays.
\newblock \emph{IEEE Transactions on Automatic Control}, 60(3):\penalty0
  857--862, 2015.

\bibitem[Briat(2016{\natexlab{a}})]{Briat:16b}
C.~Briat.
\newblock Exact stability results for linear positive systems with delays:
  alternative proofs using input-output methods.
\newblock \emph{submitted}, 2016{\natexlab{a}}.

\bibitem[Fornasini and Valcher(2010)]{Fornasini:10}
E.~Fornasini and M.~E. Valcher.
\newblock Linear copositive {L}yapunov functions for continuous-time positive
  switched systems.
\newblock \emph{IEEE Transactions on Automatic Control}, 55(8):\penalty0
  1933--1937, 2010.

\bibitem[Zappavigna et~al.(2010{\natexlab{a}})Zappavigna, Colaneri, Geromel,
  and Middleton]{Zappavigna:10a}
A.~Zappavigna, P.~Colaneri, J.~C. Geromel, and R.~H. Middleton.
\newblock Stabilization of continuous-time switched linear positive systems.
\newblock In \emph{American Control Conference}, pages 3275--3280, Baltimore,
  Maryland, {USA}, 2010{\natexlab{a}}.

\bibitem[Zappavigna et~al.(2010{\natexlab{b}})Zappavigna, Colaneri, Geromel,
  and Shorten]{Zappavigna:10b}
A.~Zappavigna, P.~Colaneri, J.~C. Geromel, and R.~Shorten.
\newblock Dwell time analysis for continuous-time switched linear positive
  systems.
\newblock In \emph{American Control Conference}, pages 6256--6261, Baltimore,
  Maryland, {USA}, 2010{\natexlab{b}}.

\bibitem[Blanchini et~al.(2015)Blanchini, Colaneri, and Valcher]{Blanchini:15}
F.~Blanchini, P.~Colaneri, and M.~E. Valcher.
\newblock Switched positive linear systems.
\newblock \emph{Foundations and Trends in Systems and Control},
  2(2-3):\penalty0 101--273, 2015.

\bibitem[Bolzern et~al.(2014)Bolzern, Colaneri, and {De Nicalao}]{Bolzern:14}
P.~Bolzern, P.~Colaneri, and G.~{De Nicalao}.
\newblock Stochastic stability of positive {M}arkov jump linear systems.
\newblock \emph{Automatica}, 50(4):\penalty0 1181--1187, 2014.

\bibitem[Bolzern and Colaneri(2015)]{Bolzern:15}
P.~Bolzern and P.~Colaneri.
\newblock Positive markov jump linear systems.
\newblock \emph{Foundations and Trends in Systems and Control},
  2(3-4):\penalty0 275--427, 2015.

\bibitem[Agarwal et~al.(2012)Agarwal, Berezansky, Braverman, and
  Domoshnitsky]{Agarwal:12}
R.~P. Agarwal, L.~Berezansky, E.~Braverman, and A.~Domoshnitsky.
\newblock \emph{Nonoscillation theory of functional differential equations with
  applications}.
\newblock Springer-Verlag, New York, USA, 2012.

\bibitem[Mazenc and Malisoff(2016)]{Mazenc:16}
F.~Mazenc and M.~Malisoff.
\newblock Stability analysis for time-varying systems with delay using linear
  {L}yapunov functionals and a positive systems approach.
\newblock \emph{IEEE Transactions on Automatic Control}, 61(3):\penalty0
  771--776, 2016.

\bibitem[Ngoc and Trinh(2016)]{Ngoc:16}
P.~H.~A. Ngoc and H.~Trinh.
\newblock Novel criteria for exponential stability of linear neutral
  time-varying differential systems.
\newblock \emph{IEEE Transactions on Automatic Control}, 61(6):\penalty0
  1590--1594, 2016.

\bibitem[Ngoc and Tinh(2016)]{Ngoc:16b}
P.~H.~A. Ngoc and C.~T. Tinh.
\newblock Explicit criteria for exponential stability of time-varying systems
  with infinite delay.
\newblock \emph{Mathematics of Control, Signals, and Systems}, 28(4):\penalty0
  1--30, 2016.

\bibitem[Gouz{\'{e}} et~al.(2000)Gouz{\'{e}}, Rapaport, and
  {Hadj-Sadok}]{Gouze:00}
J.~L. Gouz{\'{e}}, A.~Rapaport, and M.~Z. {Hadj-Sadok}.
\newblock Interval observers for uncertain biological systems.
\newblock \emph{Ecological modelling}, 133:\penalty0 45--56, 2000.

\bibitem[Mazenc and Bernard(2011)]{Mazenc:11}
F.~Mazenc and O.~Bernard.
\newblock Interval observers for linear time-invariant systems with
  disturbances.
\newblock \emph{Automatica}, 47:\penalty0 140--147, 2011.

\bibitem[Briat and Khammash(2016)]{Briat:15g}
C.~Briat and M.~Khammash.
\newblock Interval peak-to-peak observers for continuous- and discrete-time
  systems with persistent inputs and delays.
\newblock \emph{Automatica}, 74:\penalty0 206--213, 2016.

\bibitem[Dvirnyi and Slyn'ko(2004)]{Dvirnyi:04}
A.~I. Dvirnyi and V.~I. Slyn'ko.
\newblock Stability criteria for quasilinear impulsive systems.
\newblock \emph{International Applied Mechanics}, 40(5):\penalty0 592--599,
  2004.

\bibitem[Wang et~al.(2014)Wang, Zhang, and Liu]{Wang:14}
Y.-W. Wang, J.-S. Zhang, and M.~Liu.
\newblock Exponential stability of impulsive positive systems with mixed
  time-varying delays.
\newblock \emph{IET Control Theory and Applications}, 8(15):\penalty0
  1537--1542, 2014.

\bibitem[Zhang et~al.(2014)Zhang, Wang, Xiao, and Guan]{Zhang:14b}
J.-S. Zhang, Y.-W. Wang, J.-W. Xiao, and Z.-H. Guan.
\newblock Stability analysis of impulsive positive systems.
\newblock In \emph{19th IFAC World Congress}, pages 5987--5991, Cape Town,
  South Africa, 2014.

\bibitem[Perelson and Nelson(1999)]{Perelson:99}
A.~S. Perelson and P.~W. Nelson.
\newblock Mathematical analysis of hiv-i: Dynamics in vivo.
\newblock \emph{SIAM Review}, 41(1):\penalty0 3--44, 1999.

\bibitem[{Hernandez-Vargas} et~al.(2013){Hernandez-Vargas}, Colaneri, and
  Middleton]{Hernandez:13}
E.~A. {Hernandez-Vargas}, P.~Colaneri, and R.~H. Middleton.
\newblock Optimal therapy scheduling for a simpliﬁed {HIV} infection model.
\newblock 49:\penalty0 2874--2880, 2013.

\bibitem[{Ait Rami} et~al.(2014){Ait Rami}, Bokharaie, Mason, and
  Wirth]{AitRami:14}
M.~{Ait Rami}, V.~S. Bokharaie, O.~Mason, and F.~R. Wirth.
\newblock Stability criteria for {SIS} epidemiological models under switching
  policies.
\newblock \emph{Discrete and continuous dynamical series B}, 19:\penalty0
  2865--2887, 2014.

\bibitem[Blanchini et~al.(2012)Blanchini, Colaneri, and Valcher]{Blanchini:12b}
F.~Blanchini, P.~Colaneri, and M.~E. Valcher.
\newblock Co-positive lyapunov functions for the stabilization of positive
  switched systems.
\newblock \emph{IEEE Transactions on Automatic Control}, 57:\penalty0
  3038--3050, 2012.

\bibitem[Khargonekar and Sivashankar(1991)]{Khargonekar:91}
P.~P. Khargonekar and N.~Sivashankar.
\newblock ${\mathscr{h}_2}$ optimal control for sampled-data systems.
\newblock \emph{Systems \& Control Letters}, 17:\penalty0 425--426, 1991.

\bibitem[Sivashankar and Khargonekar(1994)]{Sivashankar:94}
N.~Sivashankar and P.~P. Khargonekar.
\newblock Characterization of the {${\mathcal{L}}_2$}-induced norm for linear
  systems with jumps with applications to sampled-data systems.
\newblock \emph{SIAM Journal on Control and Optimization}, 32(4):\penalty0
  1128--1150, 1994.

\bibitem[Naghshtabrizi et~al.(2008)Naghshtabrizi, Hespanha, and
  Teel]{Naghshtabrizi:08}
P.~Naghshtabrizi, J.~P. Hespanha, and A.~R. Teel.
\newblock Exponential stability of impulsive systems with application to
  uncertain sampled-data systems.
\newblock \emph{Systems \& Control Letters}, 57:\penalty0 378--385, 2008.

\bibitem[Briat(2013{\natexlab{b}})]{Briat:13d}
C.~Briat.
\newblock Convex conditions for robust stability analysis and stabilization of
  linear aperiodic impulsive and sampled-data systems under dwell-time
  constraints.
\newblock \emph{Automatica}, 49(11):\penalty0 3449--3457, 2013{\natexlab{b}}.

\bibitem[Morse(1996)]{Morse:96}
A.~S. Morse.
\newblock Supervisory control of families of linear set-point controllers -
  {P}art 1: {E}xact matching.
\newblock \emph{IEEE Transactions on Automatic Control}, 41(10):\penalty0
  1413--1431, 1996.

\bibitem[Hespanha and Morse(1999)]{Hespanha:99}
J.~P. Hespanha and A.~S. Morse.
\newblock Stability of switched systems with average dwell-time.
\newblock In \emph{38th Conference on Decision and Control}, pages 2655--2660,
  Phoenix, Arizona, USA, 1999.

\bibitem[Goebel et~al.(2009)Goebel, Sanfelice, and Teel]{Goebel:09}
R.~Goebel, R.~G. Sanfelice, and A.~R. Teel.
\newblock Hybrid dynamical systems.
\newblock \emph{{IEEE} Control Systems Magazine}, 29(2):\penalty0 28--93, 2009.

\bibitem[Geromel and Colaneri(2006)]{Geromel:06b}
J.~C. Geromel and P.~Colaneri.
\newblock Stability and stabilization of continuous-time switched linear
  systems.
\newblock \emph{{SIAM} Journal on Control and Optimization}, 45(5):\penalty0
  1915--1930, 2006.

\bibitem[Wirth(2005)]{Wirth:05}
F.~Wirth.
\newblock A converse {L}yapunov theorem for linear parameter-varying and linear
  switching systems.
\newblock \emph{{SIAM} Journal on Control and Optimization}, 44(1):\penalty0
  210--239, 2005.

\bibitem[Chesi et~al.(2012)Chesi, Colaneri, Geromel, Middleton, and
  Shorten]{Chesi:12}
G.~Chesi, P.~Colaneri, J.~C. Geromel, R.~Middleton, and R.~Shorten.
\newblock A nonconservative {LMI} condition for stability of switched systems
  with guaranteed dwell-time.
\newblock \emph{IEEE Transactions on Automatic Control}, 57(5):\penalty0
  1297--1302, 2012.

\bibitem[Blanchini and Colaneri(2008)]{Blanchini:10}
F.~Blanchini and P.~Colaneri.
\newblock Vertex/plane characterization of the dwell–time property for
  switching linear systems.
\newblock In \emph{49th IEEE Conference on Decision and Control}, pages
  3258--3263, 2008.

\bibitem[Seuret(2012)]{Seuret:12}
A.~Seuret.
\newblock A novel stability analysis of linear systems under asynchronous
  samplings.
\newblock \emph{Automatica}, 48(1):\penalty0 177--182, 2012.

\bibitem[Briat and Seuret(2012{\natexlab{a}})]{Briat:12h}
C.~Briat and A.~Seuret.
\newblock Convex dwell-time characterizations for uncertain linear impulsive
  systems.
\newblock \emph{{IEEE} Transactions on Automatic Control}, 57(12):\penalty0
  3241--3246, 2012{\natexlab{a}}.

\bibitem[Briat and Seuret(2013)]{Briat:13b}
C.~Briat and A.~Seuret.
\newblock Affine minimal and mode-dependent dwell-time characterization for
  uncertain switched linear systems.
\newblock \emph{{IEEE} Transactions on Automatic Control}, 58\penalty0
  (5):\penalty0 1304--1310, 2013.

\bibitem[Briat and Seuret(2015)]{Briat:15c}
C.~Briat and A.~Seuret.
\newblock On the necessity of looped-functionals arising in the analysis of
  pseudo-periodic, sampled-data and hybrid systems.
\newblock \emph{International Journal of Control}, 88((12):\penalty0
  2563--2569, 2015.

\bibitem[Briat(2016{\natexlab{b}})]{Briat:15f}
C.~Briat.
\newblock Theoretical and numerical comparisons of looped functionals and
  clock-dependent {L}yapunov functions - {T}he case of periodic and
  pseudo-periodic systems with impulses.
\newblock \emph{International Journal of Robust and Nonlinear Control},
  26:\penalty0 2232--2255, 2016{\natexlab{b}}.

\bibitem[Allerhand and Shaked(2011)]{Allerhand:11}
L.~I. Allerhand and U.~Shaked.
\newblock Robust stability and stabilization of linear switched systems with
  dwell time.
\newblock \emph{{IEEE} Transactions on Automatic Control}, 56(2):\penalty0
  381--386, 2011.

\bibitem[Allerhand and Shaked(2013{\natexlab{a}})]{Allerhand:13}
L.~I. Allerhand and U.~Shaked.
\newblock {Robust state-dependent switching of linear systems with dwell-time}.
\newblock \emph{IEEE Transactions on Automatic Control}, 58(4):\penalty0
  994--1001, 2013{\natexlab{a}}.

\bibitem[Allerhand and Shaked(2013{\natexlab{b}})]{Allerhand:13b}
L.~I. Allerhand and U.~Shaked.
\newblock Robust estimation of linear switched systems with dwell time.
\newblock \emph{International Journal of Control}, 86(114):\penalty0
  2067---2074, 2013{\natexlab{b}}.

\bibitem[Briat(2014)]{Briat:14a}
C.~Briat.
\newblock Convex lifted conditions for robust $\ell_2$-stability analysis and
  $\ell_2$-stabilization of linear discrete-time switched systems with minimum
  dwell-time constraint.
\newblock \emph{Automatica}, 50(3):\penalty0 976--983, 2014.

\bibitem[Shaked and Gershon(2014)]{Shaked:14}
U.~Shaked and E.~Gershon.
\newblock Robust ${H}_\infty$ control of stochastic linear switched systems
  with dwell-time.
\newblock \emph{International Journal of Robust and Nonlinear Control},
  24:\penalty0 1664--1676, 2014.

\bibitem[Briat(2015{\natexlab{a}})]{Briat:14f}
C.~Briat.
\newblock Convex conditions for robust stabilization of uncertain switched
  systems with guaranteed minimum and mode-dependent dwell-time.
\newblock \emph{Systems \& Control Letters}, 78:\penalty0 63--72,
  2015{\natexlab{a}}.

\bibitem[Briat(2015{\natexlab{b}})]{Briat:15d}
C.~Briat.
\newblock Stability analysis and control of {LPV} systems with piecewise
  constant parameters.
\newblock \emph{Systems \& Control Letters}, 82:\penalty0 10--17,
  2015{\natexlab{b}}.

\bibitem[Briat(2016{\natexlab{c}})]{Briat:15i}
C.~Briat.
\newblock Stability analysis and stabilization of stochastic linear impulsive,
  switched and sampled-data systems under dwell-time constraints.
\newblock \emph{Automatica}, 74:\penalty0 279--287, 2016{\natexlab{c}}.

\bibitem[Xiang(2016)]{Xiang:16}
W.~Xiang.
\newblock Necessary and sufﬁcient condition for stability of switched
  uncertain linear systems under dwell-time constraint.
\newblock \emph{IEEE Transactions on Automatic Control}, 61(11):\penalty0
  3619--3624, 2016.

\bibitem[Allerhand and Shaked(2016)]{Allerhand:15}
L.~I. Allerhand and U.~Shaked.
\newblock Robust switching-based fault tolerant control.
\newblock \emph{Journal of the Franklin Institute}, 61(11):\penalty0
  3619--3624, 2016.

\bibitem[Xiang et~al.(2016)Xiang, Zhai, and Briat]{Briat:TAC16}
W.~Xiang, G.~Zhai, and C.~Briat.
\newblock Stability analysis for {LTI} control systems with controller failures
  and its application in failure tolerant control.
\newblock \emph{IEEE Transactions on Automatic Control}, 61(3):\penalty0
  811--816, 2016.

\bibitem[Nguyen et~al.(2015)Nguyen, Sename, and Dugard]{Nguyen:15}
M.~Q. Nguyen, O.~Sename, and L.~Dugard.
\newblock A switched lpv observer for actuator fault estimation.
\newblock In \emph{1st IFAC Workshop on Linear Parameter Varying Systems},
  pages 194--199, 2015.

\bibitem[Briat and Seuret(2012{\natexlab{b}})]{Briat:11l}
C.~Briat and A.~Seuret.
\newblock A looped-functional approach for robust stability analysis of linear
  impulsive systems.
\newblock \emph{Systems \& Control Letters}, 61(10):\penalty0 980--988,
  2012{\natexlab{b}}.

\bibitem[Xiang(2015)]{Xiang:15a}
W.~Xiang.
\newblock On equivalence of two stability criteria for continuous-time switched
  systems with dwell time constraint.
\newblock \emph{Automatica}, 54:\penalty0 36--40, 2015.

\bibitem[Boyd and Vandenberghe(2004)]{Boyd:04}
S.~Boyd and L.~Vandenberghe.
\newblock \emph{Convex Optimization}.
\newblock Cambridge University Press, Cambridge, MA, USA, 2004.

\bibitem[Handelman(1988)]{Handelman:88}
D.~Handelman.
\newblock Representing polynomials by positive linear functions on compact
  convex polyhedra.
\newblock \emph{Pacific Journal of Mathematics}, 132(1):\penalty0 35--62, 1988.

\bibitem[Scherer and Hol(2006)]{Scherer:06}
C.~W. Scherer and C.~W.~J. Hol.
\newblock Matrix sum-of-squares relaxations for robust semi-definite programs.
\newblock \emph{Mathematical Programming: Series B}, 107:\penalty0 189--211,
  2006.

\bibitem[Kamyar and Peet(2015)]{Kamyar:15}
R.~Kamyar and M.~M. Peet.
\newblock Polynomial optimization with applications to stability analysis and
  control - alternatives to sum of squares.
\newblock \emph{Discrete and Continuous Dynamical Systems Series B},
  20(8):\penalty0 2383--2417, 2015.

\bibitem[Putinar(1993)]{Putinar:93}
M.~Putinar.
\newblock Positive polynomials on compact semi-algebraic sets.
\newblock \emph{Indiana Univ. Math. J.}, 42\penalty0 (3):\penalty0 969--984,
  1993.

\bibitem[Parrilo(2000)]{Parrilo:00}
P.~Parrilo.
\newblock \emph{Structured Semidefinite Programs and Semialgebraic Geometry
  Methods in Robustness and Optimization}.
\newblock PhD thesis, California Institute of Technology, Pasadena, California,
  2000.

\bibitem[Sturm(2001)]{Sturm:01a}
J.~F. Sturm.
\newblock Using {SEDUMI} $1. 02$, a {M}atlab {T}oolbox for {O}ptimization
  {O}ver {S}ymmetric {C}ones.
\newblock \emph{Optimization Methods and Software}, 11\penalty0 (12):\penalty0
  625--653, 2001.

\bibitem[T\"{u}t\"{u}nc\"{u} et~al.(2003)T\"{u}t\"{u}nc\"{u}, Toh, and
  Todd]{Tutuncu:03}
R.~H. T\"{u}t\"{u}nc\"{u}, K.~C. Toh, and M.~J. Todd.
\newblock {Solving semidefinite-quadratic-linear programs using {SDPT3}}.
\newblock \emph{Mathematical Programming Ser. B}, 95:\penalty0 189--217, 2003.

\bibitem[Papachristodoulou et~al.(2013)Papachristodoulou, Anderson, Valmorbida,
  Prajna, Seiler, and Parrilo]{sostools3}
A.~Papachristodoulou, J.~Anderson, G.~Valmorbida, S.~Prajna, P.~Seiler, and
  P.~A. Parrilo.
\newblock \emph{{SOSTOOLS}: Sum of squares optimization toolbox for {MATLAB}
  v3.00}, 2013.

\bibitem[Lawrence(2010)]{Lawrence:10}
D.~A. Lawrence.
\newblock Duality properties of linear impulsive systems.
\newblock In \emph{49th IEEE Conference on Decision and Control}, pages
  6028--6033, Atlanta, GA, USA, 2010.

\bibitem[Berman and Plemmons(1994)]{Berman:94}
A.~Berman and R.~J. Plemmons.
\newblock \emph{Nonnegative matrices in the mathematical sciences}.
\newblock SIAM, Philadelphia, USA, 1994.

\bibitem[Blanchini et~al.(2007)Blanchini, Miani, and Savorgnan]{Blanchini:07}
F.~Blanchini, S.~Miani, and C.~Savorgnan.
\newblock Stability results for linear parameter varying and switching systems.
\newblock \emph{Automatica}, 43:\penalty0 1817--1823, 2007.

\bibitem[Blanchini and Miani(2008)]{Blanchini:08}
F.~Blanchini and S.~Miani.
\newblock \emph{Set-Theoretic Methods in Control}.
\newblock Birkh{\"{a}}user, Boston, USA, 2008.

\bibitem[Mason and Shorten(2007{\natexlab{a}})]{Mason:07}
O.~Mason and R.~N. Shorten.
\newblock On linear copositive {L}yapunov functions and the stability of
  switched positive linear systems.
\newblock \emph{IEEE Transactions on Automatic Control}, 52(7):\penalty0
  1346--1349, 2007{\natexlab{a}}.

\bibitem[Chesi(2010)]{Chesi:10b}
G.~Chesi.
\newblock {LMI techniques for optimization over polynomials in control: A
  survey}.
\newblock \emph{IEEE Transactions on Automatic Control}, 55(11):\penalty0
  2500--2510, 2010.

\bibitem[Papachristodoulou et~al.(2007)Papachristodoulou, Peet, and
  Niculescu]{Papa:07}
A.~Papachristodoulou, M.~M. Peet, and S.~I. Niculescu.
\newblock Stability analysis of linear systems with time-varying delays: Delay
  uncertainty and quenching.
\newblock In \emph{46th Conference on Decision and Control}, New Orleans, LA,
  USA, 2007, 2007.

\bibitem[Peet et~al.(2009)Peet, Papachristodoulou, and Lall]{Peet:09}
M.~M. Peet, A.~Papachristodoulou, and S.~Lall.
\newblock Positive forms and stability of linear time-delay systems.
\newblock \emph{{SIAM} Journal on Control and Optimization}, 47(6):\penalty0
  3237--3258, 2009.

\bibitem[Peet(2013)]{Peet:13}
M.~M. Peet.
\newblock Full-state feedback of delayed system using {SOS}: A new theory of
  duality.
\newblock In \emph{11th {IFAC} Workshop on Time-Delay Systems}, pages 24--29,
  Grenoble, France, 2013.

\bibitem[Prajna and Papachristodoulou(2003)]{Prajna:03}
S.~Prajna and A.~Papachristodoulou.
\newblock Analysis of switched and hybrid systems - beyond piecewise quadratic
  methods.
\newblock In \emph{American Control Conference, Denver, Colorado, USA}, pages
  2779--2784, 2003.

\bibitem[Seuret and Peet(2013)]{Seuret:13}
A.~Seuret and M.~M. Peet.
\newblock Stability analysis of sampled-data systems using sum of squares.
\newblock \emph{IEEE Transactions on Automatic Control}, 58\penalty0
  (6):\penalty0 1620--1625, 2013.

\bibitem[Kamyar et~al.(2014)Kamyar, Murti, and Peet]{Kamyar:14}
R.~Kamyar, C.~Murti, and M.~M. Peet.
\newblock Constructing piecewise-polynomial {L}yapunov functions for local
  stability of nonlinear systems using handelman’s theorem.
\newblock In \emph{53rd IEEE Conference on Decision and Control}, pages
  5481--5487, Los Angeles, USA, 2014.

\bibitem[P{\'{o}}lya(1974)]{Polya:28}
G.~P{\'{o}}lya.
\newblock {\"{U}}ber positive {D}arstellung von {P}olynomen {V}ierteljschr,
  {N}aturforsch. {G}es. {Z}{\"{u}}rich, {V}ol. 73, pp. 141--145, 1928.
\newblock volume~2 of \emph{{C}ollected {P}apers ({E}d. {R}. {P}. {B}oas)},
  pages 309--313. MIT Press, Cambridge, MA, 1974.

\bibitem[Powers and Reznick(2001)]{Powers:01}
V.~Powers and B.~Reznick.
\newblock A new bound for {P}{\'{o}}lya {T}heorem with applications to
  polynomials positive on polyhedra.
\newblock \emph{Journal of pure and applied algebra}, 164:\penalty0 221--229,
  2001.

\bibitem[Sontag(1998)]{Sontag:98}
E.~Sontag.
\newblock \emph{Mathematical Control Theory: Deterministic Finite Dimensional
  Systems}.
\newblock Springer, New York, USA, 1998.

\bibitem[Gurvits et~al.(2007)Gurvits, Shorten, and Mason]{Gurvits:07}
L.~Gurvits, R.~Shorten, and O.~Mason.
\newblock {On the stability of switched positive linear systems}.
\newblock \emph{IEEE Transactions on Automatic Control}, 52(6):\penalty0
  1099--1103, 2007.

\bibitem[Mason and Shorten(2007{\natexlab{b}})]{Mason:07b}
O.~Mason and R.~Shorten.
\newblock Quadratic and copositive {L}yapunov functions and the stability of
  positive switched linear systems.
\newblock In \emph{American Control Conference}, pages 657--662, New York, USA,
  2007{\natexlab{b}}.

\bibitem[Knorn et~al.(2009)Knorn, Mason, and Shorten]{Knorn:09}
F.~Knorn, O.~Mason, and R.~N. Shorten.
\newblock On linear co-positive {L}yapunov functions for sets of linear
  positive systems.
\newblock \emph{Automatica}, 45(8):\penalty0 1943--1947, 2009.

\bibitem[Zhao et~al.(2012)Zhao, Zhang, Shi, and Liu]{Zhao:12b}
X.~Zhao, L.~Zhang, P.~Shi, and M.~Liu.
\newblock Stability of switched positive linear systems with average dwell time
  switching.
\newblock \emph{Automatica}, 48:\penalty0 1132--1137, 2012.

\bibitem[Tanaka and Langbort(2010)]{Tanaka:10}
T.~Tanaka and C.~Langbort.
\newblock {KYP} {L}emma for internally positive systems and a tractable class
  of distributed {H}-infinity control problems.
\newblock In \emph{American Control Conference}, pages 6238--6243, Baltimore,
  Maryland, USA, 2010.

\bibitem[Tanaka and Langbort(2013)]{Tanaka:13a}
T.~Tanaka and C.~Langbort.
\newblock Symmetric formulation of the {S}-{P}rocedure,
  {K}alman-{Y}akubovich-{P}opov {L}emma and their exact losslessness
  conditions.
\newblock \emph{IEEE Transactions on Automatic Control}, 58(6):\penalty0
  1486--1496, 2013.

\bibitem[Shen and Lam(2014)]{Shen:14}
J.~Shen and J.~Lam.
\newblock ${L_\infty}$-gain analysis for positive systems with distributed
  delays.
\newblock \emph{Automatica}, 50:\penalty0 175--179, 2014.

\end{thebibliography}
\end{document}